\documentclass[12pt]{amsart}

\usepackage{amssymb}

\newtheorem{theorem}{Theorem}[section]
\newtheorem{proposition}[theorem]{Proposition}
\newtheorem{lemma}[theorem]{Lemma}
\newtheorem{corollary}[theorem]{Corollary}

\theoremstyle{definition}
\newtheorem{definition}[theorem]{Definition}

\newtheorem{example}[theorem]{Example}

\newcommand{\ZZ}{ \ensuremath{\mathbb{Z}}}

\newcommand{\Tor}{\ensuremath{\mathrm{Tor}}}

\newcommand{\mideal}{\ensuremath{\mathfrak{m}}}

\newcommand{\lex}{{\mathrm{lex}}}

\newcommand{\dlex}{{\mathrm{dlex}}}
\def\cocoa{{\hbox{\rm C\kern-.13em o\kern-.07em C\kern-.13em o\kern-.15em A}}}

\newcommand{\sat}{\mathrm{sat}\ \!}
\newcommand{\eem}[1]{m_{\leq {#1}}}
\newcommand{\oplex}{\mathrm{opdlex}\ \!}
\newcommand{\upp}[1]{^{({#1})}}
\newcommand{\uppeq}[1]{^{(\geq {#1})}}

\begin{document}

\title
[Upper bounds of the Betti numbers]
{Sharp Upper bounds of the Betti numbers\\
for a given Hilbert polynomial}

\author{Giulio Caviglia}
\address{
Giulio Caviglia,
Department of Mathematics,
Purdue University,
West Lafayette,
IN 47901, USA.
}

\author{Satoshi Murai}
\address{
Satoshi Murai,
Department of Mathematical Science,
Faculty of Science,
Yamaguchi University,
1677-1 Yoshida, Yamaguchi 753-8512, Japan.
}

\thanks{This work was supported by KAKENHI 22740018}

\begin{abstract}
We show that there exists a saturated graded ideal in a standard graded polynomial ring which has the largest total Betti numbers among all saturated graded ideals for a fixed Hilbert polynomial.
\end{abstract}

\maketitle

\section{Introduction}

A classical problem consists in studying the number of minimal generators of ideals in a local or a graded ring in relation to other invariants of the ring and of the ideals themselves. In particular a great amount of work has
been done to establish bounds for the number of generators in terms of certain invariants, for instance: multiplicity, Krull dimension and Hilbert functions (see \cite{M,S}).
An important result was proved in \cite{ERV} where the authors established a sharp upper bound for the number of generators $\nu(I)$ of all perfect ideals  $I$ in a regular local ring $(R,\bold m, K)$ (or in a polynomial ring over a field $K$) in terms of their multiplicity and their height. 

In a subsequent paper \cite{V}, Valla provides under the same hypotheses sharps upper bounds for every Betti number $\beta_i(I)=\dim_K \Tor_i^R(I,K)$, notice that with this notation $\beta_0(I)=\nu(I).$ 
More surprisingly Valla proved that among all perfect ideals with a fixed multiplicity and height in a formal power series ring over a field $K$, there exists one which has the largest possible Betti numbers $\beta_i$'s.

The main result of this paper is an extension of Valla's Theorem. We will consider both the local and the graded case although the result we present for the local case follows directly from the graded case.

We first consider the graded case. We show that for every fixed Hilbert polynomial $p(t)$, there exist a point $Y$ in the  Hilbert scheme $Hilb^{p(t)}_{\mathbb{P}^{n-1}}$ such that  $\beta_i(I_Y) \geq \beta_i(I_X)$ for all $i$ and for all $X \in Hilb^{p(t)}_{\mathbb{P}^{n-1}}.$
Equivalently,
let $S=K[X_1,\dots,X_n]$ be a standard graded polynomial ring over a field $K,$  we prove

\begin{theorem}
\label{intro}
Let $p(t)$ be the Hilbert polynomial of a graded ideal of $S$.
There exists a saturated graded ideal $L \subset S$ with the Hilbert polynomial $p(t)$ such that
$\beta_i(S/L) \geq \beta_i(S/I)$ for all $i$ and for all saturated graded ideals $I \subset S$
with the Hilbert polynomial $p(t)$.
\end{theorem}

Notice that Valla's result corresponds to the special case of the theorem when $p(t)$ is constant.

Unfortunately we do not present an explicit formula of the bounds. We are convinced that such a formula, in the general case, would be hard to read and to interpret. Instead, as a part of the proof, we describe the construction of the lex ideal that achieve the bound.  Using the Eliahou--Kervaire resolution it is possible to write an explicit formula for the total Betti numbers of every lex ideal in terms of its minimal generators.

In particular explicit computations of the bounds can be carried out for a given Hilbert polynomial. Thus it would be possible to describe an explicit formula of the bounds for classes of simple enough Hilbert polynomials. For example in the special case when the Hilbert polynomials are constant, such a formula was given by Valla \cite{V}.

Theorem \ref{intro} induces the following upper bounds of Betti numbers of
ideals in a regular local ring (see Section 3 for the proof).
Let $\bold p_I(t)$ be the Hilbert--Samuel polynomial of an ideal $I$ (see \cite[\S 4.6]{BH})
in a regular local ring $(R,\bold m,K)$ with respect to $\bold m$.

\begin{theorem}
\label{local}
Let $\bold p(t)$ be the Hilbert--Samuel polynomial of an ideal of a regular local ring $(R,\bold m, K)$ of dimension $n$ with respect to $\bold m$.
There exists an ideal $L$ in $A=K[[x_1,\dots,x_n]]$
with $\bold p_L(t)= \bold p(t)$ such that
$\beta_i(A/L) \geq \beta_i(R/I)$ for all $i$ and
for all ideals $I \subset R$ with $\bold p_I(t)=\bold p(t)$.
\end{theorem}

Unfortunately, the proof of Theorem \ref{intro} is very long
and complicated.
Moreover, a construction of ideals which achieve the bound is not easy to understand.
Thus it would be desirable to get a simpler proof of the theorem
and to get a better understanding for the structure of ideals which attain maximal Betti numbers.

The paper is structured in the following way:
In Section 2 and 3,
we reduce a problem of Betti numbers to a problem
of combinatorics of lexicographic sets of monomials with a special structure.
In Section 4, we introduce key techniques to prove the main result.
In particular, we give a new proof of Valla's result in this section.
In Section 5, a construction of ideals which attain maximal Betti numbers
of saturated graded ideals for a fixed Hilbert polynomial will be given.
In Section 6, we give a proof of the main combinatorial result about lexicographic
sets of monomials which essentially proves Theorem \ref{intro}.
In Section 7, some examples of ideals with maximal Betti numbers are given.

\section{Universal Lex Ideals}

In this section, we introduce basic notations which are used in the paper.

Let $S=K[x_1,\dots,x_n]$ be a standard graded polynomial ring over a field $K$.
Let $M$ be a finitely generated graded $S$-module.
The \textit{Hilbert function} $H(M,-): \ZZ \to \ZZ$ of $M$ is the numerical function defined by
$$H(M,k)= \dim_K M_k$$
for all $k \in \ZZ$,
where $M_k$ is the graded component of $M$ of degree $k$.
We denote $P_M(t)$ by the Hilbert polynomial of $M$.
Thus $P_M(t)$ is a polynomial in $t$ satisfying $P_M(k)=H(M,k)$ for $k \gg 0$.
The numbers
$$\beta_{i,j}(M)=\dim_K \Tor_i(M,K)_j$$
are called the \textit{graded Betti numbers} of $M$,
and $\beta_i(M)=\sum_{ j \in \ZZ} \beta_{i,j}(M)$ are called the \textit{(total) Betti numbers} of $M$.

A set of monomials $W \subset S$ is said to be \textit{lex}
if, for all monomials $u \in W$ and $v>_\lex u$ of the same degree,
one has $v \in W$,
where $>_\lex$ is the lexicographic order induced by the ordering $x_1>_\lex  \cdots >_\lex x_n$.
A monomial ideal $I \subset S$ is said to be \textit{lex}
if the set of monomials in $I$ is lex.
By the classical Macaulay's theorem \cite{M},
for any graded ideal $I \subset S$ there exists the unique lex ideal $L \subset S$
with the same Hilbert function as $I$.
Moreover, Bigatti \cite{B}, Hulett \cite{H} and Pardue \cite{P} proved that lex ideals have the largest graded
Betti numbers among all graded ideals having the same Hilbert function.

For any graded ideal $I \subset S$, let
$$\sat I = (I: \mideal^\infty)$$
be the \textit{saturation} of $I \subset S$,
where $\mideal=(x_1,\dots,x_n)$ is the graded maximal ideal of $S$.
A graded ideal $I$ is said to be \textit{saturated} if $I=\sat I$.
It is well-known that $I$ is saturated if and only if $\mathrm{depth}(S/I)>0$ or $I=S$.

Let $L \subset S$ be a lex ideal.
Then $\sat L$ is also a lex ideal.
It is natural to ask which lex ideals are saturated.
The theory of universal lex ideals gives an answer.

A lex ideal $L \subset S$ is said to be \textit{universal} if
$L S[x_{n+1}]$ is also a lex ideal in $S[x_{n+1}]$.
The followings are fundamental results on universal lex ideals.

\begin{lemma}[\cite{MH1}]
Let $L \subset S$ be a lex ideal.
The following conditions are equivalent:
\begin{itemize}
\item[(i)] $L$ is universal;
\item[(ii)] $L$ is generated by at most $n$ monomials;
\item[(iii)] $L=S$ or there exist integers $a_1,a_2,\dots,a_t \geq 0$ with $1 \leq t \leq n$ such that
\begin{eqnarray}
\label{univlex}
L=(x_1^{a_1+1},x_1^{a_1}x_2^{a_2+1},\dots,x_1^{a_1}x_2^{a_2}\cdots x_{t-1}^{a_{t-1}}x_t^{a_t +1}).
\end{eqnarray}
\end{itemize}
\end{lemma}


A relation between universal lex ideals and saturated lex ideals is the following.

\begin{lemma}[\cite{MH1}]
Let $L \subsetneq S$ be a lex ideal.
Then $\mathrm{depth}(S/L)>0$ if and only if $L$ is generated by at most $n-1$ monomials.
\end{lemma}

A lex ideal $I \subset S$ is called a \textit{proper universal lex ideal} if $I$ is generated by at most $n-1$ monomials or $I=S$.

Let $I \subset S$ be a graded ideal.
Then there exists the unique lex ideal $L \subset S$ with the same Hilbert function as $I$.
Then $\sat L$ is a proper universal lex ideal with the same Hilbert polynomial as $I$.
This construction $I \to \sat L$ gives a one-to-one correspondence between Hilbert polynomials of graded ideals and proper universal lex ideals,
say,

\begin{proposition}
\label{1-3}
For any graded ideal $I \subset S$ there exists the unique proper universal lex ideal $L \subset S$ with the same Hilbert polynomial as $I$.
\end{proposition}

\begin{proof}
The existence is obvious.
What we must prove is that, if
$L$ and $L'$ are proper universal lex ideals with the same Hilbert polynomial then $L=L'$.

Since $L$ and $L'$ have the same Hilbert polynomial, their Hilbert function coincide in sufficiently large degrees.
This fact shows $L_d =L'_d$ for $d \gg 0$.
Thus $\sat L = \sat L'$.
Since $L$ and $L'$ are saturated, $L= \sat L = \sat L'=L$.
\end{proof}

\section{Strongly stable ideals, Betti numbers and max sequences}

In this section, we reduce a problem of Betti numbers of graded ideals
to a problem of combinatorics of lex sets of monomials.

Let $S=K[x_1,\dots,x_n]$ and $\bar S=K[x_1,\dots,x_{n-1}]$.
For a monomial ideal $I \subset S$, let $\bar I = I \cap \bar S$.
A monomial ideal $I \subset S$ is said to be \textit{strongly stable}
if $u x_j \in I$ and $i<j$ imply $u x_i \in I$. 
The following fact easily follows from the Bigatti-Hulett-Pardue theorem \cite{B,H,P}.
See e.g., the proof of \cite[Theorem 2.1]{MH1}.

\begin{lemma} \label{2-1}
For any saturated graded ideal $I \subset S$,
there exists a saturated strongly stable ideal $J \subset S$ with the same Hilbert function as $I$ such that
$\beta_{i,j}(I) \leq \beta_{i,j}(J)$ for all $i,j$.
Moreover, we may take $J$ so that $\bar J$ is a lex ideal in $\bar S$.
\end{lemma}

\begin{lemma} \label{2-2}
Let $J \subset S$ be a saturated strongly stable ideal.
Then,
\begin{itemize}
\item[(i)] $\dim_K J_d = \sum_{k=0}^d \dim_K \bar J_k$ for all $d \geq 0$.
\item[(ii)] $\beta_i^S(J)=\beta_i^{\bar S} ( \bar J)$ for all $i$.
\end{itemize}
\end{lemma}

\begin{proof}
If a strongly stable ideal $J \subset S$ is saturated then $x_n$ is regular on $S/J$.
Then $J= \bar J S$,
which proves (ii).
Also, for all $d \geq 0$,
we have a decomposition $J_d=\bigoplus_{k=0}^d J_kx_n^{d-k}$
as $K$-vector spaces.
This equality proves (i).
\end{proof}

\begin{corollary} \label{2-3}
Let $J$ and $J'$ be saturated strongly stable ideals in $S$ such that $\bar J$ and $\bar J'$ are lex.
If $J$ and $J'$ have the same Hilbert polynomial then
$\bar J_d = \bar J'_d$ for $d \gg 0$.
\end{corollary}

\begin{proof}
Lemma \ref{2-2}(i) says that $\dim_K J_{d}-\dim_K J_{d-1}= \dim \bar J_d$,
so $\dim_K \bar J_d = \dim_K \bar J'_d$ for $d \gg 0$.
Then the statement follows since $\bar J$ and $\bar J'$ are lex.
\end{proof}

Next, we describe all saturated strongly stable ideals $J$ such that $\bar J$ is lex.
By Proposition \ref{1-3},
to fix a Hilbert polynomial is equivalent to fix a proper universal lex ideal $U$.
For a proper universal lex ideal $U \subset S$,
let
\begin{eqnarray*}
&&\hspace{-10pt}\mathcal L (U)\\
&&\hspace{-10pt}=
\{ I \subset \bar S: I \mbox{ is a lex ideal with $I \subset \sat \bar U$ and $\dim_K (\sat \bar U)/I = \dim_K (\sat \bar U)/\bar U$}\}.
\end{eqnarray*}
Note that $\dim_K (\sat J)/J$ is finite for any graded ideal $J \subset S$
since $(\sat J)/J$ is isomorphic to the $0$th local cohomology module $H_\mideal^0(S/J)$.
By using Lemma \ref{2-2}, it is easy to see that if $I \in \mathcal L(U)$ then $I S$ has the same Hilbert polynomial as $U$.
Actually, the converse is also true.

\begin{lemma} \label{2-4}
Let $U$ be a proper universal lex ideal.
If $J$ is a saturated strongly stable ideal such that $\bar J$ is lex and $P_J(t)=P_U(t)$, then $\bar J \in \mathcal L (U)$.
\end{lemma}

\begin{proof}
By Corollary \ref{2-3}
we have $\bar U_d=\bar J_d$ for $d \gg 0$,
so $\sat \bar U = \sat \bar J$.
Also, since $U$ and $J$ have the same Hilbert polynomial,
for $d \gg 0$, one has
$$\dim_K U_d = \sum_{k=0}^d \dim_K \bar U_k = \sum_{k=0}^d \dim_K (\sat \bar U_k) -\dim_K (\sat \bar U/\bar U)$$
and
$$\dim_K J_d = \sum_{k=0}^d \dim_K \bar J_k = \sum_{k=0}^d \dim_K (\sat \bar J_k) -\dim_K (\sat \bar J/ \bar J).$$
Since $\sat \bar J = \sat \bar U$, we have $\dim_K (\sat \bar J/ \bar J )= \dim_K (\sat \bar U/\bar U)$
and $\bar J \in \mathcal L(U)$.
\end{proof}

By Lemmas \ref{2-1} and \ref{2-4},
to prove Theorem \ref{intro},
it is enough to find a lex ideal which has the largest Betti numbers among all ideals in $\mathcal L (U)$.
We consider a more general setting.
For any universal lex ideal $U \subset S$ (not necessary proper)
and for any positive integer $c >0$, define
$$\mathcal L(U;c)=\{I \subset U: I\mbox{ is a lex ideal with } \dim_K U/I=c\}.$$
We consider the Betti numbers of ideals in $L(U;c)$.

We first discuss Betti numbers of strongly stable ideals.
We need the following notation.
For any monomial $u \in S$,
let
$\max u$ be the largest integer $\ell$ such that $x_\ell$ divides $u$,
where $\max (1)=1$.
For a set of monomials (or a $K$-vector space spanned by monomials) $M$,
let
$$\eem i(M)= \# \{ u \in M: \max u \leq i\}$$
for $i=1,2,\dots,n$, where $\#X$ is the cardinality of a finite set $X$, and
$$m(M)=\big(\eem 1(M), \eem 2 (M),\dots, \eem n (M)\big).$$
These numbers are often used to study Betti numbers of strongly stable ideals.
The next formula was proved by Bigatti \cite{B} and Hulett \cite{H}, by using the famous Eliahou--Kervaire resolution \cite{EK}.

\begin{lemma}\label{3-1}
Let $I \subset S$ be a strongly stable ideal.
Then, for all $i,j$,
$$\beta_{i,i+j}(I)={n-1 \choose i} \dim_K I_j -\sum_{k=1}^n {k-1 \choose i} \eem k (I_{j-1})  - \sum_{k=1}^{n-1} {k-1 \choose i-1} \eem k (I_j).$$
\end{lemma}

For vectors $\mathbf a =(a_1,\dots,a_n),\ \mathbf b =(b_1,\dots,b_n) \in \mathbb Z^n$,
we define
$$\mathbf a \succeq\mathbf b \ \ \Leftrightarrow \ \ 
a_i \geq b_i \mbox{ for }i=1,2,\dots,n.
$$

\begin{corollary} \label{3-2}
Let $U$ be a universal lex ideal
and $I,J \in \mathcal L(U;c)$.
Let $\mathcal M_I$ (resp.\ $\mathcal M_J$) be the set of all monomials in $U \setminus I$
(resp.\ $U \setminus J$).
If $m(\mathcal M_I) \succeq m(\mathcal M_J)$ then $\beta_i(I) \geq \beta_i(J)$ for all $i$.
\end{corollary}

\begin{proof}
Observe that $\beta_{i,i+j}(I)=\beta_{i,i+j}(J)=0$ for $j \gg 0$.
Thus, for $d \gg 0$, we have $\beta_{i}(I)= \sum_{j=0}^d \beta_{i,i+j}(I)$.
Let $I_{\leq d} = \bigoplus_{k=0}^d I_k$.
Then by Lemma \ref{3-1},
$$\beta_{i}(I)={n-1 \choose i} \dim_K I_{\leq d} -\sum_{k=1}^n {k-1 \choose i} \eem k (I_{\leq d-1})  - \sum_{k=1}^{n-1} {k-1 \choose i-1} \eem k (I_{\leq d})$$
and the same formula holds for $J$.
Since, for $d \gg 0$, 
$$m(J_{\leq d})=m(U_{\leq d})-m(\mathcal M_J) \succeq 
m(U_{\leq d}) -m( \mathcal M_I)=m(I_{\leq d}),$$
we have $\beta_i(I) \geq \beta_i(J)$ for all $i$, as desired.
\end{proof}

Next, we study the structure of $\mathcal M_I$.
Let
$$U=(x_1^{a_1+1},x_1^{a_1}x_2^{a_2+1},\dots,x_1^{a_1}x_2^{a_2}\cdots x_{t-1}^{a_{t-1}}x_t^{a_t +1})$$
be a universal lex ideal,
$\delta_i=x_1^{a_1}\cdots x_{i-1}^{a_{i-1}}x_i^{a_i+1}$
and $b_i=a_1+ \cdots + a_i+1 = \deg \delta_i$.
(If $U=S$ then $t=1$ and $a_1=-1$.)
Let 
$$S^{(i)}=K[x_i,\dots,x_n].$$
Then, as $K$-vector spaces, we have a decomposition
$$U=\delta_1 S^{(1)} \bigoplus \delta_2 S^{(2)} \bigoplus \cdots \bigoplus \delta_t S^{(t)}.$$

\begin{definition}
\label{revlex}
A set of monomials $N \subset S \upp i$ is said to be \textit{rev-lex} if,
for all monomials $u \in N$ and $v <_\lex u$ of the same degree,
one has $v \in N$.
Moreover, $N$ is said to be \textit{super rev-lex} (in $S \upp i$)
if it is rev-lex and $u \in N$ implies $v \in N$ for any monomial $v \in S^{(i)}$ of degree $\leq \deg u -1$.
A \textit{multicomplex} is a set of monomials $N \subset S^{(i)}$ satisfying
that $u \in N$ and $v|u$ imply $v \in N$.
Thus a multicomplex is the complement of the set of monomials in a monomial ideal.
Note that super rev-lex sets are multicomplexes.
\end{definition}

Let $I \in \mathcal L(U;c)$ and $\mathcal M_I$ the set of monomials in $U \setminus I$.
Then we can uniquely write
$$\mathcal M_I = \delta_1M_{\langle 1\rangle} \biguplus \delta_2 M_{\langle 2\rangle} \biguplus \cdots \biguplus \delta_t M_{\langle t\rangle}$$
where $M_{\langle i\rangle} \subset S^{(i)}$
and where $\biguplus$ denotes the disjoint union.
The following fact is obvious.

\begin{lemma} \label{3-4}
With the same notation as above,
\begin{itemize}
\item[(i)] each $M_{\langle i\rangle}$ is a rev-lex multicomplex.
\item[(ii)]
if $\delta_i M_{\langle i\rangle}$ has a monomial of degree $d$ then
$\delta_{i+1} M_{\langle i+1\rangle}$ contains all monomials of degree $d$ in $\delta_{i+1} S \upp {i+1}$ for all $d$.
\end{itemize}
\end{lemma}

Note that Lemma \ref{3-4}(ii) is equivalent to saying that
if $M_{\langle i\rangle}$ contains a monomial of degree $d$ then $M_{\langle i+1\rangle}$ contains all monomials of degree $d-a_{i+1}$ in $S \upp {i+1}$.

We say that a set of monomials
$$M=\delta_1M_{\langle 1\rangle} \biguplus \delta_2 M_{\langle 2\rangle} \biguplus \cdots \biguplus \delta_t M_{\langle t\rangle} \subset U,$$
where $M_{\langle i \rangle} \subset S^{(i)}$, is a \textit{ladder set} if it satisfies the conditions (i) and (ii) of Lemma \ref{3-4}.
The next result is the key result in this paper.

\begin{proposition}
\label{main}
Let $U \subset S$ be a universal lex ideal.
For any integer $c\geq 0$,
there exists a ladder set $N \subset U$ with $\#N=c$
such that for any ladder set $M \subset U$ with $\#M =c$
one has
$$m(N) \succeq m(M).$$
\end{proposition}

We prove Proposition \ref{main}
in Section 6.
Here,
we prove Theorem \ref{intro} by using Proposition \ref{main}.

\begin{proof}[Proof of Theorem \ref{intro}]
Let $U \subset S$ be a proper universal lex ideal with $P_U(t)=p(t)$ and $\bar U = U \cap \bar S$.
Let $c = \dim_K (\sat \bar U/ \bar U)$.
For any lex ideal $I \subset \sat \bar U$,
let $\mathcal M_I$ be the set of monomials in $(\sat \bar U \setminus I)$.

Let $N \subset \sat \bar U$ be a ladder set of monomials with $\#N =c$
given in Proposition \ref{main}.
Consider the ideal $J \subset \bar S$ generated by all monomials in $\sat \bar U \setminus N$.
Then $J \subset \sat \bar U$ and $\mathcal M_J=N$.
In particular, $J \in \mathcal L(U)$.

Let $L=JS$.
By the construction,
$P_L(t)=P_U(t)=p(t)$.
We claim that $L$ satisfies the desired conditions.
Let $I \subset S$ be a saturated graded ideal with $P_I(t)=p(t)$.
By Lemmas \ref{2-1} and \ref{2-4},
we may assume that $I$ is a saturated strongly stable ideal with $\bar I \in \mathcal L(U)=\mathcal L(\sat \bar U ; c)$.
Since $\mathcal M_{\bar I}$ is a ladder set, by the choice of $J$,
$m(\mathcal M_J) \succeq m(\mathcal M_{\bar I})$.
Then, by Corollary \ref{3-2}, $\beta_i(L)=\beta_i(J) \geq \beta_i(\bar I)=\beta_i(I)$ for all $i$, as desired.
\end{proof}

Another interesting corollary of Proposition \ref{main} is

\begin{corollary}
Let $U \subset S$ be a universal lex ideal and $c \geq 0$.
There exists a lex ideal $L \subset U$ with $\dim_K U/L=c$ such that,
for any graded ideal $I \subset U$ with $\dim_K U/I=c$,
one has
$\beta_i(L) \geq \beta_i(I)$ for all $i$.
\end{corollary}

Finally we prove Theorem \ref{local}.

\begin{proof}[Proof of Theorem \ref{local}]
Let $I$ be an ideal in a regular local ring $(R,\bold m, K)$ with the Hilbert--Samuel polynomial $\bold p(t)$.
Then the associated graded ring $\textrm{gr}_{\bold m}(R/I)$
has the same Hilbert--Samuel polynomial as $R/I$
and $\beta_i(R/I) \leq \beta_i(\textrm{gr}_{\bold m}(R/I))$ for all $i$
(see \cite{R} and \cite{HRV}).

Let $S=K[x_1,\dots,x_n]$ and $S'=S[x_{n+1}]$ be standard graded polynomial rings.
By adjoining a variable to $\textrm{gr}_{\bold m}(R/I)$ 
we obtain a graded ring that is isomorphic to $S'/J$
for a saturated graded ideal $J \subset S'$.
Then
$\bold p_{\textrm{gr}_{\bold m}(R/I)}(t)$ is equal to the Hilbert polynomial of $S'/J$
and $\beta_i(\textrm{gr}_{\bold m}(R/I))=\beta_i(S'/J)$ for all $i$.
Let $L' \subset S'$ be the saturated ideal with the same Hilbert polynomial as $J$ given in Theorem \ref{intro}.
Observe that $L'$ has no generators which are divisible by $x_{n+1}$ by the construction given in the proof of Theorem \ref{intro}.

Let $L \subset A=K[[x_1,\dots,x_n]]$ be a monomial ideal having the same generators as $L'$.
We claim that $L$ satisfies the desired conditions.
By the construction,
the Hilbert--Samuel polynomial of $L$ is equal to the Hilbert polynomial of $L'$ and
$\beta_i(L)=\beta_i(L')$ for all $i$.
Since $\beta_i(R/I) \leq \beta_i(S'/J) \leq \beta_i(S'/L')$ and $\bold p_{R/I}(t)=P_{S'/J}(t)=P_{S'/L'}(t)$,
the ideal $L$ satisfies the desired conditions.
\end{proof}

\section{Some tools to study max sequence}

In this section, we introduce some tools to study $m(-)$.
Let $S=K[x_1,\dots,x_n]$ and $\hat S= K[x_2,\dots,x_n]$.
From now on,
we identify vector spaces spanned by monomials (such as polynomial rings and monomial ideals)
with the set of monomials in the spaces.
First, we introduce pictures which help to understand the proofs.
We associate with the set of monomials in $S$ the following picture in Figure 1.
\medskip

\begin{center}
\unitlength 0.1in
\begin{picture}( 18.6000, 15.0500)(  1.0700,-16.0000)
%
\special{pn 13}%
\special{pa 768 400}%
\special{pa 768 1600}%
\special{fp}%
\special{pa 768 1600}%
\special{pa 1968 1600}%
\special{fp}%
\special{pa 1968 1600}%
\special{pa 1968 400}%
\special{fp}%
%
\special{pn 8}%
\special{pa 768 400}%
\special{pa 1968 400}%
\special{dt 0.045}%
\special{pa 1968 700}%
\special{pa 768 700}%
\special{dt 0.045}%
\special{pa 768 1000}%
\special{pa 1968 1000}%
\special{dt 0.045}%
\special{pa 1968 1300}%
\special{pa 768 1300}%
\special{dt 0.045}%
\put(13.6700,-14.5000){\makebox(0,0){$1$}}%
\put(9.3000,-11.5000){\makebox(0,0){$x_1$}}%
\put(13.6700,-8.5000){\makebox(0,0){$x_1^2\  x_1x_2\  \cdots\  x_n^2$}}%
\put(13.6700,-5.5000){\makebox(0,0){$x_1^3\  x_1^2x_2\  \cdots\  x_n^3$}}%
\put(12.0000,-11.5000){\makebox(0,0){$x_2$}}%
\put(18.2000,-11.5000){\makebox(0,0){$x_n$}}%
\put(15.6000,-11.5000){\makebox(0,0){$\cdots$}}%
\put(6.0000,-14.4000){\makebox(0,0){$S_0$}}%
\put(6.0000,-11.5000){\makebox(0,0){$S_1$}}%
\put(6.0000,-8.6000){\makebox(0,0){$S_2$}}%
\put(6.0000,-5.5000){\makebox(0,0){$S_3$}}%
\put(13.7000,-1.8000){\makebox(0,0){Figure 1}}%
\end{picture}%
\end{center}
\medskip

\noindent
Each block in Figure 1 represents a set of monomials in $S$ of a fixed degree ordered by the lex order.
We represent a set of monomials $M \subset S$ by a shaded picture so that the set of monomials in the shade is equal to $M$.
For example,
Figure 2 represents the set
$M=\{1,x_1,x_2,\dots,x_n,x_n^2\}$.
\medskip

\begin{center}
\unitlength 0.1in
\begin{picture}( 18.6000, 15.0500)(  1.0700,-16.0000)
%
\special{pn 13}%
\special{pa 768 400}%
\special{pa 768 1600}%
\special{fp}%
\special{pa 768 1600}%
\special{pa 1968 1600}%
\special{fp}%
\special{pa 1968 1600}%
\special{pa 1968 400}%
\special{fp}%
%
\special{pn 8}%
\special{pa 768 400}%
\special{pa 1968 400}%
\special{dt 0.045}%
\special{pa 1968 700}%
\special{pa 768 700}%
\special{dt 0.045}%
\special{pa 768 1000}%
\special{pa 1968 1000}%
\special{dt 0.045}%
\special{pa 1968 1300}%
\special{pa 768 1300}%
\special{dt 0.045}%
\put(13.6700,-14.5000){\makebox(0,0){$1$}}%
\put(9.3000,-11.5000){\makebox(0,0){$x_1$}}%
\put(13.6700,-8.5000){\makebox(0,0){$x_1^2\ x_1x_2\ \dots\ x_n^2$}}%
\put(13.6700,-5.5000){\makebox(0,0){$x_1^3\ x_1^2x_2\ \dots\ x_n^3$}}%
\put(12.2000,-11.5000){\makebox(0,0){$x_2$}}%
\put(18.2000,-11.5000){\makebox(0,0){$x_n$}}%
\put(15.6000,-11.5000){\makebox(0,0){$\dots$}}%
\put(13.7000,-1.8000){\makebox(0,0){Figure 2}}%
%
\special{pn 4}%
\special{pa 1100 1300}%
\special{pa 810 1590}%
\special{dt 0.027}%
\special{pa 980 1300}%
\special{pa 770 1510}%
\special{dt 0.027}%
\special{pa 860 1300}%
\special{pa 770 1390}%
\special{dt 0.027}%
\special{pa 1220 1300}%
\special{pa 930 1590}%
\special{dt 0.027}%
\special{pa 1340 1300}%
\special{pa 1050 1590}%
\special{dt 0.027}%
\special{pa 1460 1300}%
\special{pa 1170 1590}%
\special{dt 0.027}%
\special{pa 1580 1300}%
\special{pa 1290 1590}%
\special{dt 0.027}%
\special{pa 1700 1300}%
\special{pa 1410 1590}%
\special{dt 0.027}%
\special{pa 1820 1300}%
\special{pa 1530 1590}%
\special{dt 0.027}%
\special{pa 1940 1300}%
\special{pa 1650 1590}%
\special{dt 0.027}%
\special{pa 1960 1400}%
\special{pa 1770 1590}%
\special{dt 0.027}%
\special{pa 1960 1520}%
\special{pa 1890 1590}%
\special{dt 0.027}%
%
\special{pn 4}%
\special{pa 1160 1000}%
\special{pa 860 1300}%
\special{dt 0.027}%
\special{pa 1280 1000}%
\special{pa 980 1300}%
\special{dt 0.027}%
\special{pa 1400 1000}%
\special{pa 1100 1300}%
\special{dt 0.027}%
\special{pa 1520 1000}%
\special{pa 1220 1300}%
\special{dt 0.027}%
\special{pa 1640 1000}%
\special{pa 1340 1300}%
\special{dt 0.027}%
\special{pa 1760 1000}%
\special{pa 1460 1300}%
\special{dt 0.027}%
\special{pa 1880 1000}%
\special{pa 1580 1300}%
\special{dt 0.027}%
\special{pa 1960 1040}%
\special{pa 1700 1300}%
\special{dt 0.027}%
\special{pa 1960 1160}%
\special{pa 1820 1300}%
\special{dt 0.027}%
\special{pa 1040 1000}%
\special{pa 770 1270}%
\special{dt 0.027}%
\special{pa 920 1000}%
\special{pa 770 1150}%
\special{dt 0.027}%
\special{pa 800 1000}%
\special{pa 770 1030}%
\special{dt 0.027}%
\put(4.1000,-10.0000){\makebox(0,0){$M=$}}%
%
\special{pn 8}%
\special{pa 1960 700}%
\special{pa 1700 700}%
\special{fp}%
\special{pa 1700 700}%
\special{pa 1700 1000}%
\special{fp}%
\special{pa 1700 1000}%
\special{pa 770 1000}%
\special{fp}%
%
\special{pn 4}%
\special{pa 1960 920}%
\special{pa 1880 1000}%
\special{dt 0.027}%
\special{pa 1960 800}%
\special{pa 1760 1000}%
\special{dt 0.027}%
\special{pa 1940 700}%
\special{pa 1720 920}%
\special{dt 0.027}%
\special{pa 1820 700}%
\special{pa 1720 800}%
\special{dt 0.027}%
\end{picture}%
\end{center}
\medskip

\begin{definition}
We define the \textit{opposite degree lex order} $>_\oplex$ by
$u >_\oplex v$ if (i) $\deg u < \deg v$ or (ii) $\deg u= \deg v$ and $u>_\lex v.$
\end{definition}

For monomials $u_1 \geq_\oplex u_2$, let
$$[u_1,u_2]=\{v \in S: u_1 \geq_\oplex v \geq_\oplex u_2\}.$$
A set of monomials $M \subset S$ is called an \textit{interval}
if $M=[u_1,u_2]$ for some monomials $u_1,u_2 \in S$.
Moreover, we say that $M$ is a \textit{lower lex set of degree $d$} if $M=[x_1^d,u_2]$,
and that $M$ is an \textit{upper rev-lex set of degree $d$} if $M=[u_1,x_n^d]$.
(See Fig.\ 3.)
\medskip

\begin{center}
\unitlength 0.1in
\begin{picture}( 46.5300, 15.5500)(  8.0000,-16.9500)
%
\special{pn 13}%
\special{pa 800 484}%
\special{pa 800 1522}%
\special{fp}%
\special{pa 800 1522}%
\special{pa 1836 1522}%
\special{fp}%
\special{pa 1836 1522}%
\special{pa 1836 484}%
\special{fp}%
%
\special{pn 8}%
\special{pa 1836 484}%
\special{pa 800 484}%
\special{dt 0.045}%
\special{pa 800 744}%
\special{pa 1836 744}%
\special{dt 0.045}%
\special{pa 1836 1004}%
\special{pa 800 1004}%
\special{dt 0.045}%
\special{pa 800 1262}%
\special{pa 1836 1262}%
\special{dt 0.045}%
%
\special{pn 13}%
\special{pa 4418 480}%
\special{pa 4418 1518}%
\special{fp}%
\special{pa 4418 1518}%
\special{pa 5454 1518}%
\special{fp}%
\special{pa 5454 1518}%
\special{pa 5454 480}%
\special{fp}%
%
\special{pn 8}%
\special{pa 5454 480}%
\special{pa 4418 480}%
\special{dt 0.045}%
\special{pa 4418 740}%
\special{pa 5454 740}%
\special{dt 0.045}%
\special{pa 5454 1000}%
\special{pa 4418 1000}%
\special{dt 0.045}%
\special{pa 4418 1258}%
\special{pa 5454 1258}%
\special{dt 0.045}%
%
\special{pn 13}%
\special{pa 2634 480}%
\special{pa 2634 1518}%
\special{fp}%
\special{pa 2634 1518}%
\special{pa 3670 1518}%
\special{fp}%
\special{pa 3670 1518}%
\special{pa 3670 480}%
\special{fp}%
%
\special{pn 8}%
\special{pa 3670 480}%
\special{pa 2634 480}%
\special{dt 0.045}%
\special{pa 2634 740}%
\special{pa 3670 740}%
\special{dt 0.045}%
\special{pa 3670 1000}%
\special{pa 2634 1000}%
\special{dt 0.045}%
\special{pa 2634 1258}%
\special{pa 3670 1258}%
\special{dt 0.045}%
\put(31.3200,-2.2500){\makebox(0,0){Figure 3}}%
%
\special{pn 8}%
\special{pa 800 484}%
\special{pa 1318 484}%
\special{fp}%
\special{pa 1318 484}%
\special{pa 1318 744}%
\special{fp}%
\special{pa 1318 744}%
\special{pa 1836 744}%
\special{fp}%
\special{pa 1836 1262}%
\special{pa 1060 1262}%
\special{fp}%
\special{pa 1060 1262}%
\special{pa 1060 1004}%
\special{fp}%
\special{pa 1060 1004}%
\special{pa 800 1004}%
\special{fp}%
%
\special{pn 8}%
\special{pa 4418 740}%
\special{pa 4676 740}%
\special{fp}%
\special{pa 4676 740}%
\special{pa 4676 1000}%
\special{fp}%
\special{pa 4676 1000}%
\special{pa 5454 1000}%
\special{fp}%
%
\special{pn 8}%
\special{pa 2634 1258}%
\special{pa 3670 1258}%
\special{fp}%
\special{pa 2634 740}%
\special{pa 3152 740}%
\special{fp}%
\special{pa 3152 740}%
\special{pa 3152 1000}%
\special{fp}%
\special{pa 3152 1000}%
\special{pa 3670 1000}%
\special{fp}%
%
\special{pn 4}%
\special{pa 1526 1004}%
\special{pa 1266 1262}%
\special{dt 0.027}%
\special{pa 1370 1004}%
\special{pa 1112 1262}%
\special{dt 0.027}%
\special{pa 1216 1004}%
\special{pa 1060 1158}%
\special{dt 0.027}%
\special{pa 1682 1004}%
\special{pa 1422 1262}%
\special{dt 0.027}%
\special{pa 1824 1016}%
\special{pa 1578 1262}%
\special{dt 0.027}%
\special{pa 1824 1172}%
\special{pa 1734 1262}%
\special{dt 0.027}%
%
\special{pn 4}%
\special{pa 1474 744}%
\special{pa 1216 1004}%
\special{dt 0.027}%
\special{pa 1318 744}%
\special{pa 1060 1004}%
\special{dt 0.027}%
\special{pa 1164 744}%
\special{pa 904 1004}%
\special{dt 0.027}%
\special{pa 1008 744}%
\special{pa 800 952}%
\special{dt 0.027}%
\special{pa 852 744}%
\special{pa 800 796}%
\special{dt 0.027}%
\special{pa 1630 744}%
\special{pa 1370 1004}%
\special{dt 0.027}%
\special{pa 1786 744}%
\special{pa 1526 1004}%
\special{dt 0.027}%
\special{pa 1824 860}%
\special{pa 1682 1004}%
\special{dt 0.027}%
%
\special{pn 4}%
\special{pa 1266 484}%
\special{pa 1008 744}%
\special{dt 0.027}%
\special{pa 1112 484}%
\special{pa 852 744}%
\special{dt 0.027}%
\special{pa 956 484}%
\special{pa 800 640}%
\special{dt 0.027}%
\special{pa 1318 588}%
\special{pa 1164 744}%
\special{dt 0.027}%
%
\special{pn 4}%
\special{pa 5142 740}%
\special{pa 4884 1000}%
\special{dt 0.027}%
\special{pa 4988 740}%
\special{pa 4728 1000}%
\special{dt 0.027}%
\special{pa 4832 740}%
\special{pa 4676 896}%
\special{dt 0.027}%
\special{pa 5298 740}%
\special{pa 5040 1000}%
\special{dt 0.027}%
\special{pa 5440 752}%
\special{pa 5194 1000}%
\special{dt 0.027}%
\special{pa 5440 908}%
\special{pa 5350 1000}%
\special{dt 0.027}%
%
\special{pn 4}%
\special{pa 5092 480}%
\special{pa 4832 740}%
\special{dt 0.027}%
\special{pa 4936 480}%
\special{pa 4676 740}%
\special{dt 0.027}%
\special{pa 4780 480}%
\special{pa 4522 740}%
\special{dt 0.027}%
\special{pa 4624 480}%
\special{pa 4418 688}%
\special{dt 0.027}%
\special{pa 4470 480}%
\special{pa 4418 532}%
\special{dt 0.027}%
\special{pa 5246 480}%
\special{pa 4988 740}%
\special{dt 0.027}%
\special{pa 5402 480}%
\special{pa 5142 740}%
\special{dt 0.027}%
\special{pa 5440 598}%
\special{pa 5298 740}%
\special{dt 0.027}%
%
\special{pn 4}%
\special{pa 3100 740}%
\special{pa 2840 1000}%
\special{dt 0.027}%
\special{pa 2944 740}%
\special{pa 2686 1000}%
\special{dt 0.027}%
\special{pa 2788 740}%
\special{pa 2634 896}%
\special{dt 0.027}%
\special{pa 3152 844}%
\special{pa 2996 1000}%
\special{dt 0.027}%
%
\special{pn 4}%
\special{pa 2996 1000}%
\special{pa 2736 1258}%
\special{dt 0.027}%
\special{pa 3152 1000}%
\special{pa 2892 1258}%
\special{dt 0.027}%
\special{pa 3306 1000}%
\special{pa 3048 1258}%
\special{dt 0.027}%
\special{pa 3462 1000}%
\special{pa 3204 1258}%
\special{dt 0.027}%
\special{pa 3618 1000}%
\special{pa 3358 1258}%
\special{dt 0.027}%
\special{pa 3656 1116}%
\special{pa 3514 1258}%
\special{dt 0.027}%
\special{pa 2840 1000}%
\special{pa 2634 1206}%
\special{dt 0.027}%
\special{pa 2686 1000}%
\special{pa 2634 1050}%
\special{dt 0.027}%
\put(13.1800,-17.8000){\makebox(0,0){Interval}}%
\put(49.3500,-17.7600){\makebox(0,0){Upper rev-lex set}}%
\put(31.5100,-17.7600){\makebox(0,0){Lower lex set}}%
%
\special{pn 8}%
\special{pa 5454 480}%
\special{pa 4418 480}%
\special{fp}%
\put(11.9000,-11.4000){\makebox(0,0){$u_1$}}%
\put(12.0000,-6.1000){\makebox(0,0){$u_2$}}%
\put(53.1300,-6.0600){\makebox(0,0){$x_n^d$}}%
\put(48.0300,-8.6600){\makebox(0,0){$u_1$}}%
\put(30.3500,-8.6600){\makebox(0,0){$u_2$}}%
\put(27.4500,-11.2600){\makebox(0,0){$x_1^d$}}%
\end{picture}%
\end{center}
\bigskip

A benefit of considering pictures is that we can visualize the following map $\rho: S \to \hat S$.
For any monomial $x_1^k u \in S$ with $u \in \hat S$,
let
$$\rho(x_1^k u)=u.$$
This induces a bijection
\begin{eqnarray*}
\begin{array}{cccc}
\rho: & S_d = \bigoplus_{k=0}^d x_1^k \hat S_{d-k} & \longrightarrow &\hat S_{\leq d}=\bigoplus_{k=0}^d \hat S_k\medskip\\
& x_1^k u & \longrightarrow & u.
\end{array}
\end{eqnarray*}
It is easy to see that if $[u_1,u_2] \subset S_d$ then
$\rho([u_1,u_2])=[\rho(u_1),\rho(u_2)]$ is an interval in $\hat S$.
(See Fig.\ 4.)

\begin{center}
\unitlength 0.1in
\begin{picture}( 36.3500, 16.1000)(  3.6500,-17.1500)
%
\special{pn 13}%
\special{pa 2800 400}%
\special{pa 2800 1600}%
\special{fp}%
\special{pa 2800 1600}%
\special{pa 4000 1600}%
\special{fp}%
\special{pa 4000 1600}%
\special{pa 4000 400}%
\special{fp}%
%
\special{pn 8}%
\special{pa 2800 400}%
\special{pa 4000 400}%
\special{dt 0.045}%
\special{pa 4000 700}%
\special{pa 2800 700}%
\special{dt 0.045}%
\special{pa 2800 1000}%
\special{pa 4000 1000}%
\special{dt 0.045}%
\special{pa 4000 1300}%
\special{pa 2800 1300}%
\special{dt 0.045}%
\put(23.7000,-1.9000){\makebox(0,0){Figure 4}}%
%
\special{pn 8}%
\special{pa 2000 700}%
\special{pa 800 700}%
\special{fp}%
\special{pa 2000 1000}%
\special{pa 800 1000}%
\special{fp}%
\special{pa 800 1000}%
\special{pa 800 700}%
\special{fp}%
\special{pa 2000 700}%
\special{pa 2000 1000}%
\special{fp}%
%
\special{pn 8}%
\special{pa 1200 700}%
\special{pa 1200 1000}%
\special{fp}%
%
\special{pn 8}%
\special{pa 1600 1000}%
\special{pa 1600 700}%
\special{fp}%
%
\special{pn 8}%
\special{pa 3600 700}%
\special{pa 3600 1000}%
\special{fp}%
\special{pa 3600 1000}%
\special{pa 4000 1000}%
\special{fp}%
\special{pa 3200 1300}%
\special{pa 3200 1300}%
\special{fp}%
\special{pa 3200 1000}%
\special{pa 3200 1000}%
\special{fp}%
\special{pa 2800 1000}%
\special{pa 3200 1000}%
\special{fp}%
\special{pa 3200 1000}%
\special{pa 3200 1300}%
\special{fp}%
\special{pa 3200 1300}%
\special{pa 4000 1300}%
\special{fp}%
%
\special{pn 8}%
\special{pa 3600 700}%
\special{pa 2800 700}%
\special{fp}%
\put(13.1000,-8.4000){\makebox(0,0){$u_1$}}%
\put(15.1000,-8.4000){\makebox(0,0){$u_2$}}%
\put(34.0000,-8.4000){\makebox(0,0){$\rho(u_2)$}}%
\put(34.0000,-11.4000){\makebox(0,0){$\rho(u_1)$}}%
\put(14.0000,-18.0000){\makebox(0,0){$[u_1,u_2] \subset S_d$}}%
\put(34.0000,-18.0000){\makebox(0,0){$\rho([u_1,u_2]) \subset \hat S_{\leq d}$}}%
%
\special{pn 4}%
\special{pa 1580 700}%
\special{pa 1280 1000}%
\special{dt 0.027}%
\special{pa 1600 800}%
\special{pa 1400 1000}%
\special{dt 0.027}%
\special{pa 1600 920}%
\special{pa 1520 1000}%
\special{dt 0.027}%
\special{pa 1460 700}%
\special{pa 1200 960}%
\special{dt 0.027}%
\special{pa 1340 700}%
\special{pa 1200 840}%
\special{dt 0.027}%
%
\special{pn 4}%
\special{pa 3260 700}%
\special{pa 2960 1000}%
\special{dt 0.027}%
\special{pa 3140 700}%
\special{pa 2840 1000}%
\special{dt 0.027}%
\special{pa 3020 700}%
\special{pa 2800 920}%
\special{dt 0.027}%
\special{pa 2900 700}%
\special{pa 2800 800}%
\special{dt 0.027}%
\special{pa 3380 700}%
\special{pa 3080 1000}%
\special{dt 0.027}%
\special{pa 3500 700}%
\special{pa 3200 1000}%
\special{dt 0.027}%
\special{pa 3600 720}%
\special{pa 3320 1000}%
\special{dt 0.027}%
\special{pa 3600 840}%
\special{pa 3440 1000}%
\special{dt 0.027}%
\special{pa 3600 960}%
\special{pa 3560 1000}%
\special{dt 0.027}%
%
\special{pn 4}%
\special{pa 3800 1000}%
\special{pa 3500 1300}%
\special{dt 0.027}%
\special{pa 3680 1000}%
\special{pa 3380 1300}%
\special{dt 0.027}%
\special{pa 3560 1000}%
\special{pa 3260 1300}%
\special{dt 0.027}%
\special{pa 3440 1000}%
\special{pa 3200 1240}%
\special{dt 0.027}%
\special{pa 3320 1000}%
\special{pa 3200 1120}%
\special{dt 0.027}%
\special{pa 3920 1000}%
\special{pa 3620 1300}%
\special{dt 0.027}%
\special{pa 3990 1050}%
\special{pa 3740 1300}%
\special{dt 0.027}%
\special{pa 3990 1170}%
\special{pa 3860 1300}%
\special{dt 0.027}%
\end{picture}%
\end{center}
\bigskip

In particular, we have
\begin{lemma}
Let $M \subset S_d$ be a set of monomials.
\begin{itemize}
\item[(i)] If $M$ is lex then $\rho(M)$ is a lower lex set of degree $0$ in $\hat S$.
\item[(ii)] If $M$ is rev-lex then $\rho(M)$ is an upper rev-lex set of degree $d$ in $\hat S$.
\end{itemize}
\end{lemma}

We define $\max(1)=1$ in $S$ and $\max(1)=2$ in $\hat S$.
For any monomial $u \in S_d$ with $u \ne x_1^d$,
one has $\max (u) = \max(\rho (u))$. Hence

\begin{lemma} 
\label{notchange}
Let $M \subset S_d$ be a set of monomials.
One has $m(M)\succeq m(\rho(M))$. Moreover, if $x_1^d \not \in M$ then $m(M)=m(\rho(M))$.
\end{lemma}

\begin{lemma}[Interval Lemma] \label{4-3}
Let $[u_1,u_2]$ be an interval in $S$, $0 \leq a \leq \deg u_1$ and $b \geq \deg u_2$.
Let $L \subset S$ be the lower lex set of degree $a$ and $R$ the upper rev-lex set of degree $b$
with $\#L=\#R=\#[u_1,u_2]$.
Then
$$m(L) \succeq m\big([u_1,u_2]\big) \succeq m(R).$$
\end{lemma}

\begin{proof}
We use double induction on $n$ and $\#[u_1,u_2]$.
The statement is obvious if $n=1$ or if $\#[u_1,u_2]=1$.
Suppose $n>1$ and $\#[u_1,u_2] >1$.

\textit{Case 1}. We first prove the statement when
$[u_1,u_2]$, $L$ and $R$ are contained in a single component $S_d$ for some degree $d$.
We may assume $L \ne [u_1,u_2]$ and $L \ne R$.
Then, since $x_1^d \not \in [u_1,u_2]$,
$m([u_1,u_2])=m(\rho([u_1,u_2]))$ and $m(R)=m(\rho(R))$.
Since $\rho(L) \subset \hat S_{\leq d}$ is a lower lex set of degree $0$,
$\rho([u_1,u_2]) \subset \hat S_{\leq d}$ is an interval
and $\rho(R) \subset \hat S_{\leq d}$ is an upper rev-lex set of degree $d$
in $\hat S$,
by the induction hypothesis, we have
$$m(L) \succeq m \big(\rho(L) \big)\succeq m \big( \rho \big([u_1,u_2]\big) \big)\succeq m \big(\rho(R)\big)=m(R).$$
Then the statement follows since $m(\rho([u_1,u_2]))=m([u_1,u_2])$.

\textit{Case 2}. Now we prove the statement in general.
We first prove the statement for $L$.
We identify $S_i$ with the set of monomials in $S$ of degree $i$.
Suppose $\#[u_1,u_2] > \# S_a$.
Then there exist $u_1',u_2' \in S$ such that
$$[u_1,u_2]=[u_1,u_2'] \biguplus [u_1',u_2]$$
and $\#[u_1,u_2']=\# S_a$.
Let $L'$ be the lower lex set of degree $a+1$ with $\#L' = \# [u_1',u_2]$.
By the induction hypothesis, $m(S_a) \succeq m([u_1,u_2'])$ and $m(L') \succeq m([u_1',u_2])$. Thus
$$m \big([u_1,u_2] \big) \preceq m \big(S_a \biguplus L'\big) = m(L).$$

Suppose $\#[u_1,u_2] \leq \# S_a$.
Then $L \subset S_a$.
Let $d = \deg u_1$ and $A \subset S_d$ the lex set with $\# A =\#[u_1,u_2]$.
Then $A= x_1^{d-a} L$.
Since  $m(A)=m(L)$, what we must prove is
$$m(A) \succeq  m \big([u_1,u_2] \big).$$
Since $\#[u_1,u_2] \leq \# S_a \leq \# S_{d+1}$,
we have $\deg u_2 \leq d+1$.

If $\deg u_2=d$, then $[u_1,u_2] \subset S_d$.
Then the desired inequality follows from Case 1.
Suppose $\deg u_2 = d+1$.
Then
$$[u_1,u_2]=[u_1,x_n^d] \biguplus [x_1^{d+1},u_2].$$
Recall $\#[u_1,u_2] \leq  \# S_a \leq \# S_d$.
Let $B \subset S_d$ be the lex set with $\#B=\#[x_1^{d+1},u_2]$.
Then $[x_1^{d+1},u_2]=x_1B$.
Since $\#B + \# [u_1,x_n^d]=\#[u_1,u_2] \leq \# S_d$,
$B \cap [u_1,x_n^d]= \emptyset$.
Then,
by Case 1,
$$m\big([u_1,u_2]\big) =m(B) + m\big([u_1,x_n^d] \big) \preceq m(A).$$
(See Fig.\ 5.)

\begin{center}
\unitlength 0.1in
\begin{picture}( 53.6500, 14.3000)(  3.2200,-15.4500)
%
\special{pn 13}%
\special{pa 438 400}%
\special{pa 438 1400}%
\special{fp}%
\special{pa 438 1400}%
\special{pa 1188 1400}%
\special{fp}%
\special{pa 1188 1400}%
\special{pa 1188 1400}%
\special{fp}%
\special{pa 1188 400}%
\special{pa 1188 1400}%
\special{fp}%
%
\special{pn 8}%
\special{pa 438 900}%
\special{pa 1188 900}%
\special{dt 0.045}%
\special{pa 1188 650}%
\special{pa 438 650}%
\special{dt 0.045}%
\special{pa 438 400}%
\special{pa 1188 400}%
\special{dt 0.045}%
%
\special{pn 8}%
\special{pa 1188 1150}%
\special{pa 438 1150}%
\special{dt 0.045}%
%
\special{pn 13}%
\special{pa 1938 400}%
\special{pa 1938 1400}%
\special{fp}%
\special{pa 1938 1400}%
\special{pa 2688 1400}%
\special{fp}%
\special{pa 2688 1400}%
\special{pa 2688 1400}%
\special{fp}%
\special{pa 2688 400}%
\special{pa 2688 1400}%
\special{fp}%
%
\special{pn 8}%
\special{pa 1938 900}%
\special{pa 2688 900}%
\special{dt 0.045}%
\special{pa 2688 650}%
\special{pa 1938 650}%
\special{dt 0.045}%
\special{pa 1938 400}%
\special{pa 2688 400}%
\special{dt 0.045}%
%
\special{pn 8}%
\special{pa 2688 1150}%
\special{pa 1938 1150}%
\special{dt 0.045}%
%
\special{pn 13}%
\special{pa 3438 400}%
\special{pa 3438 1400}%
\special{fp}%
\special{pa 3438 1400}%
\special{pa 4188 1400}%
\special{fp}%
\special{pa 4188 1400}%
\special{pa 4188 1400}%
\special{fp}%
\special{pa 4188 400}%
\special{pa 4188 1400}%
\special{fp}%
%
\special{pn 8}%
\special{pa 3438 900}%
\special{pa 4188 900}%
\special{dt 0.045}%
\special{pa 4188 650}%
\special{pa 3438 650}%
\special{dt 0.045}%
\special{pa 3438 400}%
\special{pa 4188 400}%
\special{dt 0.045}%
%
\special{pn 8}%
\special{pa 4188 1150}%
\special{pa 3438 1150}%
\special{dt 0.045}%
%
\special{pn 13}%
\special{pa 4938 400}%
\special{pa 4938 1400}%
\special{fp}%
\special{pa 4938 1400}%
\special{pa 5688 1400}%
\special{fp}%
\special{pa 5688 1400}%
\special{pa 5688 1400}%
\special{fp}%
\special{pa 5688 400}%
\special{pa 5688 1400}%
\special{fp}%
%
\special{pn 8}%
\special{pa 4938 900}%
\special{pa 5688 900}%
\special{dt 0.045}%
\special{pa 5688 650}%
\special{pa 4938 650}%
\special{dt 0.045}%
\special{pa 4938 400}%
\special{pa 5688 400}%
\special{dt 0.045}%
%
\special{pn 8}%
\special{pa 5688 1150}%
\special{pa 4938 1150}%
\special{dt 0.045}%
%
\special{pn 8}%
\special{pa 438 650}%
\special{pa 688 650}%
\special{fp}%
\special{pa 688 650}%
\special{pa 688 900}%
\special{fp}%
\special{pa 688 900}%
\special{pa 438 900}%
\special{fp}%
%
\special{pn 8}%
\special{pa 1188 900}%
\special{pa 938 900}%
\special{fp}%
\special{pa 938 900}%
\special{pa 938 1150}%
\special{fp}%
\special{pa 938 1150}%
\special{pa 1188 1150}%
\special{fp}%
%
\special{pn 8}%
\special{pa 1938 900}%
\special{pa 2188 900}%
\special{fp}%
\special{pa 2188 900}%
\special{pa 2188 1150}%
\special{fp}%
\special{pa 2188 1150}%
\special{pa 1938 1150}%
\special{fp}%
%
\special{pn 8}%
\special{pa 2688 900}%
\special{pa 2438 900}%
\special{fp}%
\special{pa 2438 900}%
\special{pa 2438 1150}%
\special{fp}%
\special{pa 2438 1150}%
\special{pa 2688 1150}%
\special{fp}%
%
\special{pn 8}%
\special{pa 3438 900}%
\special{pa 3938 900}%
\special{fp}%
\special{pa 3938 900}%
\special{pa 3938 1150}%
\special{fp}%
\special{pa 3938 1150}%
\special{pa 3438 1150}%
\special{fp}%
%
\special{pn 8}%
\special{pa 4938 1150}%
\special{pa 5438 1150}%
\special{fp}%
\special{pa 5438 1150}%
\special{pa 5438 1400}%
\special{fp}%
\put(36.8700,-10.2500){\makebox(0,0){$A$}}%
\put(51.8700,-12.7500){\makebox(0,0){$L$}}%
\put(10.3700,-10.2500){\makebox(0,0){$u_1$}}%
\put(5.8700,-7.7500){\makebox(0,0){$u_2$}}%
\put(20.8000,-10.3000){\makebox(0,0){$B$}}%
\put(25.6000,-10.3000){\makebox(0,0){$u_1$}}%
\put(45.4800,-9.0000){\makebox(0,0){$\Rightarrow$}}%
\put(30.4800,-9.0000){\makebox(0,0){$\Rightarrow$}}%
\put(15.4800,-9.0000){\makebox(0,0){$\Rightarrow$}}%
%
\special{pn 4}%
\special{pa 5330 1150}%
\special{pa 5090 1390}%
\special{dt 0.027}%
\special{pa 5210 1150}%
\special{pa 4970 1390}%
\special{dt 0.027}%
\special{pa 5090 1150}%
\special{pa 4940 1300}%
\special{dt 0.027}%
\special{pa 5440 1160}%
\special{pa 5210 1390}%
\special{dt 0.027}%
\special{pa 5440 1280}%
\special{pa 5330 1390}%
\special{dt 0.027}%
%
\special{pn 4}%
\special{pa 3940 980}%
\special{pa 3770 1150}%
\special{dt 0.027}%
\special{pa 3900 900}%
\special{pa 3650 1150}%
\special{dt 0.027}%
\special{pa 3780 900}%
\special{pa 3530 1150}%
\special{dt 0.027}%
\special{pa 3660 900}%
\special{pa 3440 1120}%
\special{dt 0.027}%
\special{pa 3540 900}%
\special{pa 3440 1000}%
\special{dt 0.027}%
\special{pa 3940 1100}%
\special{pa 3890 1150}%
\special{dt 0.027}%
%
\special{pn 4}%
\special{pa 2680 1040}%
\special{pa 2570 1150}%
\special{dt 0.027}%
\special{pa 2680 920}%
\special{pa 2450 1150}%
\special{dt 0.027}%
\special{pa 2580 900}%
\special{pa 2440 1040}%
\special{dt 0.027}%
%
\special{pn 4}%
\special{pa 2190 930}%
\special{pa 1970 1150}%
\special{dt 0.027}%
\special{pa 2100 900}%
\special{pa 1940 1060}%
\special{dt 0.027}%
\special{pa 1980 900}%
\special{pa 1940 940}%
\special{dt 0.027}%
\special{pa 2190 1050}%
\special{pa 2090 1150}%
\special{dt 0.027}%
%
\special{pn 4}%
\special{pa 1140 900}%
\special{pa 940 1100}%
\special{dt 0.027}%
\special{pa 1180 980}%
\special{pa 1010 1150}%
\special{dt 0.027}%
\special{pa 1180 1100}%
\special{pa 1130 1150}%
\special{dt 0.027}%
\special{pa 1020 900}%
\special{pa 940 980}%
\special{dt 0.027}%
%
\special{pn 4}%
\special{pa 690 750}%
\special{pa 540 900}%
\special{dt 0.027}%
\special{pa 690 870}%
\special{pa 660 900}%
\special{dt 0.027}%
\special{pa 670 650}%
\special{pa 440 880}%
\special{dt 0.027}%
\special{pa 550 650}%
\special{pa 440 760}%
\special{dt 0.027}%
\put(30.8000,-2.0000){\makebox(0,0){Figure 5}}%
\put(8.1700,-16.2500){\makebox(0,0){$[u_1,u_2]$}}%
\put(23.2000,-16.3000){\makebox(0,0){$B \biguplus [u_1,x_n^d]$}}%
\put(38.0000,-16.3000){\makebox(0,0){$A$}}%
\put(53.3000,-16.3000){\makebox(0,0){$L$}}%
%
\special{pn 8}%
\special{pa 580 850}%
\special{pa 580 1030}%
\special{fp}%
\special{sh 1}%
\special{pa 580 1030}%
\special{pa 600 964}%
\special{pa 580 978}%
\special{pa 560 964}%
\special{pa 580 1030}%
\special{fp}%
%
\special{pn 8}%
\special{pa 3680 1110}%
\special{pa 3680 1260}%
\special{fp}%
\special{sh 1}%
\special{pa 3680 1260}%
\special{pa 3700 1194}%
\special{pa 3680 1208}%
\special{pa 3660 1194}%
\special{pa 3680 1260}%
\special{fp}%
%
\special{pn 8}%
\special{pa 2470 1040}%
\special{pa 2270 1040}%
\special{fp}%
\special{sh 1}%
\special{pa 2270 1040}%
\special{pa 2338 1060}%
\special{pa 2324 1040}%
\special{pa 2338 1020}%
\special{pa 2270 1040}%
\special{fp}%
\end{picture}%
\end{center}
\bigskip

Next, we prove the statement for $R$.
In the same way as in the proof for $L$, we may assume $\#[u_1,u_2] \leq \# S_b$.
Let $d = \deg u_2$.

If $\deg u_1 =d$,
then $[u_1,u_2] \subset S_d$ and $A= x_1^{b-d} [u_1,u_2]$ is an interval in $S_b$.
Then, by Case 1, we have
$m\big([u_1,u_2]\big) =m(A) \succeq m(R)$ as desired.
Suppose $\deg u_1 <d$.
Then 
$$[u_1,u_2]=[u_1,x_n^{d-1}] \biguplus [x_1^d,u_2].$$
Let $R'$ be the upper rev-lex set of degree $b$ in $S$ with $\#R'=\#[u_1,x_n^{d-1}]$.
Then,
\begin{eqnarray*}
m\big([u_1,u_2]\big) &{\succeq}& m(R') + m\big([x_1^d,u_2] \big) = m(R') + m\big([x_1^b,x_1^{b-d} u_2]\big),
\end{eqnarray*}
where the first inequality follows from the induction hypothesis on the cardinality.
Since $R\setminus R' \subset S_b$ is an interval and $[x_1^b,x_1^{b-d} u_2] \subset S_b$
is lex,
by Case 1 we have
\begin{eqnarray*}
m(R') + m\big([x_1^b,x_1^{b-d} u_2]\big)
\succeq m(R') + m(R \setminus R') = m(R),
\end{eqnarray*}
as desired.
(See Fig.\ 6.)

\begin{center}
\unitlength 0.1in
\begin{picture}( 53.6500, 14.3000)(  3.7200,-15.4500)
%
\special{pn 13}%
\special{pa 438 400}%
\special{pa 438 1400}%
\special{fp}%
\special{pa 438 1400}%
\special{pa 1188 1400}%
\special{fp}%
\special{pa 1188 1400}%
\special{pa 1188 1400}%
\special{fp}%
\special{pa 1188 400}%
\special{pa 1188 1400}%
\special{fp}%
%
\special{pn 8}%
\special{pa 438 900}%
\special{pa 1188 900}%
\special{dt 0.045}%
\special{pa 1188 650}%
\special{pa 438 650}%
\special{dt 0.045}%
\special{pa 438 400}%
\special{pa 1188 400}%
\special{dt 0.045}%
%
\special{pn 8}%
\special{pa 1188 1150}%
\special{pa 438 1150}%
\special{dt 0.045}%
%
\special{pn 13}%
\special{pa 1938 400}%
\special{pa 1938 1400}%
\special{fp}%
\special{pa 1938 1400}%
\special{pa 2688 1400}%
\special{fp}%
\special{pa 2688 1400}%
\special{pa 2688 1400}%
\special{fp}%
\special{pa 2688 400}%
\special{pa 2688 1400}%
\special{fp}%
%
\special{pn 8}%
\special{pa 1938 900}%
\special{pa 2688 900}%
\special{dt 0.045}%
\special{pa 2688 650}%
\special{pa 1938 650}%
\special{dt 0.045}%
\special{pa 1938 400}%
\special{pa 2688 400}%
\special{dt 0.045}%
%
\special{pn 8}%
\special{pa 2688 1150}%
\special{pa 1938 1150}%
\special{dt 0.045}%
%
\special{pn 13}%
\special{pa 3438 400}%
\special{pa 3438 1400}%
\special{fp}%
\special{pa 3438 1400}%
\special{pa 4188 1400}%
\special{fp}%
\special{pa 4188 1400}%
\special{pa 4188 1400}%
\special{fp}%
\special{pa 4188 400}%
\special{pa 4188 1400}%
\special{fp}%
%
\special{pn 8}%
\special{pa 3438 900}%
\special{pa 4188 900}%
\special{dt 0.045}%
\special{pa 4188 650}%
\special{pa 3438 650}%
\special{dt 0.045}%
\special{pa 3438 400}%
\special{pa 4188 400}%
\special{dt 0.045}%
%
\special{pn 8}%
\special{pa 4188 1150}%
\special{pa 3438 1150}%
\special{dt 0.045}%
%
\special{pn 13}%
\special{pa 4938 400}%
\special{pa 4938 1400}%
\special{fp}%
\special{pa 4938 1400}%
\special{pa 5688 1400}%
\special{fp}%
\special{pa 5688 1400}%
\special{pa 5688 1400}%
\special{fp}%
\special{pa 5688 400}%
\special{pa 5688 1400}%
\special{fp}%
%
\special{pn 8}%
\special{pa 4938 900}%
\special{pa 5688 900}%
\special{dt 0.045}%
\special{pa 5688 650}%
\special{pa 4938 650}%
\special{dt 0.045}%
\special{pa 4938 400}%
\special{pa 5688 400}%
\special{dt 0.045}%
%
\special{pn 8}%
\special{pa 5688 1150}%
\special{pa 4938 1150}%
\special{dt 0.045}%
%
\special{pn 8}%
\special{pa 438 650}%
\special{pa 688 650}%
\special{fp}%
\special{pa 688 650}%
\special{pa 688 900}%
\special{fp}%
\special{pa 688 900}%
\special{pa 438 900}%
\special{fp}%
%
\special{pn 8}%
\special{pa 1188 900}%
\special{pa 938 900}%
\special{fp}%
\special{pa 938 900}%
\special{pa 938 1150}%
\special{fp}%
\special{pa 938 1150}%
\special{pa 1188 1150}%
\special{fp}%
%
\special{pn 8}%
\special{pa 1940 650}%
\special{pa 2190 650}%
\special{fp}%
\special{pa 2190 650}%
\special{pa 2190 900}%
\special{fp}%
\special{pa 2190 900}%
\special{pa 1940 900}%
\special{fp}%
%
\special{pn 8}%
\special{pa 2690 400}%
\special{pa 2440 400}%
\special{fp}%
\special{pa 2440 400}%
\special{pa 2440 650}%
\special{fp}%
\special{pa 2440 650}%
\special{pa 2690 650}%
\special{fp}%
\put(10.3700,-10.2500){\makebox(0,0){$u_1$}}%
\put(5.8700,-7.7500){\makebox(0,0){$u_2$}}%
\put(45.4800,-9.0000){\makebox(0,0){$\Rightarrow$}}%
\put(30.4800,-9.0000){\makebox(0,0){$\Rightarrow$}}%
\put(15.4800,-9.0000){\makebox(0,0){$\Rightarrow$}}%
%
\special{pn 4}%
\special{pa 5570 410}%
\special{pa 5330 650}%
\special{dt 0.027}%
\special{pa 5450 410}%
\special{pa 5210 650}%
\special{dt 0.027}%
\special{pa 5330 410}%
\special{pa 5180 560}%
\special{dt 0.027}%
\special{pa 5680 420}%
\special{pa 5450 650}%
\special{dt 0.027}%
\special{pa 5680 540}%
\special{pa 5570 650}%
\special{dt 0.027}%
%
\special{pn 4}%
\special{pa 2680 540}%
\special{pa 2570 650}%
\special{dt 0.027}%
\special{pa 2680 420}%
\special{pa 2450 650}%
\special{dt 0.027}%
\special{pa 2580 400}%
\special{pa 2440 540}%
\special{dt 0.027}%
%
\special{pn 4}%
\special{pa 2190 690}%
\special{pa 1970 910}%
\special{dt 0.027}%
\special{pa 2100 660}%
\special{pa 1940 820}%
\special{dt 0.027}%
\special{pa 1980 660}%
\special{pa 1940 700}%
\special{dt 0.027}%
\special{pa 2190 810}%
\special{pa 2090 910}%
\special{dt 0.027}%
%
\special{pn 4}%
\special{pa 1140 900}%
\special{pa 940 1100}%
\special{dt 0.027}%
\special{pa 1180 980}%
\special{pa 1010 1150}%
\special{dt 0.027}%
\special{pa 1180 1100}%
\special{pa 1130 1150}%
\special{dt 0.027}%
\special{pa 1020 900}%
\special{pa 940 980}%
\special{dt 0.027}%
%
\special{pn 4}%
\special{pa 690 750}%
\special{pa 540 900}%
\special{dt 0.027}%
\special{pa 690 870}%
\special{pa 660 900}%
\special{dt 0.027}%
\special{pa 670 650}%
\special{pa 440 880}%
\special{dt 0.027}%
\special{pa 550 650}%
\special{pa 440 760}%
\special{dt 0.027}%
\put(30.8000,-2.0000){\makebox(0,0){Figure 6}}%
\put(8.1700,-16.2500){\makebox(0,0){$[u_1,u_2]$}}%
\put(23.2000,-16.3000){\makebox(0,0){$R' \biguplus [x_1^d,u_2]$}}%
\put(38.0000,-16.3000){\makebox(0,0){$R' \biguplus [x_1^b,x_1^{b-d} u_2]$}}%
\put(53.3000,-16.3000){\makebox(0,0){$R$}}%
\put(25.6000,-5.3000){\makebox(0,0){$R'$}}%
\put(20.5000,-7.8000){\makebox(0,0){$u_2$}}%
%
\special{pn 8}%
\special{pa 4180 400}%
\special{pa 3930 400}%
\special{fp}%
\special{pa 3930 400}%
\special{pa 3930 650}%
\special{fp}%
\special{pa 3930 650}%
\special{pa 4180 650}%
\special{fp}%
%
\special{pn 4}%
\special{pa 4170 540}%
\special{pa 4060 650}%
\special{dt 0.027}%
\special{pa 4170 420}%
\special{pa 3940 650}%
\special{dt 0.027}%
\special{pa 4070 400}%
\special{pa 3930 540}%
\special{dt 0.027}%
%
\special{pn 8}%
\special{pa 3430 400}%
\special{pa 3680 400}%
\special{fp}%
\special{pa 3680 400}%
\special{pa 3680 650}%
\special{fp}%
\special{pa 3680 650}%
\special{pa 3430 650}%
\special{fp}%
%
\special{pn 4}%
\special{pa 3680 430}%
\special{pa 3460 650}%
\special{dt 0.027}%
\special{pa 3590 400}%
\special{pa 3430 560}%
\special{dt 0.027}%
\special{pa 3470 400}%
\special{pa 3430 440}%
\special{dt 0.027}%
\special{pa 3680 550}%
\special{pa 3580 650}%
\special{dt 0.027}%
\put(40.6000,-5.3000){\makebox(0,0){$R'$}}%
\put(54.6000,-5.3000){\makebox(0,0){$R$}}%
%
\special{pn 8}%
\special{pa 5680 400}%
\special{pa 5180 400}%
\special{fp}%
\special{pa 5180 400}%
\special{pa 5180 650}%
\special{fp}%
\special{pa 5180 650}%
\special{pa 5680 650}%
\special{fp}%
%
\special{pn 8}%
\special{pa 1080 940}%
\special{pa 1080 520}%
\special{fp}%
\special{sh 1}%
\special{pa 1080 520}%
\special{pa 1060 588}%
\special{pa 1080 574}%
\special{pa 1100 588}%
\special{pa 1080 520}%
\special{fp}%
%
\special{pn 8}%
\special{pa 2060 690}%
\special{pa 2060 520}%
\special{fp}%
\special{sh 1}%
\special{pa 2060 520}%
\special{pa 2040 588}%
\special{pa 2060 574}%
\special{pa 2080 588}%
\special{pa 2060 520}%
\special{fp}%
\special{pa 2060 520}%
\special{pa 2060 520}%
\special{fp}%
%
\special{pn 8}%
\special{pa 3610 510}%
\special{pa 3840 510}%
\special{fp}%
\special{sh 1}%
\special{pa 3840 510}%
\special{pa 3774 490}%
\special{pa 3788 510}%
\special{pa 3774 530}%
\special{pa 3840 510}%
\special{fp}%
\end{picture}%
\end{center}
\end{proof}

Recall that a set $M \subset S$ of monomials is said to be super rev-lex
if it is rev-lex and $u \in M$ implies $v \in M$
for any monomial $v \in S$ of degree $\leq \deg u -1$.

\begin{corollary}
\label{X6}
Let $R \subset S$ be an upper rev-lex set of degree $d$
and $M \subset S$ a super rev-lex set
such that $\#R + \# M \leq \# S_{\leq d}$.
Let $Q \subset S$ be the super rev-lex set with $\# Q = \#R + \# M$.
Then
$$m(Q) \succeq m(R) + m(M).$$
\end{corollary}

\begin{proof}
Let $e= \min \{k : x_1^k \not \in M\}$
and $F= \{u \in S_e :u \not \in M\}$.
If $\# F \geq \# R$ then
$$Q= M \biguplus (Q \setminus M)$$
and $Q \setminus M \subset F$ is an interval.
Thus $m(Q \setminus M) \succeq m(R)$ by the interval lemma.

Suppose $\# F < \#R$.
Write 
$$R=I \biguplus R'$$
such that $I$ is an interval with $\# I =\# F$ and $R'$ is an upper rev-lex set of degree $d$.
Since $F$ is a lex set,
the interval lemma shows
$$m(M) + m(R) =m(M)+m(I)+m(R')\preceq m \big(F \biguplus M \big) + m(R').$$
Then $F \biguplus M$ is a super rev-lex set containing $x_1^e$.
By repeating this procedure, we have $m(M) + m(R) \preceq m(Q)$.
\end{proof}

The above corollary proves the next result
which was essentially proved in \cite{ERV}.

\begin{corollary}[Elias-Robbiano-Valla] \label{ERV}
Let $M \subset S$ be a finite rev-lex set of monomials
and $R \subset S$ the super rev-lex set with $\#R = \#M$.
Then
$m(R) \succeq m(M)$.
\end{corollary}

\begin{proof}
Let $M= \biguplus_{i=0} ^N M_i$, where $M_i$ is the set of monomials in $M$ of degree $i$
and $N= \max\{i: M_i \ne \emptyset\}$.
Let $R_{(\leq j)}$ be the super rev-lex set with $\#R_{(\leq j)} =\# \biguplus_{i=0}^j M_i$.
We claim $m(R_{(\leq j)}) \succeq m(\biguplus_{i=0}^j M_i)$ for all $j$.
This follows inductively from Corollary \ref{X6} as follows:
$$m \big(\biguplus_{i=0}^j M_i \big) =m \big(\biguplus_{i=0}^{j-1} M_i \big) + m(M_j) \preceq m(R_{(\leq j-1)}) + m(M_j) \preceq m(R_{(\leq j)}).$$
(We use induction hypothesis for the second step
and use Corollary \ref{X6} for the last step.)
Then we have $m(R)=m(R_{(\leq N)}) \succeq m(\biguplus_{i=0}^N M_i)$.
\end{proof}

We finish this section by proving the result of Valla which we mentioned
in the introduction.

\begin{corollary}[Valla]
Let $c$ be a positive integer and $M \subset S$ the super rev-lex set with $\# M=c$.
Let $J \subset S$ be the monomial ideal generated by all monomials which are not in $M$.
Then, for any homogeneous ideal $I \subset S$ with $\dim_K (S/I)=c$, we have $\beta_i(S/J) \geq \beta_i(S/I)$ for all $i$.
\end{corollary}

\begin{proof}
The proof is similar to that of Corollary \ref{3-2}.
By the Bigatti-Hulett-Pardue theorem, we may assume that $I$ is lex.
Then Lemma \ref{3-1} says, for $d \gg 0$, we have 
$$\beta_{i}(I)={n-1 \choose i} \dim_K I_{\leq d} -\sum_{k=1}^n {k-1 \choose i} \eem k (I_{\leq d-1})  - \sum_{k=1}^{n-1} {k-1 \choose i-1} \eem k (I_{\leq d})$$
and the same formula holds for $J$.
Let $N \subset S$ be the set of monomials which are not in $I$.
Since $N$ is a rev-lex set with $\# N=c$, for $d \gg 0$, by Corollary \ref{ERV} we have
$$m(J_{\leq d})=m(S_{\leq d})-m(M) \preceq 
m(S_{\leq d}) -m(N)=m(I_{\leq d}).$$
Hence $\beta_i(J) \geq \beta_i(I)$ for all $i$ as desired.
\end{proof}

The proof given in this section provides a new short proof of the above result.
The most difficult part in the proof is Corollary \ref{ERV}.
The original proof given in \cite{ERV} is based on computations of binomial coefficients.
On the other hand, our proof is based on moves of interval sets of monomials.

\section{Construction}

In this section, we give a construction of sets of monomials
which satisfies the conditions of Proposition \ref{main},
and study their properties.	

Throughout this section,
we fix a universal lex ideal
$$U=(x_1^{a_1+1},x_1^{a_1}x_2^{a_2+1},\dots,x_1^{a_1}x_2^{a_2}\cdots x_{t-1}^{a_{t-1}}x_t^{a_t +1}).$$
We identify vector spaces spanned by monomials (such as polynomial rings and monomial ideals) with the set of monomials in the spaces.
Thus, $S \upp i$ is the set of monomials in $K[x_i,\dots,x_n]$
and  as we see in Section 3 the universal lex ideal $U$ is identified with
$$U=\delta_1 S \upp 1 \biguplus \delta_2 S \upp 2 \biguplus \cdots \biguplus \delta_t S \upp t,$$
where $\delta_i =x_1^{a_1} \cdots x_{i-1}^{a_{i-1}} x_i^{a_i+1}$ for $i=1,2,\dots,t$.
Let $b_i = \deg \delta_i = a_1+ \cdots + a_i +1$ for $i=1,2,\dots,t$..

Let $M \subset U$. We write
\begin{eqnarray*}
&&U^{(i)} = \delta_i S^{(i)},\
M^{(i)} = M \cap U^{(i)},\
U^{(\geq i)} = \biguplus_{k = i}^t \delta_k S^{(k)} \mbox{ and }
M^{(\geq i)} = M \cap U^{(\geq i)}.
\end{eqnarray*}
Also, we identify $U \uppeq i = \biguplus_{k \geq i} \delta_k S \upp k$
with the universal lex ideal in $K[x_i,\dots,x_n]$
generated by $\{ \delta_k'= \frac {x_i^{a_1+ \cdots + a_{i-1}}} {x_1^{a_1}\cdots x_{i-1}^{a_{i-1}}} \delta_k : k=i,i+1,\dots,t\}$.
For any set of monomials $M$, we write $M_k$ for the set of monomials in $M$ of degree $k$
and $M_{\leq j} = \biguplus_{k=0}^j M_k$.

Like Section 4, we use pictures to help to understand the proofs.
We identify $U$ with the following picture and present $M$ by a shaded picture.

\begin{center}
\unitlength 0.1in
\begin{picture}( 29.5000, 16.0000)(  0.5000,-19.1500)
%
\special{pn 8}%
\special{pa 2996 800}%
\special{pa 596 800}%
\special{dt 0.045}%
\special{pa 596 1000}%
\special{pa 2996 1000}%
\special{dt 0.045}%
\special{pa 2196 1200}%
\special{pa 596 1200}%
\special{dt 0.045}%
\special{pa 596 1400}%
\special{pa 1396 1400}%
\special{dt 0.045}%
\special{pa 1396 1600}%
\special{pa 596 1600}%
\special{dt 0.045}%
\put(9.9500,-17.0000){\makebox(0,0){$1$}}%
\put(9.9500,-15.0000){\makebox(0,0){$x_1\ \dots\ x_n$}}%
\put(9.9500,-13.0000){\makebox(0,0){$x_1^2\ \dots\ x_n^2$}}%
\put(17.9500,-13.0000){\makebox(0,0){$1$}}%
\put(17.9500,-9.0000){\makebox(0,0){$x_2^2\ \dots\ x_n^2$}}%
\put(17.9500,-11.0000){\makebox(0,0){$x_2\ \dots\ x_n$}}%
\put(25.9500,-9.0000){\makebox(0,0){$x_3\ \dots\ x_n$}}%
\put(25.9500,-11.0000){\makebox(0,0){$1$}}%
\put(32.0000,-9.0000){\makebox(0,0){$\cdots$}}%
\put(10.2000,-20.0000){\makebox(0,0){$U\upp 1$}}%
\put(18.3000,-20.0000){\makebox(0,0){$U \upp 2$}}%
\put(26.3000,-20.0000){\makebox(0,0){$U\upp 3$}}%
%
\special{pn 8}%
\special{pa 600 600}%
\special{pa 3000 600}%
\special{dt 0.045}%
\put(18.0000,-4.0000){\makebox(0,0){Figure 7}}%
%
\special{pn 13}%
\special{pa 596 600}%
\special{pa 596 1800}%
\special{fp}%
\special{pa 596 1800}%
\special{pa 1396 1800}%
\special{fp}%
\special{pa 1396 1800}%
\special{pa 1396 600}%
\special{fp}%
%
\special{pn 13}%
\special{pa 1396 1400}%
\special{pa 2196 1400}%
\special{fp}%
\special{pa 2196 1400}%
\special{pa 2196 600}%
\special{fp}%
\special{pa 2196 1200}%
\special{pa 2996 1200}%
\special{fp}%
\special{pa 2996 1200}%
\special{pa 2996 600}%
\special{fp}%
\put(10.0000,-11.0000){\makebox(0,0){$x_1^3\ \dots\ x_n^3$}}%
\put(10.0000,-9.0000){\makebox(0,0){$x_1^4\ \dots\ x_n^4$}}%
\end{picture}%
\end{center}
\bigskip

\noindent For example, Figure 8 represents $M=\delta_1\{1,x_1,x_2,\dots,x_n\} \biguplus \delta_2 \{1\}$.

\begin{center}
\unitlength 0.1in
\begin{picture}( 29.5000, 16.0000)(  0.5000,-19.1500)
%
\special{pn 8}%
\special{pa 2996 800}%
\special{pa 596 800}%
\special{dt 0.045}%
\special{pa 596 1000}%
\special{pa 2996 1000}%
\special{dt 0.045}%
\special{pa 2196 1200}%
\special{pa 596 1200}%
\special{dt 0.045}%
\special{pa 596 1400}%
\special{pa 1396 1400}%
\special{dt 0.045}%
\special{pa 1396 1600}%
\special{pa 596 1600}%
\special{dt 0.045}%
\put(9.9500,-17.0000){\makebox(0,0){$1$}}%
\put(9.9500,-15.0000){\makebox(0,0){$x_1\ \dots\ x_n$}}%
\put(9.9500,-13.0000){\makebox(0,0){$x_1^2\ \dots\ x_n^2$}}%
\put(17.9500,-13.0000){\makebox(0,0){$1$}}%
\put(17.9500,-9.0000){\makebox(0,0){$x_2^2\ \dots\ x_n^2$}}%
\put(17.9500,-11.0000){\makebox(0,0){$x_2\ \dots\ x_n$}}%
\put(25.9500,-9.0000){\makebox(0,0){$x_3\ \dots\ x_n$}}%
\put(25.9500,-11.0000){\makebox(0,0){$1$}}%
\put(32.0000,-9.0000){\makebox(0,0){$\cdots$}}%
%
\special{pn 8}%
\special{pa 600 600}%
\special{pa 3000 600}%
\special{dt 0.045}%
\put(18.0000,-4.0000){\makebox(0,0){Figure 8}}%
%
\special{pn 13}%
\special{pa 596 600}%
\special{pa 596 1800}%
\special{fp}%
\special{pa 596 1800}%
\special{pa 1396 1800}%
\special{fp}%
\special{pa 1396 1800}%
\special{pa 1396 600}%
\special{fp}%
%
\special{pn 13}%
\special{pa 1396 1400}%
\special{pa 2196 1400}%
\special{fp}%
\special{pa 2196 1400}%
\special{pa 2196 600}%
\special{fp}%
\special{pa 2196 1200}%
\special{pa 2996 1200}%
\special{fp}%
\special{pa 2996 1200}%
\special{pa 2996 600}%
\special{fp}%
%
\special{pn 8}%
\special{pa 596 1400}%
\special{pa 1396 1400}%
\special{fp}%
\special{pa 1396 1200}%
\special{pa 2196 1200}%
\special{fp}%
%
\special{pn 4}%
\special{pa 2160 1200}%
\special{pa 1970 1390}%
\special{dt 0.027}%
\special{pa 2040 1200}%
\special{pa 1850 1390}%
\special{dt 0.027}%
\special{pa 1920 1200}%
\special{pa 1730 1390}%
\special{dt 0.027}%
\special{pa 1800 1200}%
\special{pa 1610 1390}%
\special{dt 0.027}%
\special{pa 1680 1200}%
\special{pa 1490 1390}%
\special{dt 0.027}%
\special{pa 1560 1200}%
\special{pa 1400 1360}%
\special{dt 0.027}%
\special{pa 1440 1200}%
\special{pa 1400 1240}%
\special{dt 0.027}%
\special{pa 2190 1290}%
\special{pa 2090 1390}%
\special{dt 0.027}%
%
\special{pn 4}%
\special{pa 1360 1400}%
\special{pa 1160 1600}%
\special{dt 0.027}%
\special{pa 1240 1400}%
\special{pa 1040 1600}%
\special{dt 0.027}%
\special{pa 1120 1400}%
\special{pa 920 1600}%
\special{dt 0.027}%
\special{pa 1000 1400}%
\special{pa 800 1600}%
\special{dt 0.027}%
\special{pa 880 1400}%
\special{pa 680 1600}%
\special{dt 0.027}%
\special{pa 760 1400}%
\special{pa 600 1560}%
\special{dt 0.027}%
\special{pa 640 1400}%
\special{pa 600 1440}%
\special{dt 0.027}%
\special{pa 1390 1490}%
\special{pa 1280 1600}%
\special{dt 0.027}%
%
\special{pn 4}%
\special{pa 1280 1600}%
\special{pa 1090 1790}%
\special{dt 0.027}%
\special{pa 1160 1600}%
\special{pa 970 1790}%
\special{dt 0.027}%
\special{pa 1040 1600}%
\special{pa 850 1790}%
\special{dt 0.027}%
\special{pa 920 1600}%
\special{pa 730 1790}%
\special{dt 0.027}%
\special{pa 800 1600}%
\special{pa 610 1790}%
\special{dt 0.027}%
\special{pa 680 1600}%
\special{pa 600 1680}%
\special{dt 0.027}%
\special{pa 1390 1610}%
\special{pa 1210 1790}%
\special{dt 0.027}%
\special{pa 1390 1730}%
\special{pa 1330 1790}%
\special{dt 0.027}%
\put(18.2000,-20.0000){\makebox(0,0){$M$}}%
\put(10.0000,-11.0000){\makebox(0,0){$x_1^3\ \dots\ x_n^3$}}%
\put(10.0000,-9.0000){\makebox(0,0){$x_1^4\ \dots\ x_n^4$}}%
\end{picture}%
\end{center}
\bigskip

Also, we define the map $\rho: U \to U$ by extending the map given in Section 4 as follows:
For $\delta_i x_i^k u \in U \upp i$ with $u \in K[x_{i+1},\dots,x_n]$,
let
\begin{eqnarray*}
\rho(\delta_i x_i^k u)=
\left\{
\begin{array}{ll}
\delta_{i+1} u, & \mbox{ if $i \leq t-1$,}\\
0, & \mbox{ if $i=t$}.
\end{array}
\right.
\end{eqnarray*}
We call the above map $\rho: U \to U$ the \textit{moving map} of $U$.
The moving map induces a bijection from $U_j \upp i=\{\delta_i u \in U \upp i: \deg u =j-b_i\}$
to $U\upp {i+1}_{\leq j+a_{i+1}} = \{\delta_{i+1} u \in U \upp {i+1}: \deg u \leq j -b_i\}$
for $i = 1,2,\dots,t-1$.
Also, we have

\begin{lemma}
For $N \subset U_j \upp i$ with $i \leq t-1$,
one has $m(N) \succeq m(\rho(N))$.
Moreover, if $\delta_i x_i^{j- b_i} \not \in N$ then $m(N)=m(\rho(N))$.
\end{lemma}

Next,
we define ladder sets $M \subset U$ which attain maximal Betti numbers.
Recall that a subset $M \subset U$ is called a ladder set if the following conditions holds:
\begin{itemize}
\item[(i)] $\{u \in S \upp i: \delta_i u \in M\upp i\}$ is a rev-lex multicomplex for $i=1,2,\dots,t$.
\item[(ii)] if $M_j \upp i \ne \emptyset$ then $M_j \upp {i+1} = U_j \upp {i+1}$ for $i=1,2,\dots,t-1$ and for all $j \geq 0$.
\end{itemize}
To simplify the notation,
we say that $N \subset U \upp i$ is a super rev-lex set (resp.\ interval, lower lex set or upper rev-lex set of degree $d$)
if $N'=\{u \in S \upp i : \delta_i u \in N\}$ is super rev-lex (resp.\ interval, lower lex set or upper rev-lex set of degree $d-b_i$) in $S \upp i$.

\begin{definition}
\label{admdef}
A monomial $f=\delta_1 x_1^{\alpha_1} x_2^{\alpha_2}\cdots x_n^{\alpha_n} \in U_e \upp 1$
is said to be \textit{admissible over $U$} if the following conditions hold
\begin{itemize}
\item[(i)] $\deg \rho^i(f) \leq e+1$ or $\rho^i(f)=\delta_{i+1}$ for $i=1,2,\dots,t-2$,
\item[(ii)] $\rho^{t-1}(f)=\delta_t$ or $\rho^{t-1}(f) \geq _{\mathrm{opdlex}} \delta_t x_t^{e+1-b_t}$.
\end{itemize}
Note that the second condition in (ii) cannot be satisfied when $e+1-b_t <0$.
Also, if $t=1$ then all monomials in $U$ are admissible.
Also $\rho^{t-1}(f) \geq_\oplex \delta_t x_t^{e+1-b_t}$ if and only if
$\deg \rho^{t-1}(f) \leq e$ or $\rho^{t-1}(f)=\delta_t x_t^{e+1-b_t}$.

We say that $f \in U_e \upp i$ is admissible if it is admissible over $U \uppeq i$.
Note that
$\delta_i x_i^k \in U \upp i$ is admissible for all $i$ and $k$.
\end{definition}

\begin{definition}
Fix $c>0$.
Let $>_\dlex$ be the degree lex order.
Thus for monomials $u,v \in S$,
$u>_\dlex v$ if $\deg u > \deg v$ or $\deg u = \deg v$ and $u>_\lex v$.
Let
$$f=\max_{>_\dlex} \{ g \in U \upp 1: \mbox{ $g$ is admissible and } \# \{h \in U : h \leq_\dlex g\} \leq c\}$$
and
$$L_{(c)}= \{ h \in U \upp 1: h \leq_\dlex f\}.$$

Let $M=M \upp 1 \biguplus \cdots \biguplus M \upp t \subset U$ be a set of monomials with $\#M=c$.
We say that $M$ satisfies the \textit{maximal condition} if $M \upp 1 = L_{(c)}$.
Also, we say that $M$ is \textit{extremal} if $M \uppeq k \subset U \uppeq k$ satisfies the maximal condition in $U \uppeq k$ for all $k$.
\end{definition}

\begin{example}
\label{nis1}
If $t=1$ then any monomial in $U=\delta_1 S \upp 1$ is admissible
and extremal sets can be identified with super rev-lex sets in $S \upp 1$.
\end{example}

\begin{example}
\label{nis2}
Suppose $t=2$.
Then $f= \delta_1 x_1 ^{\alpha_1} x_2 ^{\alpha_2}\cdots x_n^{\alpha_n}$, where $f \ne \delta_1 x_1^{\alpha_1}$, 
is admissible in $U=\delta_1 S\upp 1 \biguplus \delta_2 S \upp 2$
if $\alpha_1 \geq a_2$ or $f=\delta_1 x_1^{a_2-1} x_2^{\alpha_2}$.
In other words,
a monomial $f \in \delta_1 S_d \upp 1$ is admissible if and only if
$f \geq_{\lex} \delta_1 x_1^{a_2-1} x_2^{d -a_2+1}$ if $a_2 \leq d$ and $f=\delta_1 x_1^d$ if $a_2 >d$.
For example, if $\delta_1=x_1^2$ and $\delta_2=x_1x_2^3$ then the admissible monomials in $U\upp 1_5=\delta_1 (S_3\upp 1)$ are
$$\delta_1 x_1^3,\delta_1 x_1^2 x_2,\delta_1 x_1^2 x_3,\dots,\delta_1 x_1^2 x_n,
\delta_1 x_1 x_2^2.$$
\end{example}

\begin{example}
\label{nis3}
Suppose $t=3$.
The situation is more complicated.
A monomial $f= \delta_1 x_1 ^{\alpha_1} x_2 ^{\alpha_2} \cdots x_n^{\alpha_n} \in U \upp 1_e$, where $f \ne \delta_1 x_1^{\alpha_1}$,
is admissible in $U$ if and only if the following conditions hold:
\begin{itemize}
\item $\alpha_1 \geq a_2 -1$;
\item $x_3^{\alpha_3} \cdots x_n^{\alpha_n} \geq_\oplex x_3^{e+1-b_3}$ or $x_3^{\alpha_3} \cdots x_n^{\alpha_n}=1$.
\end{itemize}
For example, if $\delta_1=x_1^2$,
$\delta_2=x_1x_2^3$, $\delta_3=x_1 x_2^2 x_3^3$ and $n=3$ then the set of the admissible monomials in
$U\upp 1_6=\delta_1 (K[x_1,x_2,x_3]_4)$ are
$$\{\delta_1 x_1^4\} \cup\{\delta_1 x_1^3 x_2,\delta_1 x_1^3 x_3\} \cup
\{\delta_1 x_1^2x_2^2,\delta_1 x_1^2x_2x_3\} \cup \{\delta_1 x_1x_2^3,\delta_1 x_1x_2^2x_3\}.$$
\end{example}

\begin{example}
Let $U=x_1^2S \upp 1 \biguplus x_1x_2^3 S \upp 2$.
Suppose $c= {n+2 \choose 2} +2$.
Then
$$ \max_{>_\dlex } \big\{ f \in U \upp 1: \mbox{ $f$ is admissible and } \# \{h \in U : h \leq_\dlex f\} \leq c \big\}= \delta_1 x_1^2.$$
Indeed,
$$ \# \{ h \in U: h \leq_{\dlex } \delta_1 x_1^2\} =\# \delta_1 S \upp 1_{\leq 2} \biguplus \{\delta_2\}= {n+2 \choose 2} +1$$
and
\begin{eqnarray*}
\# \{ h \in U: h \leq_{\dlex } \delta_1 x_1x_2^2\} &=&
\# \big( \delta_1 S \upp 1_{\leq 3}\setminus \delta_1\{x_1^3,x_1^2x_2,\dots,x_1^2x_n\} \big)
\biguplus \delta_2S\upp 2_{\leq 2}\\
&=& {n+3 \choose 3}>c.
\end{eqnarray*}
By Example \ref{nis2}, the lex-smallest admissible monomial in $U_5 \upp 1$ is $\delta_1 x_1x_2^2$.
Thus the extremal set $L \subset U$ with $\#L=c$ is
$$L= \delta_1 S \upp 1_{\leq 2} \biguplus \delta_2 \{1,x_n\}.$$
\end{example}

\begin{example}
In general, it is not easy to understand the shape of extremal sets,
but in some special cases they are simple.

If $b_1= b_2 = \cdots =b_t$ then any monomial in $U$ is admissible.
Thus any extremal set $M$ in $U$ is of the form 
$$M=\{h \in U: h \leq_\dlex f\}$$
for some $f \in U$.

If $b_2 > e$ then the only admissible monomial in $U\upp 1 _e$ is $\delta_1 x_1^{e-b_1}$.
Thus if $b_1 \ll  b_2 \ll \cdots \ll b_n$ (for example, if $b_{i+1} -b_i > c$ for all $i$),
then any extremal set $M$ in $U$ with $\#M=c$ is of the form
$$M= \delta_1S_{\leq e_1} \upp 1 \biguplus \delta_2S_{\leq e_2} \upp 2 \biguplus \cdots \biguplus \delta_{t-1} S_{\leq e_{t-1}} \upp {t-1} \biguplus N,$$
where $N \subset \delta_t S \upp t$ and
$\# S_{\leq e_{i+1}} \upp {i+1} \biguplus \cdots \biguplus S_{\leq e_{t-1}} \upp {t-1} \biguplus N < \# S_{e_i+1} \upp i$ for $i=1,2,\dots,t-1$.
\end{example}

In the rest of this section, we study properties of extremal sets.
Suppose $t \geq 3$.
For an integer $k \geq -a_3$,
we write $U \upp i [-k]=(x_3^k \delta_i) S \upp i$.
In the picture, $U \upp i [-k]$ is the picture obtained from that of $U \upp i$
by moving the blocks $k$ steps above.
In particular,
for any integer $k\geq -a_3$,
$U' = U \upp 2 \biguplus(\biguplus_{i=3}^t U \upp i [-k])$ is a universal lex ideal.
(See Fig.\ 9.)

\begin{center}
\unitlength 0.1in
\begin{picture}( 54.0500, 16.0500)(  8.0000,-19.2000)
%
\special{pn 8}%
\special{pa 3200 800}%
\special{pa 800 800}%
\special{dt 0.045}%
\special{pa 800 1000}%
\special{pa 3200 1000}%
\special{dt 0.045}%
\special{pa 2400 1200}%
\special{pa 800 1200}%
\special{dt 0.045}%
\special{pa 800 1400}%
\special{pa 1600 1400}%
\special{dt 0.045}%
\special{pa 1600 1600}%
\special{pa 800 1600}%
\special{dt 0.045}%
\put(20.3000,-20.0000){\makebox(0,0){$U \uppeq 2$}}%
%
\special{pn 8}%
\special{pa 806 600}%
\special{pa 3206 600}%
\special{dt 0.045}%
%
\special{pn 13}%
\special{pa 800 600}%
\special{pa 800 1800}%
\special{fp}%
\special{pa 800 1800}%
\special{pa 1600 1800}%
\special{fp}%
\special{pa 1600 1800}%
\special{pa 1600 600}%
\special{fp}%
%
\special{pn 13}%
\special{pa 1600 1400}%
\special{pa 2400 1400}%
\special{fp}%
\special{pa 2400 1400}%
\special{pa 2400 600}%
\special{fp}%
\special{pa 2400 1200}%
\special{pa 3200 1200}%
\special{fp}%
\special{pa 3200 1200}%
\special{pa 3200 600}%
\special{fp}%
\put(50.3500,-20.0500){\makebox(0,0){$U'=U \upp 2 \biguplus( \biguplus_{i=3}^t U\upp i [-k])$}}%
%
\special{pn 8}%
\special{pa 3806 606}%
\special{pa 6206 606}%
\special{dt 0.045}%
%
\special{pn 13}%
\special{pa 3800 606}%
\special{pa 3800 1806}%
\special{fp}%
\special{pa 3800 1806}%
\special{pa 4600 1806}%
\special{fp}%
\special{pa 4600 1806}%
\special{pa 4600 606}%
\special{fp}%
\put(35.0500,-11.2000){\makebox(0,0){$\Rightarrow$}}%
%
\special{pn 8}%
\special{pa 4600 1400}%
\special{pa 5400 1400}%
\special{dt 0.045}%
\special{pa 5400 1400}%
\special{pa 5400 600}%
\special{dt 0.045}%
%
\special{pn 8}%
\special{pa 5400 1200}%
\special{pa 6200 1200}%
\special{dt 0.045}%
\special{pa 6200 1200}%
\special{pa 6200 600}%
\special{dt 0.045}%
%
\special{pn 13}%
\special{pa 4600 1000}%
\special{pa 5400 1000}%
\special{fp}%
\special{pa 5400 1000}%
\special{pa 5400 600}%
\special{fp}%
%
\special{pn 13}%
\special{pa 5400 800}%
\special{pa 6200 800}%
\special{fp}%
\special{pa 6200 800}%
\special{pa 6200 600}%
\special{fp}%
%
\special{pn 8}%
\special{pa 5400 800}%
\special{pa 3800 800}%
\special{dt 0.045}%
\special{pa 3800 1000}%
\special{pa 4600 1000}%
\special{dt 0.045}%
\special{pa 4600 1200}%
\special{pa 3800 1200}%
\special{dt 0.045}%
\special{pa 3800 1400}%
\special{pa 4600 1400}%
\special{dt 0.045}%
\special{pa 4600 1600}%
\special{pa 3800 1600}%
\special{dt 0.045}%
%
\special{pn 8}%
\special{pa 5000 1350}%
\special{pa 5000 1050}%
\special{fp}%
\special{sh 1}%
\special{pa 5000 1050}%
\special{pa 4980 1118}%
\special{pa 5000 1104}%
\special{pa 5020 1118}%
\special{pa 5000 1050}%
\special{fp}%
%
\special{pn 8}%
\special{pa 5800 1150}%
\special{pa 5800 850}%
\special{fp}%
\special{sh 1}%
\special{pa 5800 850}%
\special{pa 5780 918}%
\special{pa 5800 904}%
\special{pa 5820 918}%
\special{pa 5800 850}%
\special{fp}%
\put(35.0000,-4.0000){\makebox(0,0){Figure 9}}%
\end{picture}%
\end{center}
\bigskip

\begin{lemma}
\label{X1}
Suppose $t \geq 3$.
Let $f \in U_e \upp 1$, $d=\deg \rho(f)$ and $k \geq -a_3$ with $e-d+k \geq 0$.
Then $f$ is admissible over $U$ if and only if the following conditions hold:
\begin{itemize}
\item $\deg \rho(f) \leq e+1$ or $\rho(f)=\delta_2$;
\item $\rho(f) x_2^{e-d+k}$ is admissible in $U'=U\upp 2 \biguplus(\biguplus_{i=3}^t U \upp i [-k])$.
\end{itemize}
\end{lemma}

\begin{proof}
Let $\phi$ be the moving map of $U'$,
$\delta'_i=x_3^k \delta_i$ and $\rho^i(f)=\delta_{i+1} u_{i+1}$
for $i=2,\dots,t-1$.
Then ${\phi}^{i} (\rho(f) x_2^{e-d+k})=\delta_{i+2}' u_{i+2}$ for $i=1,2,\dots,t-2$.
Thus
$\deg\rho^i(f) \leq e+1$ if and only if 
$\deg {\phi}^{i-1}(\rho(f)x_2^{e-d+k}) \leq e+1+k$ for $i\geq 2$.
Also,
$\rho^{t-1}(f) \geq_\oplex \delta_t x_t^{e+1-b_t}$
if and only if
${\phi}^{t-2}(\rho(f)x_2^{e+d+k}) \geq_\oplex \delta_t' x_t^{e+1-b_t}$.
Since $\deg \rho(f) x_2^{e-d+k} = e+k$,
the above facts prove the statement.
\end{proof}

By the definition of the maximal condition,
the following facts are straightforward.

\begin{lemma}
\label{X4}
Let $M \subset U$ be an extremal set.
\begin{itemize}
\item[(i)] If $\# M \geq \# U_{\leq e}$ then $M \supset U_{\leq e}$.
\item[(ii)] If $\# M \geq \# U_{\leq e-1} \upp 1 \biguplus U_{\leq e}^{(\geq 2)}$ then $M \supset U^{(1)}_{\leq e-1} \biguplus U_{\leq e}^{(\geq 2)}$.
\end{itemize}
\end{lemma}

\begin{proof}
Since $M$ is extremal,
there exists an $f \in U \upp 1$ such that 
$$M \upp 1=\{h \in U\upp 1: h \leq_{\dlex} f\}.$$

(i) Since $\delta_1 x_1^{e-b_1}$ is admissible and $\{ h \in U: h \leq_\dlex \delta_1 x_1^{e-b_1}\}=U_{\leq e}$,
$f \geq_\dlex \delta_1x_1^{e-b_1}$.
Then $M \upp 1 \supset \{ h \in U\upp 1: h \leq_\dlex \delta_1 x_1^{e-b_1}\}=U_{\leq e} \upp 1$.
Also, $\# M \uppeq 2 \supset \# \{ h \in U \uppeq 2: h \leq_\dlex f\} \supset U\upp 2 _{\leq e}$
by the definition of the maximal condition.
Then the statement follows by induction on $t$.

(ii).
It is clear that $M \supset U_{\leq e-1}$ by (i).
If $\deg f \geq e$ then 
$$\# M \geq \# \{h \in U: h \leq _{\dlex} f\}
= \# M \upp 1 \biguplus U_{\leq e} \uppeq 2.$$
Then $\# M \uppeq 2 \geq \# U_{\leq e} \uppeq 2$
and $M \uppeq 2 \supset U_{\leq e} \uppeq 2$ by (i) as desired.
If $\deg f <e$ then $M \upp 1 = U_{\leq e-1} \upp 1$ and $\# M \uppeq 2 \geq \# U_{\leq e} \uppeq 2$ by the assumption.
Hence $M \uppeq 2 \supset U_{\leq e} \uppeq 2$ by (i).
\end{proof}

\begin{corollary}
Extremal sets are ladder sets.
\end{corollary}

\begin{proof}
If $M \subset U$ is extremal then $M \upp i$ is super rev-lex for all $i$
by the maximal condition.
It is enough to prove that if $M \upp 1 _e \ne \emptyset$ then $M \supset U_{\leq e} \uppeq 2$.
If $M \upp 1_e \ne \emptyset$ then there exists an admissible monomial $f \in U\upp 1_e$
such that 
$$\#M \geq \#\{h \in U: h \leq_\dlex f\} \geq \# U \upp 1_{\leq e-1} \biguplus U\uppeq 2 _{\leq e}.$$
Then the statement follows from Lemma \ref{X4}.
\end{proof}

\begin{lemma}
\label{X3}
Suppose $t \geq 2$.
Let $M \subset U$ be an extremal set.
\begin{itemize}
\item[(i)] If $a_2 >0$ then $M_e \upp 1 \ne 0$ if and only if $\# M \geq \# U_{\leq e} \upp 1$.
\item[(ii)] If $a_2=0$ and $M_e \upp 1 \ne 0$ then $\#M > \#U_{\leq e} \upp 1$.
\end{itemize}
\end{lemma}

\begin{proof}
Let $f \in U_e \upp 1$ be the lex-smallest admissible monomial in $U \upp 1 _e$ over $U$.

(i)
It suffices to prove that
\begin{eqnarray}
\label{X3P1} \# \{h \in U : h \leq_\dlex f\} = \# U_{\leq e} \upp 1.
\end{eqnarray}

If $f = \delta_1 x_1^{e-b_1}$ then $f'=\delta_1 x_1^{e-b_1-1} x_2$ is not admissible.
By the definition of the admissibility,
one has $\deg \rho(f')= \deg \delta_2 x_2 > e+1$
and $b_2 > e$.
In this case we have $\{h \in U : h \leq_\dlex f\} = U_{\leq e} \upp 1.$

Suppose $f \ne \delta_1 x_1^{e-b_1}$.
We prove (\ref{X3P1}) by using induction on $t$.
Suppose $t=2$.
Then $f= \delta_1 x_1^{a_2-1} x_2^{e+1-b_2}$,
and
$$\{ h \in U : h \leq _\dlex f\}
=U_{\leq e-1} \upp 1 \biguplus [f, \delta_1 x_n^{e-b_1}] \biguplus U_{\leq e} \upp 2.$$
Since $\rho([f,\delta_1 x_n^{e-b_1}])=\biguplus_{j=e+1}^{e+a_2} U_j \upp 2$,
we have
$$\#\{h \in U :h \leq_\dlex f\}
= \# U_{\leq e-1} \upp 1 + \# U_{\leq e+a_2} \upp 2 = \# U_{\leq e} \upp 1$$
where we use $\rho(U_e \upp 1)= U_{\leq e+a_2} \upp 2$
for the last equality.

Suppose $ t\geq 3$.
Since $\rho(f) \ne \delta_2$, 
we have $\deg \rho(f) = e+1$.
Indeed, by Lemma \ref{X1}, $\deg \rho(f) \leq e+1$.
On the other hand, since $\delta_1 x_2^{a_2-1} x_2^{e+1-b_2}$ is admissible over $U$,
$f \leq_\lex \delta_1x_1^{a_2-1}x_2^{e+1-b_2}$.
Thus
$\deg \rho(f) \geq \deg \rho(\delta_1 x_1^{a_2-1} x_2^{e+1-b_2})=e+1$.

Consider $U' = U \upp 2 \biguplus_{i=3}^t U \upp i[-1]$.
By Lemma \ref{X1} (consider the case when $d=e+1$ and $k=1$),
$\rho(f)$ is the lex-smallest admissible monomial in $U\upp 2 _{e+1}$ over $U'$.
Then
\begin{eqnarray}
\label{6.1}
\# \big[\rho(f),\delta_2 x_n^{e+1 -b_2} \big] \biguplus U_{\leq e} \uppeq 2
&=& \# [\rho(f),\delta_2 x_n^{e+1 -b_2}] \biguplus U_{\leq e} \upp 2 \biguplus {U'}_{\leq e+1} \uppeq 3\\
\nonumber &=& \# \big\{h \in U':h \leq_\dlex \rho(f) \big\}\\
\nonumber &=& \# U_{\leq e+1} \upp 2
\end{eqnarray}
where the last equation follows from the induction hypothesis.
On the other hand
\begin{eqnarray}
\label{6.2}
\{h \in U : h \leq_\dlex f\} = [f, \delta_1 x_n^{e-b_1}] \biguplus U_{\leq e-1} \upp 1 \biguplus U_{\leq e} \uppeq 2
\end{eqnarray}
and
\begin{eqnarray}
\label{6.3}
\rho \big( [f, \delta_1 x_n^{e-b_1}]\big)=
\big[ \rho(f),\delta_2 x_n^{e+1-b_2} \big] \biguplus \left( \biguplus_{j=e+2}^{e+a_2} U_j \upp 2\right).
\end{eqnarray}
(\ref{6.1}), (\ref{6.2}) and (\ref{6.3}) show
$$\# \{h \in U : h \leq_\dlex f\}
=\# U_{\leq e-1} \upp 1 \biguplus U_{\leq e+a_2} \upp 2
=\# U_{\leq e-1} \upp 1 \biguplus U_e \upp 1 = \# U_{\leq e} \upp 1$$
where the second equality follows since 
$\rho(U_e \upp 1) = U_{\leq e+a_2} \upp 2$.

(ii) It suffices to prove that
$$\{h \in U: h \leq_\dlex f\} > \# U_{\leq e}\upp 1.$$
Since $a_2=0$, $\# U_{\leq e} \upp 2= \# U_e \upp 1$.
Then we have
$$\# \{h \in U: h \leq_\dlex f\} >
\# U_{\leq e-1} \upp 1 \biguplus U_{\leq e} \upp 2
=\# U_{\leq e-1} \upp 1 \biguplus U_e \upp 1
= U_{\leq e} \upp 1,$$
as desired.
\end{proof}

\begin{corollary}
\label{X10}
Suppose $t \geq 2$.
Let $B \subset U \upp 1_e$ be the rev-lex set and
$N \subset U \uppeq 2$ a ladder set with $\#N \geq \# U_{\leq e} \uppeq 2$.
Let $Y \subset U$ be the extremal set with
$\#Y=\# U_{\leq e-1} \upp 1 \biguplus B \biguplus N.$
If $\#B \biguplus N < \# U \upp 1 _e$ then
$$Y=U_{\leq e-1} \upp 1 \biguplus Y\uppeq 2.$$
\end{corollary}

\begin{proof}
Since $\# Y \geq \# U_{\leq e-1}$, we have $Y \supset U_{\leq e-1}$ by Lemma \ref{X4}.
On the other hand, since $\# Y = \#U_{\leq e-1} \upp 1 \biguplus B \biguplus N
< \# U_{\leq e} \upp 1$ by the assumption, we have $Y \upp 1_e= \emptyset$ by Lemma \ref{X3}.
Hence $Y \upp 1 = U_{\leq e-1} \upp 1$.
\end{proof}

For monomials $f>_\lex g \in U_j\upp i$,
let $[f,g)= [f,g] \setminus \{g\}$.

\begin{lemma}
\label{X9}
Let $f \in U_e \upp 1$ be the lex-smallest admissible monomial in $U_e \upp 1$ over $U$
and $g>_\lex h \in U\upp 1_e$
admissible monomials over $U$ such that there are no admissible monomials in $[g,h]$
except for $g$ and $h$.
Then $\#[g,h) \leq \# [f, \delta_1 x_n^{e-b_1}]$.
\end{lemma}

\begin{proof}
If $t=1$ then all monomials are admissible over $U$.
If $t=2$ then any monomial $w \in U_e \upp 1$ with $w >_\lex f$ is admissible over $U$.
Thus the statement is clear if $t \leq 2$.

Suppose $t\geq 3$.
Since $g \ne h$ we have $f \ne \delta_1 x_1^{e-b_1}$.
By the definition of the admissibility,
we have $\deg (\rho(f)) =e$ if $a_2=0$ and $\deg (\rho(f))= e+1$ if $a_2>0$.
We consider the case when $a_2>0$ (the proof for the case when $a_2=0$ is similar).

Consider 
$U'=U \upp 2 \biguplus (\biguplus_{i=3}^t U \upp i [-1])$.
Since any monomial $w \in U \upp 1_e$ such that $\rho(w)=\delta_2 x_2^k$ with $k \leq e+1-b_2$ is admissible over $U$,
we have $\rho([g,h)) \subset S_d$ for some $d \leq e+1$.
Let
$$A=x_2^{e+1-d}\rho\big( [g,h) \big)= \big[ x_2^{e+1-d}\rho(g),x_2^{e+1-d}\rho(h) \big) \subset U_{e+1} \upp 2.$$
(See Fig.\ 10.)

\begin{center}
\unitlength 0.1in
\begin{picture}( 24.0000, 16.8500)( 12.0000,-20.0000)
%
\special{pn 13}%
\special{pa 2200 1200}%
\special{pa 2600 1200}%
\special{fp}%
\special{pa 2600 1200}%
\special{pa 2600 1400}%
\special{fp}%
\special{pa 2600 1400}%
\special{pa 2200 1400}%
\special{fp}%
\special{pa 2200 1400}%
\special{pa 2200 1200}%
\special{fp}%
%
\special{pn 13}%
\special{pa 2200 600}%
\special{pa 2600 600}%
\special{fp}%
\special{pa 2600 600}%
\special{pa 2600 800}%
\special{fp}%
\special{pa 2600 800}%
\special{pa 2200 800}%
\special{fp}%
\put(24.1000,-7.0000){\makebox(0,0){$A$}}%
%
\special{pn 8}%
\special{pa 1640 900}%
\special{pa 2240 1300}%
\special{fp}%
\special{sh 1}%
\special{pa 2240 1300}%
\special{pa 2196 1246}%
\special{pa 2196 1270}%
\special{pa 2174 1280}%
\special{pa 2240 1300}%
\special{fp}%
\put(24.0000,-4.0000){\makebox(0,0){Figure 10}}%
\put(14.6000,-8.9000){\makebox(0,0){$g$}}%
\put(16.0000,-8.9000){\makebox(0,0){$h$}}%
%
\special{pn 13}%
\special{pa 2210 600}%
\special{pa 2210 800}%
\special{fp}%
%
\special{pn 13}%
\special{pa 1660 1000}%
\special{pa 1660 800}%
\special{fp}%
\special{pa 1660 800}%
\special{pa 1400 800}%
\special{fp}%
\special{pa 1400 800}%
\special{pa 1400 1000}%
\special{fp}%
%
\special{pn 13}%
\special{pa 1400 1000}%
\special{pa 1660 1000}%
\special{fp}%
%
\special{pn 13}%
\special{pa 1200 600}%
\special{pa 1200 2000}%
\special{fp}%
\special{pa 1200 2000}%
\special{pa 2000 2000}%
\special{fp}%
\special{pa 2000 2000}%
\special{pa 2000 600}%
\special{fp}%
\special{pa 2000 1600}%
\special{pa 2800 1600}%
\special{fp}%
\special{pa 2800 600}%
\special{pa 2800 1600}%
\special{fp}%
\special{pa 2800 1400}%
\special{pa 3600 1400}%
\special{fp}%
\special{pa 3600 1400}%
\special{pa 3600 600}%
\special{fp}%
%
\special{pn 8}%
\special{pa 2290 1230}%
\special{pa 2290 740}%
\special{fp}%
\special{sh 1}%
\special{pa 2290 740}%
\special{pa 2270 808}%
\special{pa 2290 794}%
\special{pa 2310 808}%
\special{pa 2290 740}%
\special{fp}%
%
\special{pn 8}%
\special{pa 3600 600}%
\special{pa 1200 600}%
\special{dt 0.045}%
\special{pa 1200 800}%
\special{pa 3600 800}%
\special{dt 0.045}%
\special{pa 3600 1000}%
\special{pa 1200 1000}%
\special{dt 0.045}%
\special{pa 3600 1200}%
\special{pa 1200 1200}%
\special{dt 0.045}%
\special{pa 1200 1400}%
\special{pa 2800 1400}%
\special{dt 0.045}%
\special{pa 2000 1600}%
\special{pa 1200 1600}%
\special{dt 0.045}%
\special{pa 1200 1800}%
\special{pa 2000 1800}%
\special{dt 0.045}%
\end{picture}%
\end{center}
\bigskip

Let $w \in A$.
Then $w=x_2^{e+1-d}\rho(w')$ for some $w' \in [g,h)$.
Lemma \ref{X1} says that $w$ is admissible over $U'$
if and only if $w'$ is admissible over $U$.
Hence $A$ contains no admissible monomial over $U'$
except for $x_2^{e+1-d} \rho(g)$.
By Lemma \ref{X1},
$\rho(f) \in U_{e+1} \upp 2$ is the lex-smallest admissible monomial in $U_{e+1} \upp 2$ over $U'$.
Then, by the induction hypothesis,
$$\# A \leq \# [\rho(f), \delta_2 x_n^{e-b_2}]
=\# \rho \big([f, \delta_1 x_n^{e-b_1}]\big) \cap U_{e+1} \upp 2 \leq \# [f,\delta_1 x_n^{e-b_1}].$$
Then the statement follows since $\#[g,h)=\#\rho([g,h))=\# A$.
\end{proof}

\begin{lemma}
\label{X7}
Let $M \subset U$ be an extremal set, $e=\min\{k : \delta_1 x_1 ^{k-b_1} \not \in M\}$
and $H=U_e \setminus M_e$.
Let $f \in U_e \upp 1$ be the lex-smallest admissible monomial in $U_e \upp 1$ over $U$.
Then
\begin{itemize}
\item[(i)] $\# U_{\leq e} + \#[f,\delta_1 x_n^{e-b_1}] \leq \# U_{\leq e+1} \upp 1$.
\item[(ii)] $\#M + \# H < \# U_{\leq e+1} \upp 1.$
\end{itemize}
\end{lemma}

\begin{proof}
We use induction on $t$.
If $t=1$ then then the statements are obvious.
Suppose $t>1$.

(i)
If $a_2 >0$ then by Lemma \ref{X3}
\begin{eqnarray*}
\ \#U_{\leq e} + \# [f,\delta_1 x_n^{e-b_1}]
=\# \{ h \in U: h \leq_\dlex f\} + \# U_e \upp 1
= \# U_{\leq e} \upp 1 + \# U_e \upp 1 < \# U_{\leq e+1} \upp 1
\end{eqnarray*}
as desired.
Suppose $a_2=0$.
Then
$$\rho\big( [f,\delta_1 x_n^{e-b_1}] \big) = [\rho(f),\delta_2 x_n^{e-b_2}] \subset U_e \upp 2$$
and $\rho(f)$ is the lex-smallest admissible monomial in $U \upp 2_e$ over $U \uppeq 2$
by Lemma \ref{X1}.
Then by the induction hypothesis
\begin{eqnarray*}
\#U_{\leq e} + \# [f,\delta_1 x_n^{e-b_1}]
&=& \# U_{\leq e} \upp 1 +\big( \# U_{\leq e} \uppeq 2 + \# [\rho(f),\delta_2 x_n^{e-b_2}]\big)\\
&\leq& \#U_{\leq e} \upp 1 + \# U_{\leq e+1} \upp 2\\
&=& \# U_{\leq e+1} \upp 1
\end{eqnarray*}
as desired.

(ii)
Suppose $M \upp 2_e \ne U \upp2_e$.
Then $M_e \upp 1 = \emptyset$.
Since $M \uppeq 2$ is extremal over $U \uppeq 2$,
by the induction hypothesis
\begin{eqnarray*}
\#M + \#H = \# U_{\leq e-1}\upp 1 \biguplus M \uppeq 2  +\# U_e \upp 1 \biguplus H \uppeq 2
< \# U_{\leq e} \upp 1 + \# U_{\leq e+1} \upp 2
\leq \# U_{\leq e+1} \upp 1,
\end{eqnarray*}
where we use $\# U \upp 1_{e+1} = \# U_{\leq e+1 +a_2} \upp 2 \geq \# U_{\leq e+1} \upp 2$
for the last inequality.

Suppose $M \upp 2 _e = U \upp 2 _e$.
Let $g= \max_{>_\dlex} M \upp 1$ and let
$$\mu=\min_{>_\dlex}\{ h \in U_{\leq e}\upp 1: h \mbox{ is admissible over $U$ and }h>_\dlex g\}.$$
Then $[\mu,g) \subset U \upp 1_e$ since $g \geq_\dlex \delta_1 x_1^{e-b_1-1}$.
Since $M$ is extremal,
$$\#M < \#\{h \in U: h \leq_\dlex \mu\}.$$
Since $M \upp 1=\{h \in U \upp 1: h \leq_\dlex g\}$,
$H=[\delta_1 x_1^{e-b_1},g)$.
Thus
\begin{eqnarray*}
\#M + \#H 
&<& \# \{ h \in U: h \leq_\dlex \mu\} + \# [\delta_1x_1^{e-b_1},g)\\
&=& \# U_{\leq e} + \# [\mu,g)\\
&\leq & \# U_{\leq e} + \# [f,\delta_1x_n^{e-b_1}],
\end{eqnarray*}
where the last inequality follows from Lemma \ref{X9}.
Then the desired inequality follows from (i).
\end{proof}

\section{Proof of the main theorem}

Let $U$ be the universal lex ideal as in Section 5.
The aim of this section is to prove the next result, which proves Proposition \ref{main}.

\begin{theorem}
\label{extremal}
Let $M \subset U$ be a ladder set and $L \subset U$ the extremal set with $\#L=\#M$.
Then $m(L) \succeq m(M)$.
\end{theorem}

The proof of the above theorem is long.
We prove it in subsections 6.1, 6.2 and 6.3 by case analysis.

In the rest of this section, we fix a ladder set $M \subset U$.

\subsection{Preliminary of the proof}\ 

For two subsets $A,B \subset U$,
we define 
$$A \gg B \Leftrightarrow \# A = \# B \mbox{ and }m (A) \succeq m(B).$$

Let $X \subset U\upp 1$ be the super rev-lex set with $\# X = \# M \upp 1$.
Then $\{k: M_k\upp1 \ne \emptyset\} \supset \{k : X_k \ne \emptyset\}$.
Thus $X \cup M \uppeq 2$ is also a ladder set in $U$.
Since $X \gg M \upp 1$ by Corollary \ref{ERV},
we have
 
\begin{lemma}
There exists a ladder set $N \subset U$ such that $N \upp 1$ is super rev-lex
and $N \gg M$.
\end{lemma}

Thus in the rest of this section,
we assume that $M \upp 1$ is super rev-lex.
Let
$$e= \min \{k : \delta_1 x_1^{k-b_1} \not \in M\}$$
and
$$f = \max_{>_\dlex} \{ g \in U \upp 1_{\leq e}: g \mbox{ is admissible over $U$ and } \# \{h \in U : h \leq_\dlex g \} \leq \#M\}.$$
Since $\delta_1 x_1^{e-b_1-1}$ is admissible over $U$,
we have $f=\delta_1 x_1^{e-b_1-1}$ or $\deg f=e$.
We will prove

\begin{proposition}
\label{key}
With the same notation as above,
there exists a ladder set $N$ such that $N \gg M$ and
$$N \upp 1 = \{ h \in U \upp 1: h \leq_\dlex f\}.$$ 
\end{proposition}

The above proposition proves Theorem \ref{extremal}.
Indeed,
by applying the above proposition repeatedly, one obtains a set $N$ which satisfies the maximal condition
and $N \gg M$.
Then apply the induction on $t$.
Also if $t=1$ then Proposition \ref{key} follows from Corollary \ref{ERV}.
In the rest of this section,
we assume that $t>1$ and that the statement is true for universal lex ideals generated by at most $t-1$ monomials,
and prove the proposition for $U$.
By the above argument, we may assume that Theorem \ref{extremal} is also true for universal lex ideals generated by at most $t-1$ monomials.

\begin{lemma}
\label{A}
There exists a ladder set $N \subset U$ with $N \gg M$ and $\min\{k : \delta_1 x_1^{k-b_1} \not \in N \upp 1\}=e$
satisfying the following conditions
\begin{itemize}
\item[(A1)] $N \upp 1$ is super rev-lex and $N \uppeq 2$ is extremal in $U \uppeq 2$.
\item[(A2)] $\rho(N \upp 1_e) \cup N \upp 2 \supset U_{\leq e+a_2} \upp 2$ or $\rho(N \upp 1_e) \cap N \upp 2 = \emptyset$.
\item[(A3)] If $t=2$ and $\rho(N \upp 1_e) \cap N \upp 2 = \emptyset$ then $N \upp 1_e = \emptyset$.
If $t \geq 3$ and $\rho(N \upp 1_e) \cap N \upp 2 = \emptyset$ then
$N \upp 1_e= \emptyset$ or there exists a $d \geq e$ such that $N \upp 2 = U_{\leq d} \upp 2$ and $N_{d+1} \upp 3 \ne U \upp 3 _{d+1}$.
\end{itemize}
\end{lemma}

\begin{proof}
Let $F= M\upp 1_e$.
Then $M = (U_{\leq e-1} \upp 1 \biguplus F) \biguplus M \upp 2 \biguplus M \uppeq 3$
since $M \upp 1$ is super rev-lex.

\textit{Step 1}.
We first prove that there exits $N$ satisfying $(A1)$.
Let $X$ be the extremal set in $U \uppeq 2$ with $\# X= \#M \uppeq 2$.
Let
$$N= M \upp 1 \biguplus X=U_{\leq e-1} \upp 1 \biguplus F \biguplus X.$$
Since we assume that Theorem \ref{extremal} is true for $U \uppeq 2$,
$N \gg M$.
What we must prove is that $N$ is a ladder set.
Since $M \uppeq 2 \supset U_{\leq e-1} \uppeq 2$,
$\# X = \# M \uppeq 2 \geq \# U_{\leq e-1} \uppeq 2$.
Then Lemma \ref{X3} says $X \supset U_{\leq e-1} \uppeq 2$,
which shows that $N$ is a ladder set if $F = \emptyset$.
If $F \ne \emptyset$ then $M \uppeq 2 \supset U_{\leq e} \uppeq 2$ by the definition of ladder sets,
and $X \supset U_{\leq e} \uppeq 2$ by Lemma \ref{X3}.
Hence $N$ is a ladder set.

\textit{Step 2}.
We prove that if $M$ satisfies $(A1)$ but does not satisfy either $(A2)$ or $(A3)$ then there exists an $N$ satisfying (A2) and (A3) such that $\# N \upp 1$ is strictly smaller than $\# M \upp 1$.
We may assume $\rho(F) \cup M \upp 2 \not \supset U_{\leq e+a_2} \upp2$.
Let
$$a= \min\{k: M_k \upp 2 \ne U_k \upp 2\},$$
$$b= \max\{k: k \leq e+a_2,\ \rho(F)_k \ne U_k \upp 2\}$$
and
$$d= \max\{k: M_k \upp 3 = U_k \upp 3\}$$
where $d= \infty$ if $n=2$.
Let
$H = U_{\leq d} \upp 2 \setminus M \upp 2$.
(See Fig.\ 11.)

\begin{center}
\unitlength 0.1in
\begin{picture}( 29.4500, 18.0000)( -1.4500,-21.1500)
%
\special{pn 4}%
\special{pa 720 1400}%
\special{pa 520 1600}%
\special{dt 0.027}%
\special{pa 600 1400}%
\special{pa 410 1590}%
\special{dt 0.027}%
\special{pa 480 1400}%
\special{pa 410 1470}%
\special{dt 0.027}%
\special{pa 840 1400}%
\special{pa 640 1600}%
\special{dt 0.027}%
\special{pa 960 1400}%
\special{pa 760 1600}%
\special{dt 0.027}%
\special{pa 1080 1400}%
\special{pa 880 1600}%
\special{dt 0.027}%
\special{pa 1190 1410}%
\special{pa 1000 1600}%
\special{dt 0.027}%
\special{pa 1190 1530}%
\special{pa 1120 1600}%
\special{dt 0.027}%
%
\special{pn 4}%
\special{pa 720 1600}%
\special{pa 520 1800}%
\special{dt 0.027}%
\special{pa 600 1600}%
\special{pa 410 1790}%
\special{dt 0.027}%
\special{pa 480 1600}%
\special{pa 410 1670}%
\special{dt 0.027}%
\special{pa 840 1600}%
\special{pa 640 1800}%
\special{dt 0.027}%
\special{pa 960 1600}%
\special{pa 760 1800}%
\special{dt 0.027}%
\special{pa 1080 1600}%
\special{pa 880 1800}%
\special{dt 0.027}%
\special{pa 1190 1610}%
\special{pa 1000 1800}%
\special{dt 0.027}%
\special{pa 1190 1730}%
\special{pa 1120 1800}%
\special{dt 0.027}%
%
\special{pn 4}%
\special{pa 720 1800}%
\special{pa 520 2000}%
\special{dt 0.027}%
\special{pa 600 1800}%
\special{pa 410 1990}%
\special{dt 0.027}%
\special{pa 480 1800}%
\special{pa 410 1870}%
\special{dt 0.027}%
\special{pa 840 1800}%
\special{pa 640 2000}%
\special{dt 0.027}%
\special{pa 960 1800}%
\special{pa 760 2000}%
\special{dt 0.027}%
\special{pa 1080 1800}%
\special{pa 880 2000}%
\special{dt 0.027}%
\special{pa 1190 1810}%
\special{pa 1000 2000}%
\special{dt 0.027}%
\special{pa 1190 1930}%
\special{pa 1120 2000}%
\special{dt 0.027}%
%
\special{pn 4}%
\special{pa 1520 1400}%
\special{pa 1320 1600}%
\special{dt 0.027}%
\special{pa 1400 1400}%
\special{pa 1210 1590}%
\special{dt 0.027}%
\special{pa 1280 1400}%
\special{pa 1210 1470}%
\special{dt 0.027}%
\special{pa 1640 1400}%
\special{pa 1440 1600}%
\special{dt 0.027}%
\special{pa 1760 1400}%
\special{pa 1560 1600}%
\special{dt 0.027}%
\special{pa 1880 1400}%
\special{pa 1680 1600}%
\special{dt 0.027}%
\special{pa 1990 1410}%
\special{pa 1800 1600}%
\special{dt 0.027}%
\special{pa 1990 1530}%
\special{pa 1920 1600}%
\special{dt 0.027}%
%
\special{pn 4}%
\special{pa 2320 1200}%
\special{pa 2120 1400}%
\special{dt 0.027}%
\special{pa 2200 1200}%
\special{pa 2010 1390}%
\special{dt 0.027}%
\special{pa 2080 1200}%
\special{pa 2010 1270}%
\special{dt 0.027}%
\special{pa 2440 1200}%
\special{pa 2240 1400}%
\special{dt 0.027}%
\special{pa 2560 1200}%
\special{pa 2360 1400}%
\special{dt 0.027}%
\special{pa 2680 1200}%
\special{pa 2480 1400}%
\special{dt 0.027}%
\special{pa 2790 1210}%
\special{pa 2600 1400}%
\special{dt 0.027}%
\special{pa 2790 1330}%
\special{pa 2720 1400}%
\special{dt 0.027}%
%
\special{pn 4}%
\special{pa 2320 1000}%
\special{pa 2120 1200}%
\special{dt 0.027}%
\special{pa 2200 1000}%
\special{pa 2010 1190}%
\special{dt 0.027}%
\special{pa 2080 1000}%
\special{pa 2010 1070}%
\special{dt 0.027}%
\special{pa 2440 1000}%
\special{pa 2240 1200}%
\special{dt 0.027}%
\special{pa 2560 1000}%
\special{pa 2360 1200}%
\special{dt 0.027}%
\special{pa 2680 1000}%
\special{pa 2480 1200}%
\special{dt 0.027}%
\special{pa 2790 1010}%
\special{pa 2600 1200}%
\special{dt 0.027}%
\special{pa 2790 1130}%
\special{pa 2720 1200}%
\special{dt 0.027}%
%
\special{pn 4}%
\special{pa 2320 800}%
\special{pa 2120 1000}%
\special{dt 0.027}%
\special{pa 2200 800}%
\special{pa 2010 990}%
\special{dt 0.027}%
\special{pa 2080 800}%
\special{pa 2010 870}%
\special{dt 0.027}%
\special{pa 2440 800}%
\special{pa 2240 1000}%
\special{dt 0.027}%
\special{pa 2560 800}%
\special{pa 2360 1000}%
\special{dt 0.027}%
\special{pa 2680 800}%
\special{pa 2480 1000}%
\special{dt 0.027}%
\special{pa 2790 810}%
\special{pa 2600 1000}%
\special{dt 0.027}%
\special{pa 2790 930}%
\special{pa 2720 1000}%
\special{dt 0.027}%
%
\special{pn 4}%
\special{pa 1040 1200}%
\special{pa 840 1400}%
\special{dt 0.027}%
\special{pa 1160 1200}%
\special{pa 960 1400}%
\special{dt 0.027}%
\special{pa 1190 1290}%
\special{pa 1080 1400}%
\special{dt 0.027}%
\special{pa 920 1200}%
\special{pa 800 1320}%
\special{dt 0.027}%
%
\special{pn 4}%
\special{pa 1520 1200}%
\special{pa 1320 1400}%
\special{dt 0.027}%
\special{pa 1400 1200}%
\special{pa 1210 1390}%
\special{dt 0.027}%
\special{pa 1280 1200}%
\special{pa 1210 1270}%
\special{dt 0.027}%
\special{pa 1640 1200}%
\special{pa 1440 1400}%
\special{dt 0.027}%
\special{pa 1760 1200}%
\special{pa 1560 1400}%
\special{dt 0.027}%
\special{pa 1880 1200}%
\special{pa 1680 1400}%
\special{dt 0.027}%
\special{pa 1990 1210}%
\special{pa 1800 1400}%
\special{dt 0.027}%
\special{pa 1990 1330}%
\special{pa 1920 1400}%
\special{dt 0.027}%
%
\special{pn 4}%
\special{pa 1990 1090}%
\special{pa 1880 1200}%
\special{dt 0.027}%
\special{pa 1960 1000}%
\special{pa 1800 1160}%
\special{dt 0.027}%
\special{pa 1840 1000}%
\special{pa 1800 1040}%
\special{dt 0.027}%
%
\special{pn 8}%
\special{pa 400 1400}%
\special{pa 800 1400}%
\special{fp}%
\special{pa 800 1400}%
\special{pa 800 1200}%
\special{fp}%
\special{pa 800 1200}%
\special{pa 1800 1200}%
\special{fp}%
\special{pa 1800 1200}%
\special{pa 1800 1000}%
\special{fp}%
\special{pa 1800 1000}%
\special{pa 2000 1000}%
\special{fp}%
\special{pa 2000 800}%
\special{pa 2400 800}%
\special{fp}%
\special{pa 2400 800}%
\special{pa 2400 600}%
\special{fp}%
\special{pa 2400 600}%
\special{pa 2800 600}%
\special{fp}%
%
\special{pn 8}%
\special{pa 2800 600}%
\special{pa 400 600}%
\special{dt 0.045}%
\special{pa 400 800}%
\special{pa 2800 800}%
\special{dt 0.045}%
\special{pa 2800 1000}%
\special{pa 400 1000}%
\special{dt 0.045}%
\special{pa 400 1200}%
\special{pa 2800 1200}%
\special{dt 0.045}%
\special{pa 2000 1400}%
\special{pa 400 1400}%
\special{dt 0.045}%
\special{pa 400 1600}%
\special{pa 1200 1600}%
\special{dt 0.045}%
\special{pa 1200 1800}%
\special{pa 400 1800}%
\special{dt 0.045}%
%
\special{pn 13}%
\special{pa 400 600}%
\special{pa 400 2000}%
\special{fp}%
\special{pa 400 2000}%
\special{pa 1200 2000}%
\special{fp}%
\special{pa 1200 2000}%
\special{pa 1200 600}%
\special{fp}%
\special{pa 1200 1600}%
\special{pa 2000 1600}%
\special{fp}%
\special{pa 2000 1600}%
\special{pa 2000 600}%
\special{fp}%
\special{pa 2000 1400}%
\special{pa 2800 1400}%
\special{fp}%
\special{pa 2800 1400}%
\special{pa 2800 600}%
\special{fp}%
%
\special{pn 4}%
\special{pa 2640 600}%
\special{pa 2440 800}%
\special{dt 0.027}%
\special{pa 2760 600}%
\special{pa 2560 800}%
\special{dt 0.027}%
\special{pa 2790 690}%
\special{pa 2680 800}%
\special{dt 0.027}%
\special{pa 2520 600}%
\special{pa 2400 720}%
\special{dt 0.027}%
%
\special{pn 8}%
\special{pa 2000 800}%
\special{pa 1200 800}%
\special{fp}%
\put(15.5000,-10.0000){\makebox(0,0){$H$}}%
\put(8.0000,-18.0000){\makebox(0,0){$U_{\leq e-1} \upp 1$}}%
\put(10.0000,-13.0000){\makebox(0,0){$F$}}%
%
\special{pn 8}%
\special{pa 1200 1400}%
\special{pa 800 1400}%
\special{fp}%
\put(16.0000,-22.0000){\makebox(0,0){$M$}}%
\put(16.0000,-4.0000){\makebox(0,0){Figure 11}}%
\end{picture}%
\end{center}
\bigskip

Since $\rho(F)$ is an upper rev-lex set of degree $e+a_2$,
$\rho(F)=\rho(F)_b \biguplus(\biguplus_{j=b+1}^{e+a_2} U_j\upp 2)$.
If $H = \emptyset$ then
$M \upp 2 = U_{\leq d} \upp 2$.
Since $\rho(F) \cup M \upp 2 \not \supset U_{\leq e+a_2} \upp 2$,
we have $b>d$ and $\rho(F) \cap M \upp 2 = \emptyset$,
which says that $M$ satisfies $(A2)$ and $(A3)$.
Suppose $H \ne \emptyset$.
Observe that for any super rev-lex set $L $ with $U_{\leq e} \upp 2 \subset L \subset U_{\leq d} \upp 2$,
$M \upp 1 \biguplus L \biguplus M \uppeq 3$ is a ladder set.

\textit{Case 1}:
Suppose $\# H \geq \# F$.
Note that if $t=2$ then we always have $\#H \geq \#F$.
Then $M \upp 2$ is super rev-lex and $\rho(F)$ is an upper rev-lex set of degree $e+a_2$ with $\# M \upp 2 + \# \rho(F) \leq \# U_{\leq d} \upp 2$.
Let $R \subset U \upp 2$ be the super rev-lex set in $U \upp 2$ with $\#R = \# M \upp 2 + \# \rho(F)$.
By Corollary \ref{X6},
\begin{eqnarray}
\label{*}
m(R) \succeq   m(M \upp 2) + m \big( \rho(F)\big)= m(M \upp 2 ) + m(F).
\end{eqnarray}
Also, since $R$ is super rev-lex,
$U_{\leq e} \upp 2 \subset R \subset U_{\leq d} \upp 2$.
Thus
$$N= U_{\leq e-1} \upp 1 \biguplus R \biguplus M \uppeq 3$$
is a ladder set.
Then $N \upp 1_e=\emptyset$ and $N \gg M$ by (\ref{*}).
Hence $N$ satisfies (A2) and (A3).
\medskip

\textit{Case 2}:
Suppose $\# H < \#F$.
Observe that $M \upp 2 \cup \rho(F)$ contains all monomials of degree $k$ in $U \upp 2$ for $k < a$ and $b < k \leq e+a_2$.
Since $M \cup \rho(F) \not \supset U_{\leq e+a_2} \upp 2$,
we have $a \leq b$.

Let $I \subset \rho(F)$ be the interval in $U \upp 2$ such that $\#I = \# H_a$
and $\rho(F) \setminus I$ is an upper rev-lex set of degree $e+a_2$,
and let $F' \subset F$ be the rev-lex set with $\rho(F') = \rho(F) \setminus I$.
Since $H_a$ is a lower lex set of degree $a$,
by the interval lemma,
$$m\big(M \upp 2\big) + m\big(\rho(F)\big) \ll 
m\left(H_a \biguplus M \upp 2\right)+ m\big(\rho(F) \setminus I\big)=m\big(U_{\leq a} \upp 2\big) + m \big(\rho(F') \big).$$
(See Fig.\ 12.)

\begin{center}
\input{Fig12}
\end{center}
\bigskip

\noindent If $\rho(F') \cup U_{\leq a} \upp 2 \supset U_{\leq e+a_2} \upp 2$ then
$$N= \big(U_{\leq e-1} \upp 1 \biguplus F'\big) \biguplus U_{\leq a} \upp 2 \biguplus M \uppeq 3$$
is a ladder set and satisfies $N \gg M$ and conditions (A2) and (A3) since
$\rho(N \upp 1_e) \cup N \upp 2 \supset U_{\leq e+a_2} \upp 2$.

Suppose $\rho(F') \cup U_{\leq a} \upp 2 \not \supset U_{\leq e+a_2} \upp 2$.
Then $\rho(F') \subset \biguplus_{j=a+1} ^{e+a_2} U_j \upp 2$.
Since we assume $\#H < \# F$,
$\# F' = \#F - \#H_a > \# (H \setminus H_a)$.
Let $J \subset \rho(F')$ be the interval in $U \upp 2$ such that
$\# J = \#(H \setminus H_a)$ and $\rho(F')\setminus J$ is an upper rev-lex set of degree $e+a_2$,
and let $F'' \subset F'$ be the rev-lex set satisfying $\rho(F'')=\rho(F') \setminus J$.
Since $H \setminus H_a = \biguplus_{j=a+1} ^d U_j \upp 2$ is a lower lex set of degree $a+1$,
by the interval lemma
$$ m\big(U_{\leq a} \upp 2\big) + m\big(\rho(F')\big)
\preceq 
m\big(M \upp 2 \biguplus H\big) + m\big(\rho(F'')\big)=
 m\big(U_{\leq d} \upp 2\big) + m \big(\rho(F'') \big) .$$
(See Fig.\ 13.)

\begin{center}
\input{Fig13}
\end{center}
\bigskip

\noindent Then
$$N=\big( U_{\leq e-1} \upp 1 \biguplus F''\big) \biguplus U_{\leq d} \upp 2 \biguplus M \uppeq 3$$
is a ladder set and satisfies $N \gg M$ and conditions (A2) and (A3).

Finally,
since Step 1 does not change the first component $M \upp 1$ and Step 2 decreases the first component,
by applying Step 1 and 2 repeatedly,
we obtain a set $N \subset U$ satisfying conditions (A1), (A2) and (A3).
\end{proof}

Lemma \ref{A} says that to prove Proposition \ref{key}
we may assume that $M$ satisfies (A1), (A2) and (A3).
Thus in the rest of this section we assume that $M$ satisfies these conditions.

\subsection{Proof of Proposition \ref{key} when $f \ne \delta_1 x_1^{e-b_1-1}$.}\

In this subsection,
we prove Proposition \ref{key} when $f \ne \delta_1 x_1^{e-b_1-1}$.
In this case we have $\deg f =e$.
Let
$$f=\delta_1 x_1 ^{\alpha_1} \cdots x_n^{\alpha_n}$$
and $F=M \upp 1_e$.
Since $\delta_1 x_1^{e-b_1} \not \in F$ by the choice of $e$,
we have $m(F)=m(\rho(F))$.
Also we have
$$M \uppeq 2 \supset  U_{\leq e} \uppeq 2.$$
Indeed, this is obvious when $F \ne \emptyset$ by the definition of ladder sets.
If $F = \emptyset$ then 
$$\# M \uppeq 2 = \# M - \# U_{\leq e-1} \upp 1 \geq \#\{ h \in U : h \leq_\dlex  f\} - \# U_{\leq e-1} \upp 1 \geq \# U_{\leq e} \upp 2,$$
and since $M \uppeq 2$ is extremal
we have $M \uppeq 2 \supset U_{\leq e} \uppeq 2$ by Lemma \ref{X4}.
Let
$$\epsilon = \deg \rho(f) = \alpha_2 + \cdots + \alpha_n + b_2.$$

\textit{Case 1}.
Suppose $\rho(F) \subset \biguplus_{j=\epsilon}^{e+a_2} U_j \upp 2$ and
$\# F + \# M\upp 2  \setminus \biguplus_{j=\epsilon} ^e U_j \upp 2 \leq \# U_{\leq e+a_2} \upp 2$.
Observe $M \upp 2 \supset \biguplus_{j=\epsilon}^e U_j \upp 2$.
Let $P$ be the super rev-lex set with
$\#P = \# M \upp 2 \setminus \biguplus_{j=\epsilon}^e U_j \upp 2$,
and let $Q \subset U \upp 2$ be the super rev-lex set with
$\# Q= \#F + \# M \upp 2 \setminus \biguplus_{j=\epsilon}^e U_j \upp 2$.
Since $\rho(F)$ is an upper rev-lex set of degree $e+a_2$
and $M \upp 2 \setminus \biguplus_{j=\epsilon}^e U_j \upp 2$ is rev-lex,
by Corollaries \ref{X6} and \ref{ERV}, we have
\begin{eqnarray}
\label{b1}
m(Q)\succeq m(P) + m \big(\rho(F) \big)\succeq
m \big(M \upp 2 \setminus \biguplus_{j=\epsilon}^e U_j \upp 2 \big) + m(F).
\end{eqnarray}
(See the first two steps in Fig.\ 15.)

Observe that $Q \subset U_{\leq e+a_2} \upp 2$
since $\#Q \leq \# U_{\leq e+a_2} \upp 2$ by the assumption of Case 1.
Let $U '=U \upp 2 \biguplus(\biguplus_{i=3}^t U \upp i [-a_2])$.
Since $M \uppeq 3 [-a_2] \supset U_{\leq e} \uppeq 3 [-a_2] = {U'}_{\leq e+a_2} \uppeq 3$,
$$Q \biguplus M \uppeq 3 [-a_2] \subset U'$$
is a ladder set in $U'$.
(See the third step in Fig.\ 15.)

Let $g$ be the largest admissible monomial in $U_{\leq e+a_2} \upp 2$ over $U'$ with respect to $>_\dlex$
satisfying
$$\#\{h \in U': h \leq_\dlex g\} \leq \# Q \biguplus M \uppeq 3.$$
By the induction hypothesis, there exists $Y \subset {U'} \uppeq 3$
such that
$$X= \{h \in U \upp 2: h \leq_{\dlex} g \} \biguplus Y \subset U'$$
is a ladder set in $U'$ and 
\begin{eqnarray}
\label{7.5}
X \gg Q \biguplus M \uppeq 3.
\end{eqnarray}
Let 
$$d= e+a_2 -\epsilon.$$
We claim
\begin{lemma}
$g \geq_\lex x_2^{d} \rho(f).$
\end{lemma}

\begin{proof}
To prove this, consider
$$L= \{ h \in U: h \leq _\dlex f\}.$$
Then $\#M \geq \#L$ and $L \uppeq 2 = U_{\leq e} \uppeq 2$.
Let $F' =  L_e \upp 1 = [f,\delta_1 x_n^{e-b_1}]$.
Then $\rho(F') =[\rho(f), \delta_2 x_n^{\epsilon-b_2}] \biguplus(\biguplus_{j=\epsilon+1}^{e+a_2} U_j \upp 2)$.
Also $\rho(F') \cap (L \upp 2 \setminus \biguplus_{j=\epsilon}^e U_j\upp 2) = \emptyset$
and
\begin{eqnarray*}
m\big(\rho(F') \biguplus \big(L\upp 2 \setminus \biguplus_{j=\epsilon}^e U_j\upp 2\big)\big)
&=& m\big(U_{\leq e+a_2} \upp 2\setminus \big[\delta_2 x_2^{\epsilon -b_2},\rho(f)\big)\big)\\
&=& m \big(U_{\leq e+a_2} \upp 2\setminus \big [\delta_2 x_2^{e+a_2 -b_2},x_2^d \rho(f) \big) \big).
\end{eqnarray*}
Let
$$R = U_{\leq e+a_2} \upp 2 \setminus \big[\delta_2 x_2^{e+a_2 -b_2},x_2^{d}\rho(f) \big)
= U_{\leq e+a_2-1} \upp 2 \biguplus \big[x_2^{d}\rho(f),\delta_2 x_n^{e+a_2 -b_2}\big].$$
(See Fig.\ 14).

\begin{center}
\input{Fig14}
\end{center}
\bigskip

\noindent
Then $R \biguplus L \uppeq 3 [-a_2] \subset U '$
is a ladder set in $U'$ and $x_2^{d} \rho(f)$ is admissible over $U'$
by Lemma \ref{X1}.
On the other hand,
$$\# R \biguplus L \uppeq 3 =\#L -\# U_{\leq e-1}\upp 1 -\# \biguplus_{j=\epsilon}^e U_j \upp 2
\leq \#M - \# U_{\leq e-1} \upp 1 -\# \biguplus_{j=\epsilon}^e U_j \upp 2
= \#X.$$
Since $x_2^{d} \rho(f)$ is admissible over $U'$
and since $R \biguplus L \uppeq 3 [-a_2] = \{ h \in U': h \leq _{\dlex} x_2^d \rho(f)\}$,
by the choice of $g$, we have
$$g \geq_{\lex} x_2^{d} \rho(f)$$
as desired.
\end{proof}

Let $H \subset U_e \upp 1$ be the rev-lex set such that
$$\rho(H) = \biguplus_{j=\epsilon}^{e+a_2} U_j \upp 2 \setminus x_2^{-d} \big[ \delta_2 x_2^{e+a_2-b_2},g\big).$$

Then by Lemma \ref{notchange}
\begin{eqnarray}
\label{b4}
m(H) + m \big(U_{\leq \epsilon -1} \upp 2 \big) \succeq m \big(U_{\leq e+a_2} \upp 2 \setminus [\delta_2x_2^{e+a_2-b_2},g)\big) = m\big(X \upp 2\big).
\end{eqnarray}
Let
$$N=\big (U _{\leq e-1} \upp 1 \biguplus H \big)
\biguplus U_{\leq e} \upp 2 \biguplus Y [+a_2] \subset U.$$
Since $Y \supset {U'}_{\leq e+a_2} \uppeq 3$, 
we have $Y[+a_2] \supset {U}_{\leq e} \uppeq 3$.
Thus $N$ is a ladder set in $U$.
We claim that $N$ satisfies the desired conditions.

Let $\mu =\max_{>_\lex} H$.
Then $x_2^{d} \rho(\mu) =g$.
We claim that $\mu =f$.
Since $g \geq_\lex x_2^d \rho(f)$, $\mu \geq_\lex f$.
Since $g$ is admissible over $U'$, $\mu$ is admissible over $U$ by Lemma \ref{X1}.
(If $t=2$ then Lemma \ref{X1} is not applicable, however, if $t=2$ then any monomial $h \in U_e \upp 1$ with $h >_\lex f$ is admissible).
However, since $\#N = \#M$ and $N \supset \{h \in U: h \leq _\dlex \mu\}$,
by the choice of $f$, we have $f= \mu$.

It remains to prove $N \gg M$.
This follows from (\ref{b1}), (\ref{7.5}) and (\ref{b4}) as follows:
\begin{eqnarray*}
M \setminus \biguplus_{j=\epsilon}^e U_j \upp 2
&=& \big(U_{\leq e-1} \upp 1 \biguplus F\big) \biguplus \big(M \upp 2 \setminus \biguplus_{j=\epsilon}^e U_j \upp 2 \big) \biguplus M \uppeq 3\\
&\ll& U_{\leq e-1} \upp 1 \biguplus Q \biguplus M \uppeq 3\\
&\ll& U_{\leq e-1} \upp 1 \biguplus X\\
&\ll& \big(U_{\leq e-1} \upp 1 \biguplus H \big) \biguplus U_{\leq \epsilon -1} \upp 2 \biguplus Y[+a_2]
=N \setminus \biguplus_{j=\epsilon}^e U_j \upp 2.
\end{eqnarray*}
(See Fig.\ 15.)

\textit{Case 2}.
Suppose $\rho(F) \subset \biguplus_{j=\epsilon}^{e+a_2} U_j \upp 2$ and
$\# F + \# M\upp 2  \setminus \biguplus_{j=\epsilon} ^e U_j \upp 2 > \# U_{\leq e+a_2} \upp 2$.
We claim

\begin{lemma}
$f=\delta_1 x_1^{\alpha_1} x_2^{\alpha_2},$
that is, $\alpha_3= \cdots = \alpha_n=0$.
\end{lemma}

\begin{proof}
Suppose $f \ne \delta_1 x_1^{\alpha_1} x_2^{\alpha_2}$.
Let $g=\delta_1 x_1^{\alpha_1} x_2^{\alpha_2+\alpha_3+\cdots + \alpha_n}.$
Then $g >_\dlex f$ is admissible over $U$ by the definition of the admissibility.
Also,
$$\#M < \#\{ h \in U : h \leq_\dlex g\} = \# \big(U_{\leq e-1} \upp 1 \biguplus [g,\delta_1 x_n^{e-b_1}]\big)
\biguplus U_{\leq e} \upp 2 \biguplus U_{\leq e} \uppeq 3.$$

\begin{center}
\input{Fig15}
\end{center}
\bigskip

Since $\rho ([g,\delta_1 x_n^{e-b_1}])= \biguplus_{i=\epsilon} ^{e+a_2} U_i \upp 2$ and
$M \uppeq 3 \supset U_{\leq e} \uppeq 3$,
\begin{eqnarray*}
\# F + \# \big(M \upp 2 \setminus \biguplus_{j=\epsilon} ^e U_j \upp 2 \big)
&=& \big(\#M - \# U_{\leq e-1} \upp 1 -\# M \uppeq 3\big) - \# \biguplus_{j=\epsilon}^e U_j \upp 2\\
&< & \# [g,\delta_1 x_n^{e-b_1}] + \# U_{\leq e} \upp 2 - \# \biguplus_{j=\epsilon} ^e U_j \upp 2\\
&=& \# U_{\leq e+a_2} \upp 2,
\end{eqnarray*}
which contradicts the assumption of Case 2.
Thus $f= \delta_1 x_1^{\alpha_1} x_2^{\alpha_2}$.
\end{proof}

Note that the above lemma says $\rho(f) = \delta_2 x_2^{\epsilon -b_2}$.
In particular, $\rho([f,\delta_1x_n^{e-b_1}])= \bigcup_{j=\epsilon}^{e+a_2} U_j \upp 2$.
Let
$$H = \biguplus_{j=\epsilon}^{e+a_2} U_j \upp 2 \setminus \rho (F).$$
(See Fig.\ 16).

\begin{center}
\unitlength 0.1in
\begin{picture}( 24.0000, 17.0000)(  4.0000,-20.1500)
%
\special{pn 8}%
\special{pa 2000 600}%
\special{pa 1200 600}%
\special{fp}%
\special{pa 1200 800}%
\special{pa 1400 800}%
\special{fp}%
\special{pa 1400 800}%
\special{pa 1400 1000}%
\special{fp}%
\special{pa 1400 1000}%
\special{pa 2000 1000}%
\special{fp}%
\put(16.0000,-12.0000){\makebox(0,0){$H$}}%
\put(16.0000,-8.0000){\makebox(0,0){$\rho(F)$}}%
%
\special{pn 4}%
\special{pa 1680 1200}%
\special{pa 1480 1400}%
\special{dt 0.027}%
\special{pa 1560 1200}%
\special{pa 1360 1400}%
\special{dt 0.027}%
\special{pa 1440 1200}%
\special{pa 1240 1400}%
\special{dt 0.027}%
\special{pa 1320 1200}%
\special{pa 1210 1310}%
\special{dt 0.027}%
\special{pa 1800 1200}%
\special{pa 1600 1400}%
\special{dt 0.027}%
\special{pa 1920 1200}%
\special{pa 1720 1400}%
\special{dt 0.027}%
\special{pa 1990 1250}%
\special{pa 1840 1400}%
\special{dt 0.027}%
\special{pa 1990 1370}%
\special{pa 1960 1400}%
\special{dt 0.027}%
%
\special{pn 4}%
\special{pa 1390 810}%
\special{pa 1210 990}%
\special{dt 0.027}%
\special{pa 1400 920}%
\special{pa 1320 1000}%
\special{dt 0.027}%
\special{pa 1280 800}%
\special{pa 1210 870}%
\special{dt 0.027}%
%
\special{pn 8}%
\special{pa 1200 1200}%
\special{pa 800 1200}%
\special{fp}%
\special{pa 800 1200}%
\special{pa 800 1000}%
\special{fp}%
\special{pa 800 1000}%
\special{pa 1200 1000}%
\special{fp}%
\put(10.0000,-11.0000){\makebox(0,0){$F$}}%
\put(16.0000,-21.0000){\makebox(0,0){$H$}}%
%
\special{pn 13}%
\special{pa 400 600}%
\special{pa 400 2000}%
\special{fp}%
\special{pa 400 2000}%
\special{pa 1200 2000}%
\special{fp}%
\special{pa 1200 2000}%
\special{pa 1200 600}%
\special{fp}%
\special{pa 2000 1600}%
\special{pa 1200 1600}%
\special{fp}%
\special{pa 2000 1600}%
\special{pa 2000 600}%
\special{fp}%
\special{pa 2000 1400}%
\special{pa 2800 1400}%
\special{fp}%
\special{pa 2800 1400}%
\special{pa 2800 600}%
\special{fp}%
%
\special{pn 8}%
\special{pa 2800 600}%
\special{pa 400 600}%
\special{dt 0.045}%
\special{pa 400 800}%
\special{pa 2800 800}%
\special{dt 0.045}%
\special{pa 2800 1000}%
\special{pa 400 1000}%
\special{dt 0.045}%
\special{pa 400 1200}%
\special{pa 2800 1200}%
\special{dt 0.045}%
\special{pa 2000 1400}%
\special{pa 400 1400}%
\special{dt 0.045}%
\special{pa 400 1600}%
\special{pa 1200 1600}%
\special{dt 0.045}%
\special{pa 1200 1800}%
\special{pa 400 1800}%
\special{dt 0.045}%
\put(16.0000,-4.0000){\makebox(0,0){Figure 16}}%
%
\special{pn 8}%
\special{pa 1200 1400}%
\special{pa 2000 1400}%
\special{fp}%
%
\special{pn 4}%
\special{pa 1680 1000}%
\special{pa 1480 1200}%
\special{dt 0.027}%
\special{pa 1560 1000}%
\special{pa 1360 1200}%
\special{dt 0.027}%
\special{pa 1440 1000}%
\special{pa 1240 1200}%
\special{dt 0.027}%
\special{pa 1320 1000}%
\special{pa 1210 1110}%
\special{dt 0.027}%
\special{pa 1800 1000}%
\special{pa 1600 1200}%
\special{dt 0.027}%
\special{pa 1920 1000}%
\special{pa 1720 1200}%
\special{dt 0.027}%
\special{pa 1990 1050}%
\special{pa 1840 1200}%
\special{dt 0.027}%
\special{pa 1990 1170}%
\special{pa 1960 1200}%
\special{dt 0.027}%
\end{picture}%
\end{center}
\bigskip

\noindent
Since $\rho(F)$ is an upper rev-lex set of degree $e+a_2$,
$H$ is a lower lex set of degree $\epsilon$.
Also, since $\# F + \# M \upp 2 > \# U_{\leq e+a_2} \upp 2$,
$\rho(F) \cup M \upp 2 \supset U \uppeq 2 _{\leq e+a_2}$ by (A2).
Thus $M \upp 2 \supset H$.

Let $R$ be the super rev-lex set in $U \upp 2$
with $\#R = \# M \upp 2 \setminus H$.
Since $M \upp 2 \setminus H$
is rev-lex,
by Corollary \ref{ERV} we have
\begin{eqnarray}
\label{c3}
R \gg  M \upp 2 \setminus H.
\end{eqnarray}
Then since $\# R \leq \# M \upp2$,
$$R \biguplus M \uppeq 3 \subset U \uppeq 2$$
is a ladder set.
(See the third picture in Fig.\ 17.)

Let $Y \subset U \uppeq 2$ be the extremal set in $U \uppeq 2$
with $\# Y = \#R \biguplus M \uppeq 3$.
We claim that
$$N= \{ h \in U \upp 1 : h \leq_\dlex f\} \biguplus Y$$
satisfies the desired conditions.
Indeed,
since $\rho(F) \biguplus H = \biguplus_{j=\epsilon} ^{e+a_2} U_j \upp 2 = \rho([f,\delta_2 x_n^{e-b_1}])$,
by (\ref{c3}), we have
\begin{eqnarray*}
M &=& \big(U_{\leq e-1} \upp 1 \biguplus F\big) \biguplus M \upp 2 \biguplus M \uppeq 3\\
&=& \big(U_{\leq e-1} \upp 1 \biguplus F \biguplus H \big) \biguplus 
\big(M \upp 2 \setminus H\big)
\biguplus M \uppeq 3\\
&\ll& \big(U_{\leq e-1} \upp 1 \biguplus [f,\delta_1 x_n^{e-b_1}]\big) \biguplus 
R
\biguplus M \uppeq 3\\
&\ll& \{ h \in U \upp 1 : h \leq _{\dlex} f\}\biguplus Y =N.
\end{eqnarray*}
(See Fig.\ 17.)
It remains to prove that $N$ is a ladder set.
Since 
$$\#Y = \# M - \# \{ h \in U \upp 1 : h \leq _\dlex f\} \geq \# U_{\leq e} \uppeq 2$$
by the choice of $f$,
we have $Y \supset U_{\leq e}\uppeq 2$ by Lemma \ref{X4}.
This fact guarantees that $N$ is a ladder set.
\begin{center}
\input{Fig17}
\end{center}
\bigskip
\bigskip

\textit{Case 3}.
Suppose $\rho(F) \not \subset \biguplus_{j=\epsilon} ^{e+a_2} U_j \upp 2$.
Then $\rho(F)$ properly contains $\biguplus_{j=\epsilon} ^{e+a_2} U_j \upp 2$
since $\rho(F)$ is an upper rev-lex set of degree $e+a_2$.
In particular, $F$ properly contains $[f,\delta_1 x_n^{e-b_1}]$.
We claim

\begin{lemma}
\label{6.7}
$f=\delta_1 x_1^{\alpha_1} x_2^{\alpha_2}$
and $\alpha_2 \ne 0$.
\end{lemma}

\begin{proof}
If $\alpha_k \ne 0$ for some $k \geq 3$ then $\delta_1 x_1^{\alpha_1} x_2^{\alpha_2+ \cdots + \alpha_n}>_\dlex f$ is admissible over $U$.
Then by the choice of $f$,
$F \subset [\delta_1 x_1^{\alpha_1} x_2^{\alpha_2+ \cdots + \alpha_n}, \delta_1 x_n^{e-b_1}]$
and 
$$\rho(F) \subset \rho\big( [\delta_1 x_1^{\alpha_1} x_2^{\alpha_2+ \cdots + \alpha_n}, \delta_1 x_n^{e-b_1}] \big)
= \biguplus_{j=\epsilon} ^{e+a_2} U_j \upp 2,$$
a contradiction.
Also, if $\alpha_2=0$ then $\epsilon = \deg \rho(f)=0$ which 
implies $\rho(F) \subset \rho(U_e \upp 1)=U_{\leq e+a_2} \upp 2 =\biguplus_{j=\epsilon}^{e+a_2} U_j \upp 2$,
a contradiction.
\end{proof}

Recall $\epsilon = \deg \rho(f)$.
Thus $\alpha_2 = \epsilon -b_2$.
Let
$$H=\{h \in F: h>_\lex f\}$$
and
$$g = \max _{>_\lex} H.$$
By the choice of $f$,
$H$ contains no admissible monomials over $U$.
By Lemma \ref{6.7},
$\rho(F \setminus H) = \biguplus_{j=\epsilon}^{e+a_2} U_j \upp 2$.
Hence $H \ne \emptyset$ by the assumption of Case 3.
Since $\delta_1 x_1^{\alpha_1+1} x_2^{\alpha_2-1}$
is admissible over $U$,
$$ \rho(H) \subset \rho \big([\delta_1 x_1^{\alpha_1+1} x_2^{\alpha_2-1},\delta_1 x_1^{\alpha_1} x_2^{\alpha_2})\big)= U_{\epsilon -1} \upp 2$$
is rev-lex.
Also, $\epsilon -1 > b_2$ since $U_{b_2} \upp 2=\{\delta_2\}$ and $H \ne \emptyset$.

If $t=2$ then any monomial $h\in U_e \upp 1$ with $h>_\lex f$ is admissible,
which implies $H= \emptyset$.
Thus we may assume $t\geq 3$.

To prove the statement,
it is enough to prove that
there exists $Z \subset U \uppeq 3$ such that
\begin{eqnarray}
\label{AIM}
Z \gg  H \biguplus M \uppeq 3.
\end{eqnarray}
Indeed, if such a $Z$ exists then $N=(M \upp 1 \setminus H) \biguplus M \upp 2 \biguplus Z$ satisfies the desired conditions.
Recall that $\epsilon \leq e+1$ by Definition \ref{admdef}.
\medskip

\noindent
\textbf{(subcase 3-1)}
Suppose $a_3 \geq e-(\epsilon -1)$.

Let $d=e-(\epsilon -1)$.
Then
$$U'=U \upp 2 \biguplus \left(\biguplus_{i=3}^t U \upp i [+d] \right)$$
is universal lex.
Recall $\rho(H) \subset U_{\epsilon-1}\upp 2$.
Let
$$Y= \rho(H) \biguplus U_{\leq \epsilon -2} \upp 2 \biguplus M \uppeq 3 [+d].$$
(See Fig.\ 18.)
Then $Y$ is a ladder set since $M \uppeq 3 \supset U _{\leq \epsilon -1 +d} \uppeq 3= U_{\leq e} \uppeq 3$.
Also, $U_{\leq \epsilon -2}\upp 2 \ne \emptyset$ since $\epsilon-1 >b_2$.

\begin{center}
\input{Fig18}
\end{center}
\bigskip
\medskip

Let $\mu \in U_{\leq \epsilon -1} \upp 2$ be the largest admissible monomial in $U _{\leq \epsilon -1} \upp 2$ over $U'$
with respect to $>_\dlex$
satisfying $\#\{h \in U': h \leq_\dlex \mu\} \leq \# Y$.
Then since we assume that Proposition \ref{key} is true for $U'$,
there exists $Z \subset {U'} \uppeq 3$ such that
$$Y \ll \{h \in U \upp 2: h\leq _\dlex \mu \} \biguplus Z.$$
To prove (\ref{AIM}),
it is enough to prove $\{h \in U \upp 2: h \leq_\dlex \mu\} = U_{\leq \epsilon-2} \upp 2$,
in other words,

\begin{lemma}
$\mu = \delta_2 x_2^{\epsilon-2-b_2}$.
\end{lemma}

\begin{proof}
Recall that $U_{\leq \epsilon -2} \upp 2 \ne \emptyset$.
It is enough to prove that $\deg \mu \ne \epsilon -1$.
Suppose contrary that $\deg \mu = \epsilon -1$.
Let $\mu ' \in U \upp 1_e$ be a monomial such that
$\rho(\mu')=\mu$.
Then $\mu'$ is admissible over $U$ by Lemma \ref{X1}.
Also
$$ \# Y - \# U_{\leq \epsilon -2} \upp 2 \geq \#[\mu,\delta_2 x_n^{\epsilon -1-b_2}] + \# {U'} \uppeq 3 _{\leq \epsilon -1}=\#[\mu,\delta_2 x_n^{\epsilon -1-b_2}] + \# {U} \uppeq 3 _{\leq e}.$$
Since $\#M \uppeq 3 + \#H = \#Y - \# U_{\leq \epsilon -2} \upp 2$
and since $\rho([\mu',f))=[\mu,\delta_2 x_n^{\epsilon-1-b_2}]$,
we have
\begin{eqnarray*}
\#M&=&\# (M \setminus H) \biguplus M \upp 2 \biguplus H \biguplus M \uppeq 3\\
&\geq & \# (M \setminus H) \biguplus U_{\leq e} \upp 2 \biguplus [\mu,\delta_2 x_n^{\epsilon-1-b_2}] \biguplus U_{\leq e} \uppeq 3\\
&\geq & \#[\mu',f) \biguplus  (M \setminus H) \biguplus U_{\leq e} \uppeq 2 = \#  \{ h \in U: h \leq_\dlex \mu'\},
\end{eqnarray*}
which contradicts the choice of $f$ since $\mu'>_\lex g >_\lex f$
and $\mu'$ is admissible over $U$.
\end{proof}

\noindent
\textbf{(subcase 3-2)}
Suppose $a_3 < e-(\epsilon -1)$.
We consider
$$X= x_2^{e-(\epsilon -1)} \rho(H).$$
(See Fig.\ 19.)
\begin{center}
\input{Fig19}
\end{center}
\bigskip

\noindent
Let
$$Y=\{h \in U \upp 2: h \leq _\dlex x_2^{e-(\epsilon-1)} \rho(g) \} \biguplus M \uppeq 3$$
(see Fig.\ 20) and let
$$g' = \max_{>_\dlex} \big(Y \upp 2 \setminus X \big).$$
Since $e-(\epsilon-1)>a_3$, $e-(\epsilon -1) \geq 1$. Thus
\begin{eqnarray*}
g'=
\delta_2 x_2^{e-(\epsilon-1)-1} x_3^{\epsilon -b_2}
\end{eqnarray*}
and
$$Y \upp 2 = X \biguplus \{ h \in U \upp 2: h \leq _\dlex g'\}.$$
Since $a_3< e-(\epsilon-1)$,
$\deg \rho(\delta_2 x_2^{e-(\epsilon-1)-1} x_3^{\epsilon -b_2})=\epsilon +a_3 \leq e$.
Thus $g'$ is admissible over $U \uppeq 2$.

Let $\mu$ be the largest admissible monomial in $U _{\leq e} \upp 2$ over $U \uppeq 2$ with respect to $>_\dlex$
with $\#\{h \in U \uppeq 2: h \leq_\dlex \mu\} \leq \#Y$.
Since Lemma \ref{X1} says that $X$ contains no admissible monomials over $U \uppeq 2$,
$$\mu \geq_\dlex g'  \mbox{ and } \mu \not \in X.$$
Since we assume that Proposition \ref{key} is true for $U \uppeq 2$,
there exists $Z \subset U \uppeq 3$ such that
$$W= \{ h \in U \upp 2: h \leq _\dlex \mu \} \biguplus Z$$
is a ladder set and 
$$W \gg Y$$
(See Fig.\ 20.)
\begin{center}
\input{Fig20}
\end{center}
\bigskip

We claim
\begin{lemma}
$\mu=g'$.
\end{lemma}

\begin{proof}
Suppose contrary that $\mu \ne g'$.
Then $\mu >_\dlex g'$ and 
$$W= \big[\mu ,x_2^{e-(\epsilon-1)} \rho(g)\big) \biguplus Y \upp 2 \biguplus Z.$$
Then there exists $\mu' \in U \upp 1 _e$ such that
$$x_2^{e-(\epsilon -1)} \rho(\mu') = \mu.$$
By Lemma \ref{X1},
$\mu'$ is admissible over $U$ and $\mu' >_\lex g>_\lex f$.
Observe that
$$\# M \uppeq 3 +\#H = \# Z \biguplus \big[\mu,x_2^{e-(\epsilon-1)} \rho(g)\big) \biguplus X = \# Z + \# [\mu',f)$$
by the construction of $Y$
and $Z$.
Since $Z \supset U_{\leq e} \uppeq 3$,
\begin{eqnarray*}
\#M &\geq& \# \big(M\upp 1\setminus H \big)\biguplus H \biguplus U_{\leq e} \upp 2 \biguplus M \uppeq 3\\
&=& \# \big(M\upp 1\setminus H \big) \biguplus U_{\leq e} \upp 2 \biguplus Z \biguplus [\mu',f)\\
&\geq& \# \big(M\upp 1\setminus H \big)\biguplus [\mu' ,f) \biguplus U_{\leq e} \upp 2 \biguplus U_{\leq e} \uppeq 3 \\
&=& \# \{ h \in U: h \leq _{\dlex} \mu'\}.
\end{eqnarray*}
Since $\mu'$ is admissible over $U$,
this contradicts the choice of $f$.
\end{proof}

Now
$$W= \{ h \in U \upp 2: h \leq _\dlex g'\} \biguplus Z$$
and since $W \gg Y$ and $Y=X \biguplus \{h \in U \upp 2: h \leq_\dlex g'\} \biguplus M \uppeq 3$,
we have
$$m(Z) \succeq m \big(X \biguplus M\uppeq 3 \big) = m \big(H \biguplus M \uppeq 3\big),$$
which proves (\ref{AIM}).

\subsection{Proof of Proposition \ref{key} when $f=\delta_1 x_1^{e-b_1-1}$.}\

In this subsection, we prove Proposition \ref{key} when $f=\delta_1 x_1^{e-b_1-1}$.
Let $F = M \upp 1_e$.
If $F=\emptyset$ then there is nothing to prove.
Thus we may assume $F \ne \emptyset$.
Then $M \supset U \uppeq 2_{\leq e}$ since $M$ is a ladder set.
\medskip

\textit{Case 1}.
Suppose $a_2=0$.
Since $\delta_1 x_2^{e-b_1}$ is admissible over $U$,
$\delta_1x_2^{e-b_1} \not \in F$. Indeed, if $\delta_1 x_2^{e-b_1} \in F$ then
$M \supset \{h \in U : h \leq_\dlex \delta_1 x_2^{e-b_1}\}$,
which contradicts the choice of $f$.
Thus 
$$F \subset [\delta_1 x_2^{e-b_1},\delta_1 x_n^{e-b_1}].$$
and
$$\rho(F) \subset \rho \big([\delta_1 x_2^{e-b_1},\delta_1 x_n^{e-b_1}] \big) = U_e \upp 2.$$
Consider
$$X= \rho(F) \biguplus U_{\leq e-1} \upp 2 \biguplus M\uppeq 3 \subset U \uppeq 2$$
and let $Y \subset U \uppeq 2$ be the extremal set with $\#Y = \#X$.
Since $X$ is a ladder set in $U \uppeq 2$, 
by the induction hypothesis we have
$$Y \gg  X.$$
We claim
\begin{lemma}
$Y \upp 2 = U_{\leq e-1} \upp 2.$
\end{lemma}

\begin{proof}
Suppose contrary that $Y \upp 2 \ne U \upp 2 _{\leq e-1}$.
Let $g=\delta_2 \bar g$ be the largest admissible monomial in $Y \upp 2_{\leq e}$ over $U \uppeq 2$ with respect to $>_\dlex$.
Since $X \supset U_{\leq e-1} \uppeq 2$, we have $Y \supset U_{\leq e-1} \upp 2$ by Lemma \ref{X4}.
Thus $\deg g=e$ and $Y \supset U_{\leq e} \uppeq 3$.

Let $g' =\delta_1 \bar g$.
Since $g=\delta_2 \bar g$ is admissible over $U \uppeq 2$ and since $\rho(g')=g$,
$g'$ is admissible over $U$ by Lemma \ref{X1}.
Observe $\#Y=\#X \leq \#F +\# M \uppeq 2-\# U_e \upp 2$.
Then
\begin{eqnarray*}
\# M &\geq & \# U_{\leq e-1} \upp 1 +\# U_{e} \upp 2 + \# Y\\
&\geq & \# U_{\leq e-1} \upp 1 + \# U_e \upp 2 + \# \{ h \in U \uppeq 2 : h \leq_\dlex g\}\\
&= & \# U_{\leq e-1} \upp 1 + \# U_e \upp 2 + \# U_{\leq e-1} \upp 2 \biguplus [g,\delta_2x_n^{e-b_1}]\biguplus U_{\leq e} \uppeq 3\\
&=& \# U_{\leq e-1} \upp 1 + \# U_{\leq e} \uppeq 2 + \# [g',\delta_1 x_n^{e-b_1}]\\
&=& \# \{ h \in U: h \leq _\dlex g'\}
\end{eqnarray*}
which contradicts the choice of $f$.
Hence $Y \upp 2 = U_{\leq e-1} \upp 2$.
\end{proof}

Then, since $Y \gg  X$, we have
\begin{eqnarray}
\label{d1}
Y \uppeq 3 \gg  F \biguplus M \uppeq 3.
\end{eqnarray}
Let
$$N= U_{\leq e-1} \upp 1 \biguplus M \upp 2 \biguplus Y \uppeq 3.$$
Then $N$ is a ladder set since $\# Y \uppeq 3 \geq \# M \uppeq 3$.
Also $N \gg M$ by (\ref{d1}).
Thus $N$ satisfies the desired conditions.
\medskip

\textit{Case 2}.
Suppose $a_2 > 0$.
Since $\deg f \ne e$,
we have $\#M < \# U_{\leq e} \upp 1$ by Lemma \ref{X3}.
Hence
\begin{eqnarray}
\label{bound}
\#F + \# M \upp 2 \leq \#M -\#U_{\leq e -1} \upp 1 < \# U_e \upp 1 \leq \# U_{\leq e+a_2} \upp 2.
\end{eqnarray}
Then, by $(A2)$ and $(A3)$,
we may assume that
$\rho(F) \cap M \upp 2 = \emptyset$, $t \geq 3$ and there exists a $d \geq e$ such that
$M \upp 2 = U_{\leq d} \upp 2$ and
$M \upp 3_{d+1} \ne U_{d+1} \upp 3$.

Let
$$A=\big\{ \delta_2 u \in \rho(F)_{e+a_2} :  x_2^{(e+a_2)-(d+1)} \mbox{ divides $u$  and 
$\delta_2 u/x_2^{(e+a_2)-(d+1)} \not \in \rho(F)_{d+1}$} \big\},$$
$$E= x_2^{-(e+a_2)+(d+1)} A \subset U_{d+1} \upp 2,$$
and
$$B= \rho(F)_{e+a_2} \setminus A \subset U_{e+a_2} \upp 2.$$
(See the second picture in Fig.\ 21.)

\textbf{(subcase 2-1)}
Suppose $\#B + \#M \uppeq 3 < \# U_{e+a_2} \upp 2$.
Consider
$$U' = U \upp 2 \biguplus \left(\biguplus_{i=3}^t U \upp i [-a_2]\right).$$
Since $M \uppeq 3 [-a_2] \supset {U'}\uppeq 3_{\leq e+a_2}$,
by Corollary \ref{X10} and the induction hypothesis, there exists the extremal set $Q \subset {U'} \uppeq 3$
such that
\begin{eqnarray}
\label{e4}
Q \gg B \biguplus M \uppeq 3.
\end{eqnarray}

Let $P$ be the super rev-lex set in $U \upp 2$ with
$\# P = \# M \upp 2 +  \# \rho(F) \setminus B$.
Then since $\rho(F)_{\leq e+a_2-1} \biguplus E$ is rev-lex,
Corollary \ref{ERV} shows
\begin{eqnarray}
\label{e5}
&&
m \big(M \upp 2 \biguplus \rho(F) \setminus B\big)
=m \big(M\upp 2 \big) + m \big(\rho(F)_{\leq e+a_2-1} \biguplus E \big)
\preceq m(P).
\end{eqnarray}
(See the second step in Fig.\ 21.)
We claim that
$$N=U_{\leq e-1} \upp 1 \biguplus P \biguplus Q[+a_2] \subset U$$
satisfies the desired conditions.
Indeed,
by (\ref{e4}) and (\ref{e5}),
\begin{eqnarray*}
m(N) \succeq  m \big(U_{\leq e-1} \upp 1 \biguplus M \upp 2 \biguplus \big(\rho(F) \setminus B\big) \biguplus \big(B \biguplus M \uppeq 3\big)\big) =m(M).
\end{eqnarray*}
(See Fig.\ 21).
It remains to prove that $N$ is a ladder set.
If $\rho(F) \setminus B = \emptyset$ then $P=M \upp 2$,
and therefore $N$ is a ladder set since $\# Q \geq \# M \uppeq 3$.
Suppose $\rho(F) \setminus B \ne \emptyset$.
Recall that $\rho(F) \cap M \upp 2 = \emptyset$.
Since 
$$\# U_{\leq e} \upp 2 \leq \#M \upp 2 \leq \# P = \# \rho(F)_{\leq e+a_2-1} \biguplus E \biguplus M \upp 2 \leq \#U_{\leq e+a_2-1} \upp 2,$$
we have
$$U\upp 2_{\leq e} \subset P \subset U_{\leq e+a_2-1} \upp 2.$$
Then by Lemma \ref{X4} what we must prove is
$$\# Q \geq \# U_{\leq e+a_2-1} \uppeq 3.$$

\begin{center}
\input{Fig21}
\end{center}
\bigskip
\bigskip

Since $\# S_k \upp i = \sum_{j=i} ^n \# S_{k-1} \upp j$ for all $i >0$ and $k >0$,
we have
\begin{eqnarray}
\label{Q1}
\# U_k \upp 3 \geq  \sum_{j=3}^t \# U_{k-1} \upp j=
\# U_{k-1} \uppeq 3
\end{eqnarray}
for all $k >0$.
Since $\rho(F) \setminus B \ne \emptyset$,
$\# B = \# \rho(F)_{e+a_2} \setminus A \geq \# U_{e+a_2} \upp 2-\# U_{d+1} \upp 2$.
Thus
$$\#B \geq \# U_{e+a_2} \upp 2 - \# U_{d+1} \upp 2
=\# \biguplus _{j=d+2} ^{e+a_2} U_{j+a_3} \upp 3 \geq \# \biguplus_{j=d+2}^{e+a_2} U_j \upp 3 \geq \sum_{j=d+1} ^{e+a_2 -1} \# U_j \uppeq 3,$$
(we use (\ref{Q1}) for the last step)
and therefore
$$\#Q = \#M \uppeq 3 + \#B \geq \# U_{\leq d} \uppeq 3 + \sum_{d+1} ^{e+a_2-1} U_j \uppeq 3
\geq \# U_{\leq e+a_2-1} \uppeq 3
$$
as desired.
\medskip

\textbf{(subcase 2-2)}
Suppose $\#B+ \# M \uppeq 3 \geq \# U_{e+a_2} \upp 2$.
We first prove 
\begin{lemma}
$\rho(F) \not \supset \biguplus_{j=d+2}^{e+a_2} U_j \upp 2$.
\end{lemma}

\begin{proof}
Suppose contrary that $\rho(F) \supset\biguplus_{j=d+2}^{e+a_2} U_j \upp 2$.
Then 
$$\# \rho(F) \setminus B =\# \big(\rho(F)\setminus(A \biguplus B) \big) \biguplus E = \# \biguplus_{j=d+1}^{e+a_2-1} U_j \upp 2$$
by the choice of $E$.
Then
$$\# \big(\rho(F) \setminus B \big)\biguplus M \upp 2 \geq \# U_{\leq e+a_2-1} \upp 2$$
and
$$\#M = \#U_{\leq e-1} \upp 1 + \# \rho(F) + \# M \uppeq 2
\geq \# U_{\leq e-1} \upp 1 + \# U_{\leq e+a_2-1} \upp 2 + \# U_{e+a_2} \upp 2= \# U_{\leq e} \upp 1,$$
where we use the assumption $\#B + \#M \uppeq 3 \geq \# U_{e+a_2} \upp 2$ for the second step.
However, since $\deg f <e$ and $a_2>0$, Lemma \ref{X3} says
$$\# M < \# U_{\leq e} \upp 1,$$
a contradiction.
\end{proof}

The above lemma says that $e+a_2 \geq d+2$ and $\rho(F)_{d+1}=\emptyset$.
Thus $B$ does not contain any monomial $\delta_2 u$ such that
$u$ is divisible by $x_2^{(e+a_2)-(d+1)}$.
Hence
$$\rho(B) \subset \biguplus _{j=d+2+a_3} ^{e+a_2+a_3} U_j \upp 3.$$
Since $M_{d+1} \upp 3 \ne U_{d+1} \upp 3$,
by Lemma \ref{X7},
$$\# M \uppeq 3 < \# U_{\leq d+2} \upp 3.$$
We claim

\begin{lemma}
$a_3=0.$
\end{lemma}

\begin{proof}
If $a_3>0$ then
$$\#B + \# M \uppeq 3 <
\# \biguplus_{j=d+2+a_3} ^{e+a_2+a_3} U_j \upp 3 + \# U_{\leq d+2} \upp 3
\leq U_{\leq e+a_2+a_3} \upp 3 = \# U_{e+a_2} \upp 2,$$
which contradicts the assumption of (subcase 2-2).
\end{proof}

Let
$$H = \{ h \in U_{d+1} \uppeq 3: h \not \in M \uppeq 3\}.$$
(See Fig.\ 22.)
\bigskip

\begin{center}
\unitlength 0.1in
\begin{picture}( 24.0000, 20.0000)( 12.0000,-25.1500)
%
\special{pn 8}%
\special{pa 1200 800}%
\special{pa 3600 800}%
\special{dt 0.045}%
\special{pa 3600 1000}%
\special{pa 1200 1000}%
\special{dt 0.045}%
\special{pa 1200 1200}%
\special{pa 3600 1200}%
\special{dt 0.045}%
\special{pa 3600 1400}%
\special{pa 1200 1400}%
\special{dt 0.045}%
\special{pa 1200 1600}%
\special{pa 3600 1600}%
\special{dt 0.045}%
\special{pa 3600 1800}%
\special{pa 1200 1800}%
\special{dt 0.045}%
\special{pa 1200 2000}%
\special{pa 2000 2000}%
\special{dt 0.045}%
\special{pa 2000 2200}%
\special{pa 1200 2200}%
\special{dt 0.045}%
%
\special{pn 13}%
\special{pa 1200 800}%
\special{pa 1200 2400}%
\special{fp}%
\special{pa 1200 2400}%
\special{pa 2000 2400}%
\special{fp}%
\special{pa 2000 2400}%
\special{pa 2000 800}%
\special{fp}%
%
\special{pn 13}%
\special{pa 2000 1800}%
\special{pa 2800 1800}%
\special{fp}%
\special{pa 2800 1800}%
\special{pa 2800 800}%
\special{fp}%
%
\special{pn 13}%
\special{pa 2800 1800}%
\special{pa 3600 1800}%
\special{fp}%
\special{pa 3600 1800}%
\special{pa 3600 800}%
\special{fp}%
%
\special{pn 8}%
\special{pa 1200 1600}%
\special{pa 1800 1600}%
\special{fp}%
\special{pa 1800 1600}%
\special{pa 1800 1400}%
\special{fp}%
\special{pa 1800 1400}%
\special{pa 3200 1400}%
\special{fp}%
\special{pa 3200 1400}%
\special{pa 3200 1200}%
\special{fp}%
\special{pa 3200 1200}%
\special{pa 3600 1200}%
\special{fp}%
%
\special{pn 8}%
\special{pa 3200 1200}%
\special{pa 2800 1200}%
\special{fp}%
\put(30.0000,-13.0000){\makebox(0,0){$H$}}%
%
\special{pn 4}%
\special{pa 3160 1400}%
\special{pa 2960 1600}%
\special{dt 0.027}%
\special{pa 3040 1400}%
\special{pa 2840 1600}%
\special{dt 0.027}%
\special{pa 2920 1400}%
\special{pa 2800 1520}%
\special{dt 0.027}%
\special{pa 3280 1400}%
\special{pa 3080 1600}%
\special{dt 0.027}%
\special{pa 3400 1400}%
\special{pa 3200 1600}%
\special{dt 0.027}%
\special{pa 3520 1400}%
\special{pa 3320 1600}%
\special{dt 0.027}%
\special{pa 3590 1450}%
\special{pa 3440 1600}%
\special{dt 0.027}%
\special{pa 3590 1570}%
\special{pa 3560 1600}%
\special{dt 0.027}%
%
\special{pn 4}%
\special{pa 3480 1200}%
\special{pa 3280 1400}%
\special{dt 0.027}%
\special{pa 3590 1210}%
\special{pa 3400 1400}%
\special{dt 0.027}%
\special{pa 3590 1330}%
\special{pa 3520 1400}%
\special{dt 0.027}%
\special{pa 3360 1200}%
\special{pa 3200 1360}%
\special{dt 0.027}%
\special{pa 3240 1200}%
\special{pa 3200 1240}%
\special{dt 0.027}%
%
\special{pn 4}%
\special{pa 1990 1490}%
\special{pa 1880 1600}%
\special{dt 0.027}%
\special{pa 1960 1400}%
\special{pa 1800 1560}%
\special{dt 0.027}%
\special{pa 1840 1400}%
\special{pa 1800 1440}%
\special{dt 0.027}%
%
\special{pn 4}%
\special{pa 3160 1600}%
\special{pa 2960 1800}%
\special{dt 0.027}%
\special{pa 3040 1600}%
\special{pa 2840 1800}%
\special{dt 0.027}%
\special{pa 2920 1600}%
\special{pa 2800 1720}%
\special{dt 0.027}%
\special{pa 3280 1600}%
\special{pa 3080 1800}%
\special{dt 0.027}%
\special{pa 3400 1600}%
\special{pa 3200 1800}%
\special{dt 0.027}%
\special{pa 3520 1600}%
\special{pa 3320 1800}%
\special{dt 0.027}%
\special{pa 3590 1650}%
\special{pa 3440 1800}%
\special{dt 0.027}%
\special{pa 3590 1770}%
\special{pa 3560 1800}%
\special{dt 0.027}%
%
\special{pn 4}%
\special{pa 2360 1600}%
\special{pa 2160 1800}%
\special{dt 0.027}%
\special{pa 2240 1600}%
\special{pa 2040 1800}%
\special{dt 0.027}%
\special{pa 2120 1600}%
\special{pa 2000 1720}%
\special{dt 0.027}%
\special{pa 2480 1600}%
\special{pa 2280 1800}%
\special{dt 0.027}%
\special{pa 2600 1600}%
\special{pa 2400 1800}%
\special{dt 0.027}%
\special{pa 2720 1600}%
\special{pa 2520 1800}%
\special{dt 0.027}%
\special{pa 2790 1650}%
\special{pa 2640 1800}%
\special{dt 0.027}%
\special{pa 2790 1770}%
\special{pa 2760 1800}%
\special{dt 0.027}%
%
\special{pn 4}%
\special{pa 2360 1400}%
\special{pa 2160 1600}%
\special{dt 0.027}%
\special{pa 2240 1400}%
\special{pa 2040 1600}%
\special{dt 0.027}%
\special{pa 2120 1400}%
\special{pa 2000 1520}%
\special{dt 0.027}%
\special{pa 2480 1400}%
\special{pa 2280 1600}%
\special{dt 0.027}%
\special{pa 2600 1400}%
\special{pa 2400 1600}%
\special{dt 0.027}%
\special{pa 2720 1400}%
\special{pa 2520 1600}%
\special{dt 0.027}%
\special{pa 2790 1450}%
\special{pa 2640 1600}%
\special{dt 0.027}%
\special{pa 2790 1570}%
\special{pa 2760 1600}%
\special{dt 0.027}%
%
\special{pn 4}%
\special{pa 1560 1600}%
\special{pa 1360 1800}%
\special{dt 0.027}%
\special{pa 1440 1600}%
\special{pa 1240 1800}%
\special{dt 0.027}%
\special{pa 1320 1600}%
\special{pa 1200 1720}%
\special{dt 0.027}%
\special{pa 1680 1600}%
\special{pa 1480 1800}%
\special{dt 0.027}%
\special{pa 1800 1600}%
\special{pa 1600 1800}%
\special{dt 0.027}%
\special{pa 1920 1600}%
\special{pa 1720 1800}%
\special{dt 0.027}%
\special{pa 1990 1650}%
\special{pa 1840 1800}%
\special{dt 0.027}%
\special{pa 1990 1770}%
\special{pa 1960 1800}%
\special{dt 0.027}%
%
\special{pn 4}%
\special{pa 1560 1800}%
\special{pa 1360 2000}%
\special{dt 0.027}%
\special{pa 1440 1800}%
\special{pa 1240 2000}%
\special{dt 0.027}%
\special{pa 1320 1800}%
\special{pa 1200 1920}%
\special{dt 0.027}%
\special{pa 1680 1800}%
\special{pa 1480 2000}%
\special{dt 0.027}%
\special{pa 1800 1800}%
\special{pa 1600 2000}%
\special{dt 0.027}%
\special{pa 1920 1800}%
\special{pa 1720 2000}%
\special{dt 0.027}%
\special{pa 1990 1850}%
\special{pa 1840 2000}%
\special{dt 0.027}%
\special{pa 1990 1970}%
\special{pa 1960 2000}%
\special{dt 0.027}%
%
\special{pn 4}%
\special{pa 1560 2000}%
\special{pa 1360 2200}%
\special{dt 0.027}%
\special{pa 1440 2000}%
\special{pa 1240 2200}%
\special{dt 0.027}%
\special{pa 1320 2000}%
\special{pa 1200 2120}%
\special{dt 0.027}%
\special{pa 1680 2000}%
\special{pa 1480 2200}%
\special{dt 0.027}%
\special{pa 1800 2000}%
\special{pa 1600 2200}%
\special{dt 0.027}%
\special{pa 1920 2000}%
\special{pa 1720 2200}%
\special{dt 0.027}%
\special{pa 1990 2050}%
\special{pa 1840 2200}%
\special{dt 0.027}%
\special{pa 1990 2170}%
\special{pa 1960 2200}%
\special{dt 0.027}%
%
\special{pn 4}%
\special{pa 1560 2200}%
\special{pa 1360 2400}%
\special{dt 0.027}%
\special{pa 1440 2200}%
\special{pa 1240 2400}%
\special{dt 0.027}%
\special{pa 1320 2200}%
\special{pa 1200 2320}%
\special{dt 0.027}%
\special{pa 1680 2200}%
\special{pa 1480 2400}%
\special{dt 0.027}%
\special{pa 1800 2200}%
\special{pa 1600 2400}%
\special{dt 0.027}%
\special{pa 1920 2200}%
\special{pa 1720 2400}%
\special{dt 0.027}%
\special{pa 1990 2250}%
\special{pa 1840 2400}%
\special{dt 0.027}%
\special{pa 1990 2370}%
\special{pa 1960 2400}%
\special{dt 0.027}%
\put(23.6000,-26.0000){\makebox(0,0){$M$}}%
\put(23.6000,-6.0000){\makebox(0,0){Figure 22}}%
\end{picture}%
\end{center}
\bigskip

\noindent
By Lemma \ref{X7},
$$\#H + \#M\uppeq 3 < \# U_{\leq d+2} \upp 3.$$
Hence by the assumption of (subcase 2-2)
$$\#B \geq \# U_{e+a_2} \upp 2 - \#M \uppeq 3  =\# U_{\leq e+a_2} \upp 3 - \#M \uppeq 3 > \#H + \# \biguplus_{j=d+3}^{e+a_2} U_j\upp 3.$$
Let
$$B=I \biguplus J \biguplus G$$
such that $I$ is the set of lex-largest $\# H$ monomials in $B$
and $G$ is the rev-lex set with $\rho(G) = \biguplus_{j=d+3}^{e+a_2} U_j \upp 3$.
(See Fig.\ 23.)

\begin{center}
\unitlength 0.1in
\begin{picture}( 18.8000, 14.0000)(  7.2000,-15.1500)
%
\special{pn 13}%
\special{pa 1000 400}%
\special{pa 1000 1400}%
\special{fp}%
\special{pa 1000 1400}%
\special{pa 1800 1400}%
\special{fp}%
\special{pa 1800 1400}%
\special{pa 1800 400}%
\special{fp}%
\special{pa 1800 1400}%
\special{pa 2600 1400}%
\special{fp}%
\special{pa 2600 1400}%
\special{pa 2600 400}%
\special{fp}%
%
\special{pn 8}%
\special{pa 2600 400}%
\special{pa 1000 400}%
\special{dt 0.045}%
\special{pa 1000 600}%
\special{pa 2600 600}%
\special{dt 0.045}%
\special{pa 2600 800}%
\special{pa 1000 800}%
\special{dt 0.045}%
\special{pa 1000 1000}%
\special{pa 2600 1000}%
\special{dt 0.045}%
\special{pa 2600 1200}%
\special{pa 1000 1200}%
\special{dt 0.045}%
%
\special{pn 8}%
\special{pa 1800 400}%
\special{pa 1400 400}%
\special{fp}%
%
\special{pn 8}%
\special{pa 1300 400}%
\special{pa 1400 400}%
\special{fp}%
%
\special{pn 8}%
\special{pa 1800 600}%
\special{pa 1300 600}%
\special{fp}%
\special{pa 1300 600}%
\special{pa 1300 400}%
\special{fp}%
%
\special{pn 8}%
\special{pa 1430 600}%
\special{pa 1430 400}%
\special{fp}%
%
\special{pn 8}%
\special{pa 1560 400}%
\special{pa 1560 600}%
\special{fp}%
\put(13.6000,-5.0000){\makebox(0,0){$I$}}%
\put(15.0000,-5.0000){\makebox(0,0){$J$}}%
\put(17.0000,-5.0000){\makebox(0,0){$G$}}%
%
\special{pn 8}%
\special{pa 1800 1000}%
\special{pa 2200 1000}%
\special{fp}%
\special{pa 2200 1000}%
\special{pa 2200 800}%
\special{fp}%
\special{pa 2200 800}%
\special{pa 2600 800}%
\special{fp}%
\put(20.0000,-8.8000){\makebox(0,0){$H$}}%
%
\special{pn 8}%
\special{pa 2200 800}%
\special{pa 1800 800}%
\special{fp}%
%
\special{pn 4}%
\special{pa 2280 1200}%
\special{pa 2090 1390}%
\special{dt 0.027}%
\special{pa 2160 1200}%
\special{pa 1970 1390}%
\special{dt 0.027}%
\special{pa 2040 1200}%
\special{pa 1850 1390}%
\special{dt 0.027}%
\special{pa 1920 1200}%
\special{pa 1800 1320}%
\special{dt 0.027}%
\special{pa 2400 1200}%
\special{pa 2210 1390}%
\special{dt 0.027}%
\special{pa 2520 1200}%
\special{pa 2330 1390}%
\special{dt 0.027}%
\special{pa 2590 1250}%
\special{pa 2450 1390}%
\special{dt 0.027}%
%
\special{pn 4}%
\special{pa 2360 1000}%
\special{pa 2160 1200}%
\special{dt 0.027}%
\special{pa 2240 1000}%
\special{pa 2040 1200}%
\special{dt 0.027}%
\special{pa 2120 1000}%
\special{pa 1920 1200}%
\special{dt 0.027}%
\special{pa 2000 1000}%
\special{pa 1810 1190}%
\special{dt 0.027}%
\special{pa 1880 1000}%
\special{pa 1800 1080}%
\special{dt 0.027}%
\special{pa 2480 1000}%
\special{pa 2280 1200}%
\special{dt 0.027}%
\special{pa 2590 1010}%
\special{pa 2400 1200}%
\special{dt 0.027}%
\special{pa 2590 1130}%
\special{pa 2520 1200}%
\special{dt 0.027}%
%
\special{pn 4}%
\special{pa 2560 800}%
\special{pa 2360 1000}%
\special{dt 0.027}%
\special{pa 2440 800}%
\special{pa 2240 1000}%
\special{dt 0.027}%
\special{pa 2320 800}%
\special{pa 2200 920}%
\special{dt 0.027}%
\special{pa 2590 890}%
\special{pa 2480 1000}%
\special{dt 0.027}%
%
\special{pn 4}%
\special{pa 1760 400}%
\special{pa 1570 590}%
\special{dt 0.027}%
\special{pa 1790 490}%
\special{pa 1680 600}%
\special{dt 0.027}%
\special{pa 1640 400}%
\special{pa 1560 480}%
\special{dt 0.027}%
%
\special{pn 4}%
\special{pa 1560 480}%
\special{pa 1440 600}%
\special{dt 0.027}%
\special{pa 1520 400}%
\special{pa 1430 490}%
\special{dt 0.027}%
\put(18.0000,-16.0000){\makebox(0,0){$B \biguplus M \uppeq 3$}}%
%
\special{pn 4}%
\special{pa 1430 490}%
\special{pa 1320 600}%
\special{dt 0.027}%
\special{pa 1400 400}%
\special{pa 1300 500}%
\special{dt 0.027}%
\put(18.0000,-2.0000){\makebox(0,0){Figure 23}}%
\end{picture}%
\end{center}
\bigskip

\noindent
Since $\rho(B) \subset \biguplus_{j=d+2} ^{e+a_2} U_j \upp 2$,
$\rho(I) \subset U_{d+2}\upp 3$.
Let $C \subset U_{d+2} \upp 3$ be the lex set in $U_{d+2} \upp 3$ with $\#C=\# \rho(I)=\#H$.
If we regard $U \uppeq 3$ as an universal lex ideal in $K[x_3,\dots,x_n]$,
then $H$ and $C$ are lex sets in $K[x_3,\dots,x_n]$ with the same cardinality.
Hence $C=x_3H$.
Then, by the interval lemma,
\begin{eqnarray}
\label{QQ1}
m(H)=m(C) \succeq m\big(\rho(I) \big) = m(I)
\end{eqnarray}

Let $P \subset U \upp 2$ be the super rev-lex set with $\#P = \#A + \#J + \#M \upp2$.
By the choice of $G$,
$G$ is the set of all monomials $\delta_2 u \in \rho(F)$ such that
$u$ is not divisible by $x_2^{e+a_2 -(d+2)}$.
Also, since $B$ does not contain any monomial $\delta_2 u$ such that $u$ is divisible by $x_2^{e+a_2-(d+1)}$,
any monomial in $J$ is divisible by $\delta_2 x_2^{e+a_2-(d+2)}$.
Then
$x^{-(e+a_2)+d+2} J \subset U_{d+2} \upp 2$ is a rev-lex set.
Since $M\upp 2 \biguplus E \biguplus (x_2^{-(e+a_2)+(d+2)} J)$
is rev-lex,
\begin{eqnarray}
\label{e7}
m(P) \succeq m\big(M\upp 2 \biguplus E \biguplus x_2^{-(e+a_2)+(d+1)} J \big)= m \big(M \upp 2 \biguplus A \biguplus J \big).
\end{eqnarray}
(See Fig.\ 24.)

\begin{center}
\input{Fig24}
\end{center}
\bigskip
\bigskip

Let
$$Q=\rho(F) \setminus (A \biguplus B)=\rho(F)_{\leq e+a_2-1}.$$

\textbf{(subcase 2-2-a)}
Suppose that $\# P + \#Q  \leq \# U_{\leq e+a_2-1} \upp 2$.
Let $R \subset U \upp 2$ be the super rev-lex set with $\#R= \#P + \#Q$. 
Then since $Q$ is an upper rev-lex set of degree $e+a_2-1$,
by Corollary \ref{X6} and (\ref{e7})
\begin{eqnarray}
\label{e8}
m(R) \succeq m \big(P \biguplus Q \big)
\succeq \big(M\upp 2 \biguplus A \biguplus J \biguplus Q \big)
\end{eqnarray}

On the other hand,
by Lemma \ref{X7},
$$\#H + \# M \uppeq 3 < \# U_{\leq d+2} \upp 3.$$
Then since $\rho(G) = \biguplus_{j=d+3} ^{e+a_2} U_j \upp 3$,
\begin{eqnarray*}
\# I \biguplus G \biguplus M \uppeq 3=
\# G \biguplus H \biguplus M \uppeq 3
< \# U_{\leq e+a_2} \upp 3 = \# U \upp 2 _{e+a_2}.
\end{eqnarray*}
Let $U'=U \upp 2 \biguplus(\biguplus _{i=3}^t U \upp i [-a_2])$.
Observe that $M \upp 3[-a_2] \supset {U'} \uppeq 3_{\leq e+a_2}$.
Then Lemma \ref{X10} and (\ref{QQ1}) say that there exists an extremal set
$Z \subset U\uppeq 3 [-a_2]$ such that
\begin{eqnarray}
\label{e9}
Z \gg  G \biguplus H \biguplus M \uppeq 3 \gg G \biguplus I \biguplus M \uppeq 3
\end{eqnarray}
(See Fig.\ 25.)

\begin{center}
\input{Fig25}
\end{center}
\bigskip

We claim that
$$N= U_{\leq e-1} \upp 1 \biguplus R \biguplus Z$$
satisfies the desired conditions.
Indeed, by (\ref{e8}) and (\ref{e9}),
\begin{eqnarray*}
N 
&\gg& U_{\leq e-1} \upp 1 \biguplus \big(M \upp 2 \biguplus A \biguplus J \biguplus Q \big) \biguplus G \biguplus I \biguplus M \uppeq 3\\
&\gg& U_{\leq e-1} \upp 1 \biguplus F \biguplus M \upp 2 \biguplus M \uppeq 3\\
&=& M.
\end{eqnarray*}
(We use $\rho(F)=A \biguplus I \biguplus J \biguplus G \biguplus Q$
and $m(F)=m(\rho(F))$ for the second step.)
It remains to prove that $N$ is a ladder set.
Since $ U_{\leq d} \upp 2 \subset R \subset U_{\leq e+a_2-1} \upp 2$ it is enough to prove that
$Z \supset U_{\leq e+a_2-1} \uppeq 3$.
Since $\rho(G) = \biguplus_{j=d+3}^{e+a_2} U_j \upp 3$,
$$\# Z  = \# \big(H \biguplus M \uppeq 3 \biguplus G\big)
\geq \# U_{\leq d+1} \uppeq 3 \biguplus \big( \biguplus_{j=d+3}^{e+a_2} U_j \upp 3 \big) \geq \# U_{\leq e+a_2-1} \uppeq 3.$$
(We use $\# U_j \upp 3 \geq \# U_{j-1} \uppeq 3$ for the last step.)
Then $Z \supset U_{\leq e+a_2-1} \uppeq 3$ by Lemma \ref{X4} as desired.
\medskip

\textbf{(subcase 2-2-b)}
Suppose that $\# P + \#Q  > \# U_{\leq e+a_2-1} \upp 2$.
Note that
$$\#P + \#Q + \#I + \#G = \#F + \# M \upp 2.$$
Then $\# M\upp 2 \biguplus F > \# U \upp 2 _{\leq e+a_2-1}$.
Let $R$ be the super rev-lex set with $\#R= \# M \upp 2 + \#F$.
Then $\# R = \#M \upp 2 + \# F \leq \#U_{\leq e+a_2} \upp 2$ by (\ref{bound}).
Since $\#R \geq \#P+ \#Q > U_{\leq e+a_2-1} \upp 2$,
there exists a rev-lex set $B' \subset U_{e+a_2} \upp 2$
such that
$$R= U_{\leq e+a_2-1} \upp 2 \biguplus B'.$$
Also by Corollary \ref{X6},
\begin{eqnarray}
\label{20.5}
 B' \biguplus U \upp 2_{ \leq e+a_2-1}=R \gg M \upp 2 \biguplus \rho(F).
\end{eqnarray}
Since
$\# F + \# M \uppeq 2 < \# U\upp 2 _{\leq e+a_2}$,
we have $\#B' + \#M \uppeq 3 < \# U_{e+a_2} \upp 2$.
Then by Lemma \ref{X10} there exists the extremal set $Z \subset U \uppeq 3[-a_2]$ such that
\begin{eqnarray}
\label{20}
B' \biguplus M \uppeq 3 [-a_2]\ll Z.
\end{eqnarray}
We claim that
$$N=U_{\leq e-1} \upp 1 \biguplus U_{\leq e+a_2-1} \upp 2 \biguplus Z[+a_2]$$
satisfies the desired conditions.

By (\ref{20.5}) and (\ref{20}),
\begin{eqnarray*}
N &\gg & U_{\leq e-1} \upp 1 \biguplus U_{\leq e+a_2-1} \upp 2 \biguplus B' \biguplus M \uppeq 3\\
&\gg & U_{\leq e-1} \upp 1 \biguplus F \biguplus M \upp 2 \biguplus M \uppeq 3 =M.
\end{eqnarray*}
(See Fig.\ 26.)

\begin{center}
\input{Fig26}
\end{center}
\medskip
\bigskip

It remains to prove that $N$ is a ladder set.
What we must prove is
$$Z[+a_2] \supset  U_{\leq e+a_2-1} \uppeq 3.$$
By the assumption of (subcase 2-2-b),
$$\#M \upp 2 + \#F -\#\big(I \biguplus G\big) = \#Q + \#P  > \# U_{\leq e+a_2-1} \upp 2.$$
Then
$$\#B' = \#M \upp 2 + \#F - \#U_{\leq e+a_2-1} \upp 2 > \#I \biguplus G.$$
Then in the same way as the computation of $\#Z$ in (subcase 2-2-a),
we have
$$\# Z = \#M \uppeq 3 \biguplus B' \geq \#M \uppeq 3 \biguplus (I \biguplus G) \geq \# U_{\leq e+a_2-1} \uppeq 3.$$
Then by Lemma \ref{X4}, $Z[+a_2] \supset U_{\leq e+a_2-1} \uppeq 3$
as desired.

\section{Examples}

In this section, we give some examples of saturated graded ideals
which attain maximal Betti numbers for a fixed Hilbert polynomial.
Observe that, by the decomposition given before Definition \ref{revlex},
 the Hilbert polynomial of a proper universal lex ideal $I=(\delta_1,\delta_2,\dots,\delta_t)$
is given by
$$H_I(t)= {t-b_1+n-1 \choose n-1} + {t-b_2 + n-2 \choose n-2} + \cdots + {t-b_t + n-t \choose n-t},$$
where $b_i = \deg \delta_i$ for $i=1,2,\dots,t$.

\begin{example}
Let $S=K[x_1,\dots,x_4]$
and $\bar S =K[x_1,\dots,x_3]$.
Consider the ideal
$I=(x_1^3,x_1^2x_2,x_1x_2^2,x_2^3,x_1^2x_3) \subset S$.
Then
$$H_I(t)=\frac 1 6 t^3 + t^2 - \frac {15} 6 t +1={ t+2 \choose 3} + {t-4 \choose 2} + {t-9 \choose 1}$$
and the proper universal lex ideal with the same Hilbert polynomial as $I$ is
$$L=(x_1,x_2^6,x_2^5 x_3^5).$$
Let
$$U=\sat \bar L = \big( \bar L: x_3^\infty \big )= (x_1,x_2^5) \subset \bar S$$
and $c= \dim _K  U /\bar L =5$.
Then the extremal set $M \subset U$ with $\#M = 5$ is
$$M=x_1\{1,x_1,x_2,x_3\} \biguplus x_2^5\{1\}.$$
Then the ideal in $S$ generated by all monomials in $U \setminus M$ is
$$J=x_1(x_1^2,x_1x_2,x_1x_3,x_2^2,x_2x_3,x_3^2)+x_2^5(x_2,x_3) \subset S,$$
and $J$ has the largest total Betti numbers among all saturated graded ideals in $S$ having the same Hilbert polynomial as $I$.
\end{example}

\begin{example}
Let $S=K[x_1,\dots,x_5]$ and $\bar S=K[x_1,\dots,x_4]$.
Consider the ideal $I=(x_1,x_2^2,x_2x_3^3,x_2x_3^2x_4^{15})$.
Then $I$ is a proper universal lex ideal. Let
$$U=\sat \bar I = \big(\bar I :x_4^\infty \big) =(x_1,x_2^2,x_2x_3^2) \subset \bar S$$
and $c= \dim U/\bar I=15$.
Then the extremal set $M \subset U$ with $\#M = 15$ is
$$M=x_1\{1,x_1,x_2,x_3,x_4,x_2x_3,x_2x_4,x_3^2,x_3x_4,x_4^2\} \uplus x_2^2\{1,x_2,x_3,x_4\} \uplus x_2x_3^2\{1\}.$$
Then the ideal in $S$ generated by all monomials in $U \setminus M$ is
\begin{eqnarray*}
J &=&x_1(x_1^2,x_1x_2,x_1x_3,x_1x_4,x_2^2,x_2x_3^2,x_2x_3x_4,x_2x_4^2,x_3^3,x_3^2x_4,x_3x_4^2,x_4^3)\\
&&+x_2^2(x_2^2,x_2x_3,x_2x_4,x_3^2,x_3x_4,x_4^2) + x_2x_3^2(x_3,x_4)
\end{eqnarray*}
and $J$ has the largest total Betti numbers among all saturated graded ideals in $S$ having the same Hilbert polynomial as $I$.
\end{example}


\end{document}